\documentclass{amsart}
\usepackage{appendix}
\usepackage{amssymb}

\usepackage{thmtools}
\usepackage{thm-restate}
\usepackage{hyperref}
\usepackage{cleveref}

\usepackage{mathrsfs}

\usepackage{graphicx}

\usepackage[cmtip,all]{xy}

\usepackage{amsthm}
\newtheorem{theorem}{Theorem}[section]
\newtheorem{proposition}[theorem]{Proposition}

\newtheorem{lem}[theorem]{Lemma}
\newtheorem{lemma}[theorem]{Lemma}
\newtheorem{cor}[theorem]{Corollary}
\newtheorem{prop}[theorem]{Proposition}
\newtheorem*{conjecture*}{Conjecture}

\theoremstyle{definition}
\newtheorem{defn}[theorem]{Definition} 
\newtheorem{construction}[theorem]{Construction}
\newtheorem{example}[theorem]{Example}

\newtheorem{remark}[theorem]{Remark}
\newtheorem*{question*}{Question}

\newcommand{\lf}{\mathrm{lf}}
\newcommand{\Aut}{\mathrm{Aut}}

\newcommand{\End}{\mathrm{End}}
\newcommand{\PSL}{\mathrm{PSL}}
\newcommand{\SL}{\mathrm{SL}}
\newcommand{\GL}{\mathrm{GL}}
\newcommand{\PGL}{\mathrm{PGL}}

\newcommand{\Hom}{\mathrm{Hom}}
\newcommand{\im}{\mathop{\mathrm{im}}}

\newcommand{\al}{\alpha}

\newcommand{\CQ}{\mathcal Q}
\newcommand{\CalC}{\mathcal C}

\newcommand{\SA}{\mathcal A}
\newcommand{\NN}{\mathbb N}
\newcommand{\RR}{\mathbb R}
\newcommand{\ZZ}{\mathbb Z}
\newcommand{\CC}{\mathbb C}

\newcommand{\Qbb}{{\mathbb Q}}
\newcommand{\Rbb}{{\mathbb R}}
\newcommand{\Zbb}{{\mathbb Z}}
\newcommand{\Nbb}{{\mathbb N}}
\newcommand{\Cbb}{{\mathbb C}}
\newcommand{\Fbb}{{\mathbb F}}
\newcommand{\Ebb}{{\mathbb E}}

\makeatletter
\def\l@subsection{\@tocline{2}{0pt}{2.5pc}{5pc}{}}
\makeatother

\usepackage[style=alphabetic,
]{biblatex}

\usepackage{booktabs}  
\usepackage{adjustbox}
\usepackage{tikz}
\usepackage{quiver}
\usepackage{hyperref}
\usepackage{mathtools}
\usepackage{enumitem}
\usepackage{soul}

\newcommand{\nocontentsline}[3]{}
\newcommand\stoptoc{%
   \let\origcontentsline\addcontentsline
   \let\addcontentsline\nocontentsline
}
\newcommand\resumetoc{%
   \let\addcontentsline\origcontentsline
}

\numberwithin{equation}{section}

\begin{document}

\title{$K$-Theoretic Obstructions to Linearizing QCA Representations}


\author{Mattie Ji}
\address{University of Pennsylvania}
\curraddr{}
\email{mji13@sas.upenn.edu}
\thanks{}

\author{Bowen Yang}
\address{Harvard University}
\curraddr{}
\email{bowen\_yang@g.harvard.edu}
\thanks{}


\date{}

\dedicatory{}

\maketitle

\begin{abstract}
Projective representations arise naturally in physics and representation theory, and determining whether they can be linearized has been a fundamental problem. In this work, we study the analogous problem for quantum cellular automata (QCA) representations, which incorporate locality constraints imposed by a metric space $X$. Over an arbitrary field $\mathbb{F}$, we develop an obstruction theory for the linearization of QCA representations, using the algebraic $K$-theory spectrum of QCA constructed in previous work of the authors. The resulting obstructions are governed by the homotopy type of the QCA spaces, from which we extract universal obstruction classes to linearization. In the complex algebraic and unitary case, we also fully compute the homotopy types of the QCA spaces over a point, a line, and a plane.
\end{abstract}

\renewcommand\contentsname{Table of Contents}
\tableofcontents

\section{Introduction}

\subsection{Quantum Physics and Linearization} In his note~\cite{freed2023anomaly} \textit{``What is an anomaly?"}, Freed writes:
\begin{quote}
\centering
``Quantum theory is projective. Quantization is linear.''
\end{quote}

We explain his first point in the familiar setting of single-particle quantum mechanics. Let
\(V\) be a finite dimensional complex Hilbert space. There are two parallel
ways of describing quantum states, called the Schrödinger and
Heisenberg pictures.\\

\noindent \textbf{Schrödinger Picture.} In the Schrödinger picture, a pure state is a ray in \(V\), i.e. a
one-dimensional subspace
\[
    [v]\in \mathbb P V
\]
spanned by a nonzero vector \(v\in V\). Equivalently, the pure state is
represented by the rank-one orthogonal projection \(P_{[v]}\) onto \([v]\).

 Mixed states are probabilistic mixtures of pure states, and are represented by
 density operators
 \[
     \rho=\sum_i p_i P_{[v_i]}\in \End(V),
     \qquad
    p_i\geq 0,
    \qquad
   \sum_i p_i=1.
\]
 The pure states are precisely those density operators of rank one.\\

\noindent \textbf{Heisenberg Picture.} In the Heisenberg picture, one instead starts with the algebra of observables
\[
    \mathcal A \coloneqq \End(V).
\]
The Hermitian inner product on \(V\) gives \(\mathcal A\) a \(*\)-operation.
A state, pure or mixed, is a normalized positive linear functional
\[
    \omega:\mathcal A\to \mathbb C,
    \qquad
    \omega(I)=1,
    \qquad
    \omega(A^*A)\geq 0
    \quad\text{for all }A\in\mathcal A.
\]
 The set of states is convex, and its extreme points are the pure states.\\


The two descriptions are equivalent for finite dimensional $V$. From either picture, the action of a group $G$ on quantum states (i.e., a $G$-symmetry) is naturally projective. 

In the Schrödinger picture, symmetries are encoded through a projective unitary representation
\[
    \varphi:G\longrightarrow PU(V)\subset PGL(V).
\]

In the Heisenberg picture, the same symmetry acts on the algebra of
observables by \(*\)-automorphisms:
\[
    \varphi:G\longrightarrow \Aut_{*}(\mathcal A)=PU(V).
\]
The induced action on a state \(\omega\) is then given by composition.

To generalize this to an arbitrary field $\mathbb F$, we forget the Hermitian structure. This leads us to a projective representations:
\[
    \varphi:G\longrightarrow \Aut(\mathcal A)=PGL(V).
\]

Freed's second point above leads naturally to the question of linearization.
\begin{question*}[Linearization]
   Does there exist a lift $\Tilde{\varphi}$ to the following diagram?
\[
\begin{tikzcd}
& GL(V) \arrow[d] \\
G \arrow[r, "\varphi"'] \arrow[ur, dashed, "\widetilde{\varphi}"] 
& PGL(V).
\end{tikzcd}
\]
\end{question*}

This question has a well-known answer in terms of an obstruction class in group cohomology.

\begin{theorem}\label{thm:classical}
Let \(\varphi:G\to PGL(V)\) be a projective representation over a field \(\Fbb\).
Choose lifts \(U_g\in GL(V)\) of \(\varphi(g)\). Then there are scalars
\(\alpha(g,h)\in \Fbb^\times\) such that
\[
    U_gU_h=\alpha(g,h)U_{gh}.
\]
The function \(\alpha\) is a \(2\)-cocycle, and its cohomology class
\[
    [\alpha]\in H^2(G,\Fbb^\times)
\]
is independent of the choice of lifts. The projective representation
\(\varphi\) is linearizable if and only if $[\alpha]=0$.
\end{theorem}
So far, we have focused on a \textit{single-particle} system, in the sense that $V$ may be regarded as concentrated at a single point in space. We now pass to the many-body setting in the Heisenberg picture, where one assigns endomorphism algebras to points of a lattice $X$ and considers  tensor product of them. The resulting objects are quantum spin systems~\cite{Nachtergaele2004QuantumSpinSystems, naaijkens2017quantum}, and their natural automorphisms are quantum cellular automata (QCA)~\cite{farrelly2020review}.

In this paper, a quantum spin system is specified by a function $q: X \to \Zbb_{>0}$, where $q(x)$ is the dimension of the local $\mathbb{F}$-vector space at $x$, and its QCA group $\mathcal{Q}(X,q)$ consists of the locality-preserving automorphisms of the associated algebra of observables; see Section~\ref{sec::background}. Thus a $G$-symmetry of $(X,q)$ is modeled by a homomorphism $\varphi: G \to \mathcal{Q}(X,q)$. In this setting, linearization has two components: factorization into pointwise projective representations, followed by pointwise lifting to genuine linear representations.

\begin{question*}[Linearization for QCA]
Does there exist $\beta \in \mathcal{Q}(X, q)$ such that the following diagram can be completed?
\begin{equation}\label{eq::qca_linearization}
\begin{tikzcd}[ampersand replacement=\&]
	\&\& {\prod_{x \in X} \operatorname{GL}(\mathbb{F}^{q(x)})} \\
	\&\& {\prod_{x \in X} \operatorname{PGL}(\mathbb{F}^{q(x)})} \\
	G \&\& {\mathcal{Q}(X, q)}
	\arrow[from=1-3, to=2-3]
	\arrow[from=2-3, to=3-3]
	\arrow["{(2)}"{description}, dashed, from=3-1, to=1-3]
	\arrow["{(1)}"{description}, dashed, from=3-1, to=2-3]
	\arrow["{\beta \varphi \beta^{-1}}"', from=3-1, to=3-3]
\end{tikzcd}
\end{equation}
\end{question*}

A QCA representation $\varphi: G \to \mathcal{Q}(X, q)$ fitting in (\ref{eq::qca_linearization}) is \textit{linearizable}\footnote{In the literature, the property of being \textit{linearizable} is often called “onsiteable,” indicating that the action can be decomposed into separate actions on individual sites. We find this term ambiguous, however, as it could also refer to lifting only to (1), which is a distinct notion. The property of being \textit{stably linearizable} is often called ``stably onsiteable".}. One can also impose a relaxed condition that $\varphi$ is \textit{stably linearizable} if there is some linearizable $\psi: G \to \mathcal{Q}(X, r)$ such that the map $\varphi \otimes \psi: G \to \mathcal{Q}(X, qr)$, induced by pointwise multiplying $q$ and $r$, is linearizable. One can also specify that $\varphi$ is \textit{weakly linearizable} if it is liftable in (\ref{eq::qca_linearization}) after taking the colimit over all quantum spin systems (see Section~\ref{sec::background} for details).\\

To continue the single-particle story in Theorem~\ref{thm:classical}, a natural question to study is to find obstruction classes to (stable/weak) linearizations of QCA representations.

\subsection{Stable Homotopy Theory and Obstruction Classes.}

In this work, we develop an obstruction theory for (stable/weak) linearizations of QCA representations over any field $\Fbb$ and general metric spaces $X$ with reasonable assumptions. We also adapt them suitably to the unitary case. We first set up the precise definitions of QCA and their representations in Section~\ref{sec::background}. We then study linearizations from the perspective of \textit{stable homotopy theory} and \textit{cohomological obstruction theory}.\\

\noindent \textbf{Stable Homotopy Theory.} In the recent work of \cite{qca_grp_space_spectrum}, the authors developed an algebraic $K$-theory spectrum of quantum cellular automata over a fixed commutative ring $R$. For a reasonable space $X$, they constructed a space $K(\mathbf{C}(X; R))$ using Segal's $K$-theory of symmetric monoidal categories \cite{Segal1974} and defined $\mathbf{Q}(X; R)$, the \textit{QCA space of $X$}, as the based loop space of $K(\mathbf{C}(X; R))$ at the identity component. They then showed the spaces $\{\mathbf{Q}(X \times \Zbb^n; R)\}_{n \geq 0}$ assembles into an $\Omega$-spectrum $\mathbb{QCA}(X; R)$. When $X = *$, we write $\mathbb{QCA}(R) = \mathbb{QCA}(*; R)$. For the unitary case, the authors also constructed the unitary QCA spectrum using a similar a space $K(\mathbf{C}^*(X))$ and its based loop space $\mathbf{Q}^*(X)$.

For an $\Omega$-spectrum $X = \{X_n\}_{n \geq 0}$, the based homotopy classes of maps $X^n(-) \coloneqq [-; X_n]$ forms a (reduced) generalized cohomology theory, and the correspondence goes backward by Brown's representability \cite{brown}. Our observation is that the QCA representation $\varphi: G \to \mathcal{Q}(X, q)$ induces a natural map $\varphi^{st}: BG \to K(\mathbf{C}(X; R))_1$ and hence defines a \textit{cohomology class} in
\[[\varphi^{st}] \in [BG, K(\mathbf{C}(X; R))_1] =\footnote{This step holds because the proof of the delooping result in \cite{qca_grp_space_spectrum} showed that $\Omega K(\mathbf{C}(X \times \Zbb; R))_1$ is a space whose component at identity is $K(\mathbf{C}(X; R))_1$.} [BG, \mathbf{Q}(X \times \Zbb; R)] = \mathbb{QCA}(X; R)^1(BG).\]

When $X$ is for example $\Zbb^n$, $[\varphi^{st}]$ is identified as a cohomology class $\mathbb{QCA}(R)^{n+1}(BG)$. We can then use the machinery of (stable) homotopy theory to study the class $[\varphi^{st}]$ and infer about (stable/weak) linearizations.

We explore this connection in Section~\ref{sec::homotopy}. Our first theorem gives a homotopical obstruction to linearization.
\begin{restatable}{introthm}{ratNull}\label{thm::fin_gp_null}
    Let $G$ be a finite or, more generally, rationally acyclic group. Suppose a QCA representation $\varphi: G \to \mathcal{Q}(X, q)$ over any field $\Fbb$  is stably linearizable or weakly linearizable, then the stabilization $\varphi^{st}: BG \to K(\mathbf{C}(X; \Fbb))_1$ is null-homotopic.
\end{restatable}

Write $\operatorname{PGL}(X; \Fbb)$ to denote the colimit of $\prod_{x \in X} \operatorname{PGL}_{q(x)}(\Fbb)$ over all quantum spin systems $q$. Our second theorem shows that projectivizable null-homotopies are linearizable. When $X = *$, this gives a converse to Theorem~\ref{thm::fin_gp_null} for finite groups and is a stable analog of Theorem~\ref{thm:classical}.
\begin{restatable}{introthm}{PGLlift}
\label{thm::converse_on_PGL}
    Let $G$ be a group. Suppose a QCA representation $\varphi: G \to \mathcal{Q}(X, q)$ lands in $\prod_{x \in X} \operatorname{PGL}_{q(x)}(\Fbb)$, and the induced map $\varphi': BG \to \operatorname{BPGL}(X; \Fbb) \to \operatorname{BPGL}(X; \Fbb)^+$ is null-homotopic, then $\varphi$ is weakly linearizable. If $G$ is furthermore finite, then $\varphi$ is stably linearizable.
\end{restatable}

We also develop unitary analogs of Theorem~\ref{thm::fin_gp_null} and~\ref{thm::converse_on_PGL} in Section~\ref{sec::unitary_obstruction}.\\

\noindent \textbf{Obstruction Classes.} For a map $f: BG \to K(\mathbf{C}(X; \Fbb))_1$, there are successive cohomological obstruction classes
\[u_i(f) \in H^{i}(BG; \pi_{i-1}(\mathbf{Q}(X; \Fbb))), i \geq 1\]
that are obstructions to $f$ being null-homotopic.

In Section~\ref{sec::cohomology}, we explain how to sequentially extract these cohomology classes using a universal construction called Dror's tower~\cite{Dror1972}. We specialize to the case of algebraic and unitary QCA spaces to define what we call \textit{universal obstruction classes} in Definition~\ref{def::universalObstruction} (after the Universal Coefficient Theorem used in the construction).
\begin{restatable}{introthm}{anomalies}
\label{thm::anomalies}
Suppose $G$ is finite or rationally acyclic and $\varphi$ is a (unitary) QCA representation. If $\varphi$ is stably or weakly linearizable, then each universal obstruction classes $u_i(\varphi)$ vanishes. 
\end{restatable}

Using this result and computations of the algebraic $K$-theory spectrum of unitary QCA, we obtain an obstruction theory in the unitary case, as anticipated in~\cite{Tu_2026}.
\begin{restatable}{introcor}{ObstructionUnitary}
    When $G$ is finite (or more generally, rationally acyclic) and $X=\ZZ^n$, the only potentially nonzero universal obstruction classes occur in
\begin{equation}
H^{i+1}\bigl(BG;\mathcal Q^*(\ZZ^{n-i})/\mathcal C^*(\ZZ^{n-i})\bigr),
\end{equation}
for $0\leq i\leq n+1$,
and in
\begin{equation}
H^{n+2}(BG; U(1)).
\end{equation}
\end{restatable}

\noindent \textbf{Computations.} Motivated by the stable homotopy theory interpretation of linearization and invertible field theory, we also compute the groups $\mathbb{QCA}(X; R)^1(BG)$ for various cases. In Section~\ref{sec::computation}, we discuss how to compute the structure of QCA spaces for special examples, which in turn provide structural information of their generalized cohomologies on $BG$.
 
\begin{restatable}{introthm}{CSplit}\label{thm::complex_qca_split}
The following spaces are equivalent to products of Eilenberg-MacLane spaces:
\begin{enumerate}[leftmargin=*]
    \item (Complex Case): For $n = 0,1$, the space $K(\mathbf{C}(\Zbb^n; \mathbb{C}))$ and hence $\mathbf{Q}(\Zbb^n; \mathbb{C})$.
    \item (Complex Case): The universal cover of $K(\mathbf{C}(\Zbb^2; \mathbb{C}))$ and hence $\mathbf{Q}(\Zbb^2; \mathbb{C})$.
    \item (Unitary Case): For $n = 0, 1, 2$, the space $K(\mathbf{C}^*(\Zbb^n))$ and hence $\mathbf{Q}^*(\Zbb^n)$.
    \item (Finite Field Case): For $n = 0, 1$, the spaces $\mathbf{Q}(\Zbb^n; \Fbb_{p^k})$ for all primes $p$, and the spaces $K(\mathbf{C}(\Zbb^n; \Fbb_{2^k}))$.
\end{enumerate}
\end{restatable}
For all spaces in Theorem~\ref{thm::complex_qca_split} other than $\mathbf{Q}(\Zbb^2; \Cbb)$, this is in fact an equivalence of group-like $\Ebb_{\infty}$-spaces. In the unitary case, a result of \cite{freedman2020classification} implies $K(\mathbf{C}^*(\Zbb^2))$ is simply connected. If the analogous $K(\mathbf{C}(\Zbb^2; \mathbb{C}))$ is simply connected, the equivalence in for $\mathbf{Q}(\Zbb^2; \Cbb)$ is an equivalence of group-like $\Ebb_{\infty}$ spaces.\\

In Appendix~\ref{appendix::obstruction}, we compare our universal obstruction classes with several low-degree obstruction classes from the literature. The degree-$1$ obstruction agrees with ours, the degree-$2$ construction gives a concrete model for ours, and the degree-$3$ obstruction factors through ours. This comparison shows the universal obstruction classes refine the existing low-degree obstructions. We also clarify the ambiguity around hyperplane choice in degree $2$ over any field. Appendix~\ref{appendix::arithmetic} gives a detailed account of a family of QCA representations whose universal obstruction classes have arithmetic origins.\\

Although this paper applies the QCA space and spectrum constructed in our recent work~\cite{qca_grp_space_spectrum}, the study of finite-group representations by QCAs in the unitary setting predates that construction. In the physics literature, these questions have been studied under the name of lattice anomalies, with cohomological obstruction classes proposed and computed in many examples, particularly in dimensions one and two~\cite{kawagoe2025anomaly, Tu_2026, czajka2025anomalies, shirley2026anomaly, seifnashri2026disentangling, bols2026classification, kapustinXu2025higher}. In particular, Kapustin and Xu~\cite{kapustinXu2025higher} use crossed modules and crossed squares to encode higher symmetries in one and two dimensions, respectively, and interpret anomalies as obstructions to promoting ordinary group actions to morphisms of higher groups. Kapustin and Spodyneiko~\cite{kapustin2025higher} develop a related description for lattice gauge theories by restricting symmetry transformations to spatial regions. Although one of the authors drew considerable inspiration from these two works, our approach proceeds largely independently through $K$-theory. Our contribution is to develop an obstruction theory within the uniform framework provided by the QCA spectrum~\cite{qca_grp_space_spectrum}. This framework allows us to prove rigorous results over arbitrary fields and for a broad class of spaces, including the lattices $\ZZ^d$ in every dimension. It also allows us to identify the corresponding universal obstruction classes in all cohomological degrees. We do not address anomalies of higher symmetries; see~\cite{kapustin2025higher, feng2026higher, kapustinXuInPreparation} for related treatments.\\

\noindent \textbf{Acknowledgements.} We thank Dan Freed for encouraging us to compute a homotopy fiber, which led to Section~\ref{sec::cohomology}. We thank Mark Behrens and Mike Hopkins for insights that greatly clarified the computation in Section~\ref{sec::computation}.

MJ would like to thank Mark Behrens and Mona Merling for many helpful conversations throughout this project. MJ also thanks the following people for helpful discussions: A. J. (Jon) Berrick on \cite{Berrick1982, BERRICK1983172} and Remark~\ref{rmk::pullback}, Nir Gadish and David Zhaoqi Zhu on obstruction theory, Pengkun Huang and Juan Moreno on integral Steenrod operations, and Daniel Krashen on group cohomology.

BY thanks Anton Kapustin introducing him to anomalous lattice symmetries at the first place. BY also thanks Yu-An Chen, Jeongwan Haah, David Long, and Wilbur Shirley for patiently answering his questions.

MJ is partially supported by the National Science Foundation Graduate Research Fellowship (DGE-2236662). BY acknowledges support from the Simons Foundation through the Simons Collaboration on Global Categorical Symmetries.

\section{QCA Representation of Discrete Groups}\label{sec::background}

\subsection{QCA and Circuits}\label{subsec::qca} To model infinitely many particles in physical space, we imagine them as being
distributed ``uniformly'' over an unbounded metric space \(X\). To impose mild
regularity assumptions on \(X\), we always assume that \(X\) is contains a countable, coarsely dense subspace that is uniformly discrete, locally finite metric space of bounded geometry. Note that this does not require \(X\) itself to be discrete or locally finite. Typical examples include lattices in Lie
groups, such as
\[
    \mathbb{Z}^d \subset \mathbb{R}^d.
\]
For two such metric spaces $X$ and $X'$, the space $X \times X'$ has the $L^{\infty}$-metric.

\begin{defn}
A function \(q\colon X \to \mathbb{N}\) is called locally finite if
\[
    \Lambda_q := \{x\in X : q_x \coloneqq q(x)  > 1\}
\]
is a locally finite subset of \(X\). That is, \(\Lambda_q\) has finite
intersection with every bounded subset of \(X\).\footnote{In particular,
\(\Lambda_q\) is finite if \(X\) is bounded.} We often call \(\Lambda_q\) the
lattice associated to \(q\). We denote the collection of such functions by
\(\mathbb{N}^X_{\mathrm{lf}}\). This is a partially ordered set under the
relation
\[
    q \leq r
    \quad\Longleftrightarrow\quad
    q_x \mid r_x \text{ for all } x\in X.
\]
\end{defn}

If \(X\) is already locally finite, then every function \(q\colon X\to\mathbb{N}\)
is locally finite. 
\begin{defn}\label{def::algebra_observables}
    A \textit{$\mathbb F$-quantum spin system} consists of a metric space $(X, \rho)$, and a locally finite function $q\in \NN^X_\lf$. Its \textit{algebra of local observables} is defined as
    \begin{equation}
  \SA(X,q) = \bigotimes_{x\in X} M_{q_x}(\mathbb F)= \varinjlim_{B\subset X \text{ bounded }} \bigotimes_{x\in B} M_{q_x}(\mathbb F). 
\end{equation}
An element of $a\in \SA(X,q)$ can be understood as a finite $k$-linear combination of restricted tensor products
\[
\bigotimes_{x\in X} a_x,
\]
where $a_x\in M_{q_x}(\mathbb F)$ and $a_x=I_{q_x}$ outside some bounded subset of $X$. The minimal bounded subset $B\subset X$ such that the element lies in $\bigotimes_{x\in B} M_{q_x}(\mathbb F)$ is called its support denoted by $\mathrm{supp}(a)$.
\end{defn}

\begin{defn}\label{def:locality_preserve_iso}
A \emph{quantum cellular automaton (QCA)} of $\SA(X,q)$ is a $\mathbb F$-algebra automorphism $\al$ of finite spread. That is, there exists $\ell>0$ such that for every $a\in \SA(X,q)$
the support of $\al(a)$ is contained in
\[
D_\ell(\mathrm{supp} (a)):=\{y\in X:\rho(\mathrm{supp}(a),y)\leq \ell\}.
\] The smallest such $\ell$, when it exists, is called the spread of $\al$. The \textit{group of QCA} of $\SA(X,q)$ is denoted as $\CQ(X,q)$.
\end{defn}

For quantum spin systems $q \leq r$, there is a well-defined map
\[\CQ(X, q) \hookrightarrow \CQ(X, r) \]
induced by pointwise Kronecker product with the identity matrices (see Lemma 2.13 of \cite{qca_grp_space_spectrum}).

\begin{defn}\label{def::qca_group}
The \emph{total QCA group} over $X$ is the direct limit in the category of groups:
\[
\CQ(X)\ :=\ \operatorname{colim}_{q\in \NN^X_\mathrm{lf}}\ \CQ(X,q).
\]
Concretely, $\CQ(X)$ is the set of equivalence classes $[q,\alpha]$ with
$\alpha\in\CQ(X,q)$, modulo the condition that
\[
[q,\alpha]\sim [r,\beta] \text{ if and only if there exists } s\ \text{with}\ q\mid s,\ r\mid s\ \text{such that}\
\iota_{q\to s}(\alpha)=\iota_{r\to s}(\beta).
\]
The product is
\[
[q,\alpha]\cdot[r,\beta]\;\coloneqq\;
\big[s,\ \iota_{q\to s}(\alpha)\circ \iota_{r\to s}(\beta)\big],
\qquad s_x \coloneqq \mathrm{lcm}(q_x,r_x).
\]
\end{defn}

    When $\mathbb F=\CC$, there is a unitary variant of QCA that commutes with the adjoint operation on elements of $\mathcal A(X,q)$. In Section 5 of \cite{qca_grp_space_spectrum}, they are referred to as $*$-QCA, to distinguish them from the algebraic version. We denote the unstable $*$-QCA group as $\mathcal{Q}^*(X, q)$, and the total $*$-QCA group as $\mathcal{Q}^*(X)$. When $X$ is a point, the adjoint preserving QCA are precisely the projective unitary matrices. Thus, we refer to this as the \textit{unitary case} and use the word \textit{unitary QCA} interchangeably with $*$-QCA. All of our results and proofs admit straightforward adaptations to this setting. Throughout, we will indicate the corresponding unitary versions where appropriate.

We now discuss an important class of QCA known as \textit{circuits}.
\begin{example}[Single-layer general, special, and unitary circuit]
Let $q\in \NN^X_{\mathrm{lf}}$ be a spin system. A \emph{single-layer circuit} on $(X,q)$ consists of the following data.

\begin{enumerate}[leftmargin=*]
    \item First, choose a uniformly bounded partition
\begin{equation}
X=\coprod_j X_j,
\qquad
\sup_j \operatorname{diam}(X_j)<\infty .
\end{equation}
For each block $X_j$, set
\begin{equation}
q(X_j):=\prod_{x\in X_j} q_x .
\end{equation}

\item Second, for each block $X_j$, choose a local automorphism $\alpha_j$ of the matrix algebra associated to $X_j$. We will consider either of the following cases:
\begin{enumerate}
    \item A \textit{single layer general circuit} chooses
\[
\alpha_j\in \operatorname{PGL}_{q(X_j)}(\mathbb F),
\]
corresponding to conjugation by an invertible matrix.
\item A \textit{single layer special circuit} chooses
\[
\alpha_j\in \operatorname{PSL}_{q(X_j)}(\mathbb F),
\]
corresponding to conjugation by an invertible matrix of determinant one.
\item A \textit{single layer unitary circuit} chooses
\[
\alpha_j\in \operatorname{PU}_{q(X_j)},
\]
corresponding to conjugation by a unitary matrix.
\end{enumerate}
\end{enumerate}

Each circuit can be identified as a QCA as follows. Since the subsets $X_j$ are pairwise disjoint, the automorphisms $\alpha_j$ act on mutually commuting tensor factors of $\SA(X,q)$. Hence their tensor product
\begin{equation}
\alpha:=\bigotimes_j \alpha_j
\end{equation}
defines a locality-preserving automorphism of $\SA(X,q)$. The uniform bound on the diameters of the blocks $X_j$ implies that the propagation of $\alpha$ is uniformly bounded.
\end{example}

\begin{defn}
We define $\mathcal{C}(X, q), \mathcal{C}^{\mathrm{sp}}(X, q), \mathcal{C}^*(X, q)$ as the subgroup generated by single-layer general (resp. special, unitary) circuits in their corresponding QCA groups over $(X, q)$ respectively. We write $\mathcal{C}(X), \mathcal{C}^{sp}(X), \mathcal{C}^*(X)$ to denote the subgroups generated by all single layer circuits in $\mathcal{Q}(X)$. 
\end{defn}

\begin{prop}\label{prop::circuit_type}
The groups $\CalC(X)$ and $\CalC^\mathrm{sp}(X)$ are normal subgroups of $\CQ(X)$. Furthermore, $\CalC^\mathrm{sp}(X)$ is equal to the commutator subgroup of $\CQ(X)$. Similarly, $\CalC^*(X)$ is the commutator subgroup of $\CQ^*(X)$. When $q(x) \geq 3$ for all $x \in X$, the unstable counterpart of the statements for $(X, q)$ also holds. Furthermore, $\mathcal{C}(X) = \mathcal{C}^{sp}(X)$ when $X = Y \times \Zbb$ or $\Fbb$ has $n$-th roots for all $n$.
\end{prop}

\begin{proof}
    See Proposition~2.19, Lemma~2.31, Corollary C, and Proposition 5.6 in~\cite{qca_grp_space_spectrum}.
\end{proof}

Theorem A of \cite{qca_grp_space_spectrum} shows that $\mathcal{Q}(\Zbb)/\mathcal{C}(\Zbb)$ is $K_0(\operatorname{Az}(\Fbb))$, the $K_0$ group of the symmetric monoidal category of $\Fbb$-Azumaya algebras with tensor product. 

\begin{remark}
    In the unitary case, $\mathcal{Q}^*(\Zbb)/\mathcal{C}^*(\Zbb)$ is shown to be $\mathbb Q_{>0}$ through the so-called \textit{GNVW index}~\cite{Gross_2012}. In the algebraic setting, as $K_0(\mathrm{Az}(\Fbb))\cong \mathbb Q_{>0}\oplus \mathrm{Br}(\mathbb F)$, we still call the first component the GNVW index. We refer to the second component the \textit{Brauer index} for the obvious reason. 
\end{remark}

\subsection{QCA Representations} For our discussions on linearization, we are interested in a special class of single-layer circuits known as \textit{on-site QCA}.
\begin{example}[On-site QCA]
For each $x\in X$, choose an invertible matrix $u_x\in GL_{q_x}(\mathbb F)$. Define
\[
\al\left(\bigotimes_{x\in X} a_x\right)
=
\bigotimes_{x\in X} u_x a_x u_x^{-1}.
\]
This is a QCA of spread $0$.
\end{example}

Now we will define QCA representations and their linearization questions. We first do this over a field $\Fbb$ and then discuss the unitary case later.
\begin{defn}[QCA representations]
Let $\CQ(X,q)$ denote the group of QCA of $\SA(X,q)$. A \emph{QCA representation} of a group $G$ on the quantum system parametrized by $(X,q)$ is a group homomorphism
\begin{equation}
\varphi:G\longrightarrow \CQ(X,q).
\end{equation}

Given QCA representations
\begin{equation}
\varphi:G\longrightarrow \CQ(X,q),
\qquad
\psi:G\longrightarrow \CQ(X,r),
\end{equation}
their tensor product is the QCA representation
\begin{equation}
\varphi\otimes \psi:G\longrightarrow \CQ(X,qr),
\end{equation}
defined by
\begin{equation}
(\varphi\otimes \psi)(g):=\varphi(g)\otimes \psi(g).
\end{equation}
Here $qr\in \NN^X_{\mathrm{lf}}$ denotes the pointwise product spin system, so that $(qr)_x=q_xr_x$. The same construction applies to unitary QCA representations, giving
\begin{equation}
\varphi\otimes \psi:G\longrightarrow \CQ^*(X,qr)
\end{equation}
whenever $\varphi$ and $\psi$ are unitary.
\end{defn}

The on-site example above defines a group homomorphism
\[
\prod_{x\in X} GL_{q_x}(\mathbb F)\longrightarrow \prod_{x\in X} PGL_{q_x}(\mathbb F)\hookrightarrow \CQ(X,q),
\]
where an element $(u_x)_{x\in X}$ acts by
\[
\bigotimes_{x\in X} a_x
\longmapsto
\bigotimes_{x\in X} u_x a_x u_x^{-1}.
\]

The algebra $\SA(X,q)$ should be regarded as the infinite-particle analogue of $\SA=\End(V)$. Under this analogy, $PGL(V)$ is replaced by $\CQ(X,q)$, and projective representations are replaced by QCA representations. Thus the question of linearization generalizes as follows.

\begin{question*}[Linearization of QCA Representation]
Given a QCA representation $\varphi: G \to \CQ(X, q)$. Does there exist a QCA $\beta\in \CQ(X,q)$ such that the following diagram can be completed?

\begin{equation}\label{eq::unstable_lift}
\begin{tikzcd}
	& {\prod_{x\in X} \operatorname{GL}_{q_x}(\mathbb{F})} \\
	& {\prod_{x\in X} \operatorname{PGL}_{q_x}(\mathbb{F})} \\
	G & {\CQ(X,q)}
	\arrow[two heads, from=1-2, to=2-2]
	\arrow[hook, from=2-2, to=3-2]
	\arrow[shift left=2, dashed, from=3-1, to=1-2]
	\arrow[dashed, from=3-1, to=2-2]
	\arrow["{\beta \varphi \beta^{-1}}"', from=3-1, to=3-2]
\end{tikzcd}
\end{equation}
\end{question*}

The question has two parts.
\begin{enumerate}[label=(\alph*), leftmargin=*]
    \item First, after conjugating by a QCA $\beta$, does the QCA representation land in
$$
\prod_{x\in X} \operatorname{PGL}_{q_x}(\mathbb F)\subset \CQ(X,q)?$$
Equivalently, is it induced by a collection of independent projective representations? 
\item Second, do these projective representations lift to genuine linear representations, so that the conjugated action factors through
\[
\prod_{x\in X} GL_{q_x}(\mathbb F)?
\]
\end{enumerate}

\begin{defn}\label{def::factorizability}
We say that the QCA representation $\varphi$ is:
\begin{enumerate}
    \item \emph{linearizable} if $\varphi$ fits into diagram (\ref{eq::unstable_lift}).
    \item \emph{linearly onsite} if $\varphi$ fits into diagram (\ref{eq::unstable_lift}) with $\beta$ being the identity map.
    \item \emph{stably linearizable} if $\varphi \otimes \psi: G \to \CQ(X, qr)$ fits into diagram (\ref{eq::unstable_lift}) for some linearly onsite representation $\psi: G \to \CQ(X, r)$.
    \item \emph{trivially stably linearizable} if $\varphi \otimes c_{id}: G \to \CQ(X, qr)$ fits into diagram $(\ref{eq::unstable_lift}$) and $c_{id}: G \to \CQ(X, r)$ is the constant map to the identity element of $\CQ(X, r)$.
\end{enumerate}
\end{defn}
Note that from Definition~\ref{def::factorizability}, we clearly have the sequences of implications
\[(2) \implies (1) \implies (4) \implies (3).\]
The four conditions all consider lifts after propagating a finite number of stages along the colimit in $\Nbb^{X}_{\lf}$. We can also ask the question of whether or not the QCA representation $\varphi$ is linearizable after passing through the entire colimit. For convenience, we introduce the following notations:
\begin{defn}\label{def::gl_notation}
We write:
\begin{itemize}
    \item $\operatorname{GL}_{\otimes}(X; \Fbb)$ to be $\operatorname{colim}_{q \in \Nbb^{X}_{\lf}} \prod_{x \in X} \operatorname{GL}_{q_x}(\Fbb)$.
    \item $\operatorname{PGL}(X; \Fbb)$ to be $\operatorname{colim}_{q \in \Nbb^{X}_{\lf}} \prod_{x \in X} \operatorname{PGL}_{q_x}(\Fbb)$.
    \item When $X = *$, we write $\operatorname{GL}_{\otimes}(\Fbb) = \operatorname{GL}_{\otimes}(*; \Fbb)$ and $\operatorname{PGL}_\otimes(\Fbb) = \operatorname{PGL}(*; \Fbb)$.
\end{itemize}
\end{defn}

\begin{defn}\label{def:weak_lift}
Continuing from Definition~\ref{def::factorizability}, we say that the QCA representation $\varphi$ is:
\begin{enumerate}
\setcounter{enumi}{4}
    \item \textit{weakly linearizable} if the induced map $G \to \mathcal{Q}(X)$ factors through $\operatorname{GL}_{\otimes}(X; \Fbb) \to \operatorname{PGL}(X; \Fbb)$ up to conjugation.
    \item \textit{weakly linearly onsite} if $G \to \mathcal{Q}(X)$ factors through $\operatorname{GL}_{\otimes}(X; \Fbb) \to \operatorname{PGL}(X; \Fbb)$.
\end{enumerate}
\end{defn}
From Definition~\ref{def::factorizability} and Definition~\ref{def:weak_lift}, we have the following implications:
\[(4) \implies (5) \impliedby (6) \impliedby (2).\]

Each condition in the two definitions is a generalization of linearizability of a projective representation. In particular we have

\begin{prop}
When $X = *$, we have $\CQ(X,q) \cong \PGL_q(\mathbb{F})$, and a QCA representation becomes a projective representation. Each condition in Definition~\ref{def::factorizability} and~\ref{def:weak_lift} is equivalent to linearizability of the projective representation.
\end{prop}

\begin{proof}
Let $\varphi:G\to \PGL_q(\mathbb F)$ be a projective representation, and let $\alpha_\varphi$ denote its obstruction cocycle, as in Theorem~\ref{thm:classical}. Conjugating $\varphi$ by an element of $\PGL_q(\mathbb F)$ clearly does not change $\alpha_\varphi$.

The obstruction is also additive under tensor product. Indeed, if $\psi:G\to \PGL_r(\mathbb F)$ is another projective representation, then choosing lifts of $\varphi$ and $\psi$ gives
\begin{equation}
\alpha_{\varphi\otimes\psi}=\alpha_\varphi+\alpha_\psi,
\end{equation}
where the coefficient group $\mathbb F^\times$ is written additively. In particular, tensoring with a linearizable projective representation does not change the obstruction class.

Thus conditions $(1)$--$(4)$ are all equivalent to the vanishing of the obstruction index. By Theorem~\ref{thm:classical}, this is equivalent to linearizability.

The same obstruction may be defined for a homomorphism $G\to \PGL_\otimes(\mathbb F)$, using the central extension
\begin{equation}
    1\longrightarrow \mathbb F^\times \longrightarrow \GL_\otimes(\mathbb F)\longrightarrow \PGL_\otimes(\mathbb F)\longrightarrow 1 .
\end{equation}

Its vanishing is equivalent to liftability to $\GL_\otimes(\mathbb F)$, by the same argument as in the finite-dimensional case. Moreover, for a projective representation $\varphi:G\to \PGL(V)$, the obstruction class is the same as the obstruction class of the induced map $G\to \PGL_\otimes(\mathbb F)$. Therefore the stabilized conditions $(5)$ and $(6)$ are again equivalent to the vanishing of the same index, and hence to linearizability.

\end{proof}
\begin{remark}
For a general $X$, these conditions are distinct. To see the effect of conjugation, suppose that $X$ consists of two points. The onsite automorphisms
\begin{equation}
\PGL(\mathbb F^m)\times \PGL(\mathbb F^n)
\subset \PGL(\mathbb F^{mn})
\cong \Aut(M_m(\mathbb F)\otimes M_n(\mathbb F))
= \CQ(X,(m,n))
\end{equation}
do not form a normal subgroup. Stabilization can also have a nontrivial effect for unbounded spaces, such as $X=\ZZ$: see Appendix~F of~\cite{bols2026classification} for an example of a stably linearizable representation that is not linearizable. Finally, Conditions~(5) and~(6) differ from the rest when the group is infinite.
\end{remark}

For the unitary case, we consider an analog of the linearization of QCA representations in (\ref{eq::unstable_lift}).

\begin{defn}
Let $\CQ^*(X,q)$ denote the group of $*$-QCA over $\Cbb$. we call $\varphi: G \to \mathcal{Q}(X, q)$ a \emph{unitary QCA representation} if it takes values in $\CQ^*(X,q)$.
\end{defn}

\begin{question*}[Unitary Linearization of QCA Representation]
Given an unitary QCA representation $\varphi: G \to \CQ^*(X, q)$. Does there exist an unitary QCA $\beta\in \CQ^*(X,q)$ such that the following diagram can be completed?
\begin{equation}\label{eq::unitary_unstable_lift}
\begin{tikzcd}
	& {\prod_{x\in X} \operatorname{U}(q_x)} \\
	& {\prod_{x\in X} \operatorname{PU}(q_x)} \\
	G & {\CQ^*(X,q)}
	\arrow[two heads, from=1-2, to=2-2]
	\arrow[hook, from=2-2, to=3-2]
	\arrow[shift left=2, dashed, from=3-1, to=1-2]
	\arrow[dashed, from=3-1, to=2-2]
	\arrow["{\beta \varphi \beta^{-1}}"', from=3-1, to=3-2]
\end{tikzcd}
\end{equation}
\end{question*}

We can then define the unitary analogs of Definition~\ref{def::factorizability} directly with respect to the diagram (\ref{eq::unitary_unstable_lift}). We will also use $\operatorname{U}_{\otimes}(X)$ and $\operatorname{PU}(X)$ to denote the unitary analogs of Definition~\ref{def::gl_notation}. This lets us define the unitary analogs of Definition~\ref{def:weak_lift} as well. Note here we do not impose any topology on the unitary and projective unitary groups and only consider their underlying discrete structures.
\subsection{Examples}
We end the section with some examples. First, we show that not every QCA representation is stably or weakly linearizable. It suffices to exhibit a QCA that is not stably a circuit. Indeed, we have inclusions
\begin{equation}
\PGL(X;\mathbb F)\subset \CalC(X)\subset \CQ(X),
\end{equation}
and $\CalC(X)$ is a normal subgroup of $\CQ(X)$. Hence any stably or weakly linearizable QCA representation must take values in $\CalC(X)$. Conversely, any element of $\CQ(X)$ which does not lie in $\CalC(X)$ determines a QCA representation of the integer group that is neither stably nor weakly linearizable.

\begin{example}[Algebraic pumps]
As shown in Theorem~A of~\cite{qca_grp_space_spectrum}, when $X=\ZZ$ there is an element of $\CQ(\ZZ^1)$ which does not lie in $\CalC(\ZZ^1)$, associated to a pair $(A,B)$ of central simple $\mathbb F$-algebras. Let $A'$ and $B'$ denote their respective opposite algebras. The corresponding QCA represents the element $\frac{A}{B} \in K_0(\operatorname{Az}(\Fbb)) = \mathcal{Q}(\Zbb^1)/\mathcal{C}(\Zbb^1)$ and is illustrated in Figure~\ref{fig::surjectivity}.
\end{example}

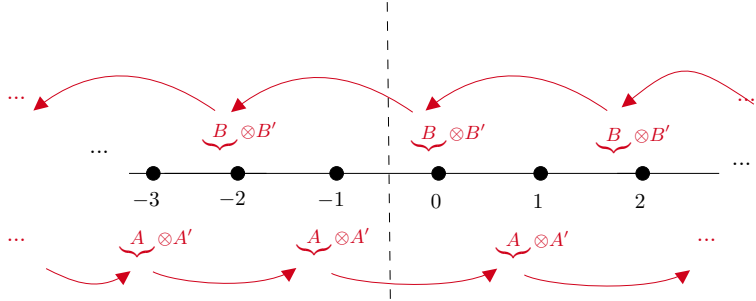
\begin{figure}[htb]
    \centering
\[\scalebox{0.8}{\begin{tikzpicture}[x=0.75pt,y=0.75pt,yscale=-1,xscale=1]

\draw  [fill={rgb, 255:red, 0; green, 0; blue, 0 }  ,fill opacity=1 ] (247,163.75) .. controls (247,161.4) and (248.9,159.5) .. (251.25,159.5) .. controls (253.6,159.5) and (255.5,161.4) .. (255.5,163.75) .. controls (255.5,166.1) and (253.6,168) .. (251.25,168) .. controls (248.9,168) and (247,166.1) .. (247,163.75) -- cycle ;
\draw  [fill={rgb, 255:red, 0; green, 0; blue, 0 }  ,fill opacity=1 ] (311,163.75) .. controls (311,161.4) and (312.9,159.5) .. (315.25,159.5) .. controls (317.6,159.5) and (319.5,161.4) .. (319.5,163.75) .. controls (319.5,166.1) and (317.6,168) .. (315.25,168) .. controls (312.9,168) and (311,166.1) .. (311,163.75) -- cycle ;
\draw    (251.25,163.75) -- (315.25,163.75) ;
\draw  [fill={rgb, 255:red, 0; green, 0; blue, 0 }  ,fill opacity=1 ] (375,163.75) .. controls (375,161.4) and (376.9,159.5) .. (379.25,159.5) .. controls (381.6,159.5) and (383.5,161.4) .. (383.5,163.75) .. controls (383.5,166.1) and (381.6,168) .. (379.25,168) .. controls (376.9,168) and (375,166.1) .. (375,163.75) -- cycle ;
\draw    (315.25,163.75) -- (379.25,163.75) ;
\draw  [fill={rgb, 255:red, 0; green, 0; blue, 0 }  ,fill opacity=1 ] (439,163.75) .. controls (439,161.4) and (440.9,159.5) .. (443.25,159.5) .. controls (445.6,159.5) and (447.5,161.4) .. (447.5,163.75) .. controls (447.5,166.1) and (445.6,168) .. (443.25,168) .. controls (440.9,168) and (439,166.1) .. (439,163.75) -- cycle ;
\draw    (379.25,163.75) -- (443.25,163.75) ;
\draw  [fill={rgb, 255:red, 0; green, 0; blue, 0 }  ,fill opacity=1 ] (185,163.75) .. controls (185,161.4) and (186.9,159.5) .. (189.25,159.5) .. controls (191.6,159.5) and (193.5,161.4) .. (193.5,163.75) .. controls (193.5,166.1) and (191.6,168) .. (189.25,168) .. controls (186.9,168) and (185,166.1) .. (185,163.75) -- cycle ;
\draw    (189.25,163.75) -- (253.25,163.75) ;
\draw    (143.25,163.75) -- (207.25,163.75) ;
\draw    (427.25,163.75) -- (491.25,163.75) ;
\draw  [dash pattern={on 4.5pt off 4.5pt}]  (283,56.4) -- (285,248) ;
\draw  [fill={rgb, 255:red, 0; green, 0; blue, 0 }  ,fill opacity=1 ] (132,163.75) .. controls (132,161.4) and (133.9,159.5) .. (136.25,159.5) .. controls (138.6,159.5) and (140.5,161.4) .. (140.5,163.75) .. controls (140.5,166.1) and (138.6,168) .. (136.25,168) .. controls (133.9,168) and (132,166.1) .. (132,163.75) -- cycle ;
\draw    (121,163.75) -- (185,163.75) ;
\draw [color={rgb, 255:red, 208; green, 2; blue, 27 }  ,draw opacity=1 ]   (136,225.4) .. controls (150.48,232.64) and (202.2,237.64) .. (223.77,224.5) ;
\draw [shift={(226,223)}, rotate = 143.13] [fill={rgb, 255:red, 208; green, 2; blue, 27 }  ,fill opacity=1 ][line width=0.08]  [draw opacity=0] (8.93,-4.29) -- (0,0) -- (8.93,4.29) -- cycle    ;
\draw [color={rgb, 255:red, 208; green, 2; blue, 27 }  ,draw opacity=1 ]   (248,226.4) .. controls (262.48,233.64) and (327.24,238.08) .. (349.72,224.9) ;
\draw [shift={(352,223.4)}, rotate = 143.13] [fill={rgb, 255:red, 208; green, 2; blue, 27 }  ,fill opacity=1 ][line width=0.08]  [draw opacity=0] (8.93,-4.29) -- (0,0) -- (8.93,4.29) -- cycle    ;
\draw [color={rgb, 255:red, 208; green, 2; blue, 27 }  ,draw opacity=1 ]   (369,227.4) .. controls (383.48,234.64) and (448.24,239.08) .. (470.72,225.9) ;
\draw [shift={(473,224.4)}, rotate = 143.13] [fill={rgb, 255:red, 208; green, 2; blue, 27 }  ,fill opacity=1 ][line width=0.08]  [draw opacity=0] (8.93,-4.29) -- (0,0) -- (8.93,4.29) -- cycle    ;
\draw [color={rgb, 255:red, 208; green, 2; blue, 27 }  ,draw opacity=1 ]   (421,124.4) .. controls (379.84,95.98) and (343.48,95.41) .. (309.1,122.69) ;
\draw [shift={(307,124.4)}, rotate = 320.36] [fill={rgb, 255:red, 208; green, 2; blue, 27 }  ,fill opacity=1 ][line width=0.08]  [draw opacity=0] (8.93,-4.29) -- (0,0) -- (8.93,4.29) -- cycle    ;
\draw [color={rgb, 255:red, 208; green, 2; blue, 27 }  ,draw opacity=1 ]   (299,125.4) .. controls (257.84,96.98) and (221.48,96.41) .. (187.1,123.69) ;
\draw [shift={(185,125.4)}, rotate = 320.36] [fill={rgb, 255:red, 208; green, 2; blue, 27 }  ,fill opacity=1 ][line width=0.08]  [draw opacity=0] (8.93,-4.29) -- (0,0) -- (8.93,4.29) -- cycle    ;
\draw [color={rgb, 255:red, 208; green, 2; blue, 27 }  ,draw opacity=1 ]   (175,124.4) .. controls (133.84,95.98) and (97.48,95.41) .. (63.1,122.69) ;
\draw [shift={(61,124.4)}, rotate = 320.36] [fill={rgb, 255:red, 208; green, 2; blue, 27 }  ,fill opacity=1 ][line width=0.08]  [draw opacity=0] (8.93,-4.29) -- (0,0) -- (8.93,4.29) -- cycle    ;
\draw [color={rgb, 255:red, 208; green, 2; blue, 27 }  ,draw opacity=1 ]   (513,120.4) .. controls (471.84,91.98) and (465.25,93.33) .. (432.07,120.69) ;
\draw [shift={(430,122.4)}, rotate = 320.36] [fill={rgb, 255:red, 208; green, 2; blue, 27 }  ,fill opacity=1 ][line width=0.08]  [draw opacity=0] (8.93,-4.29) -- (0,0) -- (8.93,4.29) -- cycle    ;
\draw [color={rgb, 255:red, 208; green, 2; blue, 27 }  ,draw opacity=1 ]   (69,223.4) .. controls (83.4,230.6) and (100.56,239.27) .. (119.61,226.69) ;
\draw [shift={(122,225)}, rotate = 143.13] [fill={rgb, 255:red, 208; green, 2; blue, 27 }  ,fill opacity=1 ][line width=0.08]  [draw opacity=0] (8.93,-4.29) -- (0,0) -- (8.93,4.29) -- cycle    ;

\draw (238,174.4) node [anchor=north west][inner sep=0.75pt]    {$-1$};
\draw (309,175.4) node [anchor=north west][inner sep=0.75pt]    {$0$};
\draw (373,175.4) node [anchor=north west][inner sep=0.75pt]    {$1$};
\draw (437,175.4) node [anchor=north west][inner sep=0.75pt]    {$2$};
\draw (176,173.4) node [anchor=north west][inner sep=0.75pt]    {$-2$};
\draw (95,148.4) node [anchor=north west][inner sep=0.75pt]  [font=\large]  {$...$};
\draw (498,155.4) node [anchor=north west][inner sep=0.75pt]  [font=\large]  {$...$};
\draw (297,131.4) node [anchor=north west][inner sep=0.75pt]  [font=\small,color={rgb, 255:red, 208; green, 2; blue, 27 }  ,opacity=1 ]  {$\underbrace{B} \otimes B'$};
\draw (122,174.4) node [anchor=north west][inner sep=0.75pt]    {$-3$};
\draw (114,197.4) node [anchor=north west][inner sep=0.75pt]  [font=\small,color={rgb, 255:red, 208; green, 2; blue, 27 }  ,opacity=1 ]  {$\underbrace{A} \otimes A'$};
\draw (225,196.4) node [anchor=north west][inner sep=0.75pt]  [font=\small,color={rgb, 255:red, 208; green, 2; blue, 27 }  ,opacity=1 ]  {$\underbrace{A} \otimes A'$};
\draw (350,198.4) node [anchor=north west][inner sep=0.75pt]  [font=\small,color={rgb, 255:red, 208; green, 2; blue, 27 }  ,opacity=1 ]  {$\underbrace{A} \otimes A'$};
\draw (413,132.4) node [anchor=north west][inner sep=0.75pt]  [font=\small,color={rgb, 255:red, 208; green, 2; blue, 27 }  ,opacity=1 ]  {$\underbrace{B} \otimes B'$};
\draw (166,130.4) node [anchor=north west][inner sep=0.75pt]  [font=\small,color={rgb, 255:red, 208; green, 2; blue, 27 }  ,opacity=1 ]  {$\underbrace{B} \otimes B'$};
\draw (43,203.4) node [anchor=north west][inner sep=0.75pt]  [font=\large,color={rgb, 255:red, 208; green, 2; blue, 27 }  ,opacity=1 ]  {$...$};
\draw (501,115.4) node [anchor=north west][inner sep=0.75pt]  [font=\large,color={rgb, 255:red, 208; green, 2; blue, 27 }  ,opacity=1 ]  {$...$};
\draw (43,114.4) node [anchor=north west][inner sep=0.75pt]  [font=\large,color={rgb, 255:red, 208; green, 2; blue, 27 }  ,opacity=1 ]  {$...$};
\draw (476,203.4) node [anchor=north west][inner sep=0.75pt]  [font=\large,color={rgb, 255:red, 208; green, 2; blue, 27 }  ,opacity=1 ]  {$...$};

\end{tikzpicture}}
\]
    \caption{A QCA $\alpha$ For a pair $(A, B)$ of central simple algebras. Here we place an alternating sequence of $A \otimes A' \cong M_n(\mathbb F)$ and $B \otimes B' \cong M_m(\mathbb F)$ on the metric space $\mathbb{Z}$, and $\alpha$ simultaneously transports the tensor factor $B$ to the left and the tensor factor $A$ to the right.}
    \label{fig::surjectivity}
\end{figure}
For general metric spaces, finding non-circuit QCA remains an active area of research. In the unitary case over $X=\ZZ^3$, a candidate was first constructed in~\cite{haah2023nontrivial}. Subsequently, further examples have been found~\cite{fidkowski2024qca, sun2025clifford, shirley2022three}. At present, these examples and the group $\CQ^*(X)/\CalC^*(X)$ are still not well understood; understanding this quotient is part of the problem of computing the homotopy groups of the space $\mathbf{Q}^*(X)$. However, a special class of QCA, namely Clifford QCA, has become well understood~\cite{haah2021clifford, haah2025topological, Yang_2026}.

Non-circuit QCA are not the only way a QCA representation fails to be stably or weakly linearizable. Indeed, they correspond only to the degree-$1$ cohomology obstruction class in Section~\ref{sec::cohomology}. The paper~\cite{shirley2026anomaly} contains an example of a QCA representation with nontrivial degree-$2$ obstruction class. We give some examples in Appendix~\ref{appendix::arithmetic}. In general, there exists examples with vanishing higher degree obstruction classes. 

Lastly, we mention a bewildering QCA which only occurs for QCA representation of infinite order. 
\begin{example}[Annular shifts on the plane, Figure~\ref{fig:annular}]\label{exp::annular}
Let $X=\ZZ^2$ and $q$ uniformly constant. Decompose $X$ into square annuli centered at the origin. For instance, write
\begin{equation}
A_r={(i,j)\in \ZZ^2:\max(|i|,|j|)=r},
\qquad r\geq 1.
\end{equation}
Each $A_r$ is a finite cycle, with adjacent sites ordered cyclically around the boundary of the square. Choose signs $\epsilon_r\in {\pm 1}$, and define an automorphism by shifting the matrix algebras on $A_r$ by one step in the direction $\epsilon_r$, independently for each $r$. The center site may be fixed. Since each site is moved only a uniformly bounded distance, the resulting automorphism is a QCA on $\ZZ^2$. One may therefore ask whether the $\ZZ$-action generated by this annular shift is stably or weakly linearizable. We will discuss it again in Example~\ref{exp::revisit}.
\end{example}
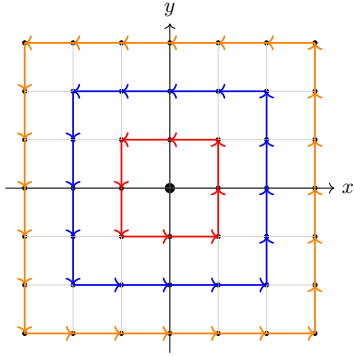
\begin{figure}
    \centering
\[\scalebox{0.8}{\begin{tikzpicture}[scale=0.8]

  \foreach \x in {-3,...,3} {
    \foreach \y in {-3,...,3} {
      \fill (\x,\y) circle (1.5pt);
    }
  }

\fill (0,0) circle (3pt);
  \draw[step=1, gray!35, very thin] (-3,-3) grid (3,3);

  \draw[->] (-3.4,0) -- (3.4,0) node[right] {$x$};
  \draw[->] (0,-3.4) -- (0,3.4) node[above] {$y$};

  \draw[red, thick, ->] (1,0) -- (1,1);
  \draw[red, thick, ->] (1,1) -- (0,1);
  \draw[red, thick, ->] (0,1) -- (-1,1);
  \draw[red, thick, ->] (-1,1) -- (-1,0);
  \draw[red, thick, ->] (-1,0) -- (-1,-1);
  \draw[red, thick, ->] (-1,-1) -- (0,-1);
  \draw[red, thick, ->] (0,-1) -- (1,-1);
  \draw[red, thick, ->] (1,-1) -- (1,0);

  \draw[blue, thick, ->] (2,-1) -- (2,0);
  \draw[blue, thick, ->] (2,0) -- (2,1);
  \draw[blue, thick, ->] (2,1) -- (2,2);
  \draw[blue, thick, ->] (2,2) -- (1,2);
  \draw[blue, thick, ->] (1,2) -- (0,2);
  \draw[blue, thick, ->] (0,2) -- (-1,2);
  \draw[blue, thick, ->] (-1,2) -- (-2,2);
  \draw[blue, thick, ->] (-2,2) -- (-2,1);
  \draw[blue, thick, ->] (-2,1) -- (-2,0);
  \draw[blue, thick, ->] (-2,0) -- (-2,-1);
  \draw[blue, thick, ->] (-2,-1) -- (-2,-2);
  \draw[blue, thick, ->] (-2,-2) -- (-1,-2);
  \draw[blue, thick, ->] (-1,-2) -- (0,-2);
  \draw[blue, thick, ->] (0,-2) -- (1,-2);
  \draw[blue, thick, ->] (1,-2) -- (2,-2);
  \draw[blue, thick, ->] (2,-2) -- (2,-1);

  \draw[orange, thick, ->] (3,-2) -- (3,-1);
  \draw[orange, thick, ->] (3,-1) -- (3,0);
  \draw[orange, thick, ->] (3,0) -- (3,1);
  \draw[orange, thick, ->] (3,1) -- (3,2);
  \draw[orange, thick, ->] (3,2) -- (3,3);
  \draw[orange, thick, ->] (3,3) -- (2,3);
  \draw[orange, thick, ->] (2,3) -- (1,3);
  \draw[orange, thick, ->] (1,3) -- (0,3);
  \draw[orange, thick, ->] (0,3) -- (-1,3);
  \draw[orange, thick, ->] (-1,3) -- (-2,3);
  \draw[orange, thick, ->] (-2,3) -- (-3,3);
  \draw[orange, thick, ->] (-3,3) -- (-3,2);
  \draw[orange, thick, ->] (-3,2) -- (-3,1);
  \draw[orange, thick, ->] (-3,1) -- (-3,0);
  \draw[orange, thick, ->] (-3,0) -- (-3,-1);
  \draw[orange, thick, ->] (-3,-1) -- (-3,-2);
  \draw[orange, thick, ->] (-3,-2) -- (-3,-3);
  \draw[orange, thick, ->] (-3,-3) -- (-2,-3);
  \draw[orange, thick, ->] (-2,-3) -- (-1,-3);
  \draw[orange, thick, ->] (-1,-3) -- (0,-3);
  \draw[orange, thick, ->] (0,-3) -- (1,-3);
  \draw[orange, thick, ->] (1,-3) -- (2,-3);
  \draw[orange, thick, ->] (2,-3) -- (3,-3);
  \draw[orange, thick, ->] (3,-3) -- (3,-2);
\end{tikzpicture}}
\]
    \caption{Annular shifts on the plane. This has spread $1$.}
    \label{fig:annular}
\end{figure}

\section{Homotopical Interpretation of Linearization}\label{sec::homotopy}

We first review the QCA spaces constructed in \cite{qca_grp_space_spectrum}. For each space $X$ and field $\Fbb$, the authors of \cite{qca_grp_space_spectrum} defined a certain symmetric monoidal category $\mathbf{C}(X; \Fbb)$ of quantum spin systems over $(X, \Fbb)$. 

\begin{defn}
The \textit{space of QCA} $\mathbf{Q}(X; \Fbb)$ over $(X, \Fbb)$ is defined to be
\[\mathbf{Q}(X; \Fbb) = \Omega K(\mathbf{C}(X; \Fbb))_1\]
where $K(\mathbf{C}(X; \Fbb))$ is the Segal's K-theory of a symmetric monoidal category (see IV.4 of \cite{weibel2013k} or appendix of \cite{Thomason01011982}) and the subscript $1$ denotes the component at identity. We often suppress the notation $\Fbb$.  
\end{defn}

Theorem E and the remarks after in \cite{qca_grp_space_spectrum} showed that $\mathbf{Q}(X)$ assembles into an $\Omega$-spectrum of the form $\{\mathbf{Q}(X \times \Zbb^n)\}_{n \geq 0}$. Theorem B of \cite{qca_grp_space_spectrum} constructs a natural map
\[B\mathcal{Q}(X) \to K(\mathbf{C}(X))_{1}\]
that induces a factorization
\[B\mathcal{Q}(X) \to B\mathcal{Q}(X)^+ \xrightarrow{\cong} K(\mathbf{C}(X))_1.\]
Here, $B\mathcal{Q}(X)^+$ refers to the plus-construction of the space $B\mathcal{Q}(X)$, which is defined axiomatically as follows.
\begin{defn}
Let $X$ be a connected CW complex, its plus-construction is a space $X^+$ and a map $i: X \to X^+$ such that (a) $\ker(i_*: \pi_1(X) \to \pi_1(X^+))$ is the maximal perfect normal subgroup $\mathcal{P}\pi_1(X)$ of $\pi_1(X)$ and (b) $i: X \to X^+$ is an isomorphism in homology on any local coefficients. The map $i$ is initial among any maps $f: X \to Y$ that sends $\mathcal{P}\pi_1(X)$ to $0$.
\end{defn}

Unless mentioned otherwise, the maximal perfect normal subgroups considered for plus-constructions in this paper would happen to be the commutator subgroup. This is certainly the case for $B\mathcal{Q}(X)$, due to Theorem B of \cite{qca_grp_space_spectrum}. Given this plus-construction space, we can stabilize a QCA representation by looking at its induced map into the plus-construction.

\begin{defn}
For any QCA representation $\varphi: G \to \mathcal{Q}(X, q)$, we define the \textit{stabilization} of $\varphi$ as the composition
\[\varphi^{st}: BG \xrightarrow{B\varphi} B\mathcal{Q}(X, q) \to B\mathcal{Q}(X) \to B\mathcal{Q}(X)^+ \to K(\mathbf{C}(X))_{1}.\]
\end{defn}

In what follows, we say a group $G$ is \textit{rationally acyclic} if $\Tilde{H}^k(BG; \Qbb) = 0$ for $k \geq 0$. Examples of rationally acyclic groups include finite groups and infinite torsion groups such as $\frac{\Qbb}{\Zbb}$.\\

The purpose of this section is to prove the following two theorems.

\ratNull*

\begin{remark}
We do not anticipate the converse to be true. An evidence comes from the case $G = \Zbb$ (which is not rationally acyclic). A QCA representation $\varphi: \Zbb \to \mathcal{Q}(X, q)$ is equivalent to specifying a generator $\alpha \in \mathcal{Q}(X, q)$. In this case, the stabilization $\varphi^{st}: B\Zbb \cong S^1 \to K(\mathbf{C}(X))_{1}$ is null-homotopic if and only if $\varphi^{st}$ represents $0$ in $\pi_1(K(\mathbf{C}(X))_{1}) = \mathcal{Q}(X)^{ab}$. In other words, $\varphi^{st}$ is null-homotopic if and only if $\alpha$ is stably a commutator. When $\Fbb = \mathbb{C}$ or $X = Y \times \Zbb$, Corollary C of \cite{qca_grp_space_spectrum} implies $\mathcal{Q}(X)^{ab} = \mathcal{Q}(X)/\mathcal{C}(X)$, so $\varphi^{st}$ is null-homotopic if and only if $\alpha$ is stably a circuit. Under this assumption, we see that $\varphi$ being stably linearizable would imply it is stably a circuit and hence $\varphi^{st}$ is null-homotopic. However, not every circuit is conjugate to an onsite QCA with stabilization.
\end{remark}

We note that a partial converse for Theorem~\ref{thm::fin_gp_null} is true over a point, which will follow from the following more general statement about linearizing projectivizable null-homotopies.
\PGLlift*

Since $\operatorname{PGL}(*; \Fbb) = \mathcal{Q}(*)$, we have the following corollary.

\begin{cor}\label{cor::converse_on_pt}
    Let $G$ be a finite group and $X = *$ over any field $\Fbb$, then a QCA representation $\varphi: G \to \mathcal{Q}(*, q)$ is stably linearizable if and only if the stabilization $\varphi^{st}: BG \to K(\mathbf{C}(X))_1$ is null-homotopic.
\end{cor}

\begin{remark}
    We observe that that $\CQ(*, q)\cong \PGL_q(\mathbb F)$, so $\varphi$ is a projective representation of $G$. The problem of stable linearization is therefore equivalent to that of unstable linearization. In this sense, Corollary~\ref{cor::converse_on_pt} may be viewed as a stabilized homotopical analogue of Theorem~\ref{thm:classical}.
\end{remark}

\subsection{Null-Homotopy of Stable Rationally Acyclic QCA Representations}

In this section, we will prove Theorem~\ref{thm::fin_gp_null} over a fixed field $\Fbb$. The proof relies on the following lemmas whose proof we defer to the end of this section.

\begin{lem}\label{lem::commutator_perfect}
The commutator subgroups of $\operatorname{BGL}_{\otimes}(X; \Fbb)$ and $\operatorname{BPGL}(X; \Fbb)$ are both perfect and hence their maximal perfect normal subgroups.
\end{lem}

\begin{lem}\label{lem::BGL_rational}
The homotopy groups of $\operatorname{BGL}_{\otimes}(X; \Fbb)^+$ are rational. 
\end{lem}

Let us first see how the theorem follows from Lemma~\ref{lem::commutator_perfect} and Lemma~\ref{lem::BGL_rational} in the following steps.

\begin{lem}\label{lem::factorized_null_homotopuic}
Let $G$ be a rationally acyclic group. Suppose a QCA representation $\varphi: G \to \mathcal{Q}(X, q)$ is linearly onsite, then the stabilization $\varphi^{st}$ is null-homotopic.  
\end{lem}

\begin{proof}
Since $\varphi$ is linearly onsite, it fits in (\ref{eq::unstable_lift}) with $\beta = \operatorname{id}$. We can extend the diagram to consider the maps:
\[\begin{tikzcd}
	& {\prod_{x\in X} \operatorname{GL}_{q_x}(\mathbb{F})} & {\operatorname{GL}_{\otimes}(X; \mathbb{F})} \\
	& {\prod_{x\in X} \operatorname{PGL}_{q_x}(\mathbb{F})} & {\operatorname{PGL}(X; \mathbb{F})} \\
	G & {\CQ(X,q)} & {\CQ(X; \mathbb{F})}
	\arrow[from=1-2, to=1-3]
	\arrow[two heads, from=1-2, to=2-2]
	\arrow[from=1-3, to=2-3]
	\arrow[from=2-2, to=2-3]
	\arrow[hook, from=2-2, to=3-2]
	\arrow[from=2-3, to=3-3]
	\arrow[shift left=2, dashed, from=3-1, to=1-2]
	\arrow[dashed, from=3-1, to=2-2]
	\arrow["\varphi"', from=3-1, to=3-2]
	\arrow[from=3-2, to=3-3]
\end{tikzcd}.\]
The natural map $\operatorname{BGL}_{\otimes}(X; \Fbb) \to \operatorname{BPGL}(X; \Fbb) \to \operatorname{BPGL}(X;\Fbb)^{+}$ sends the commutator subgroup to $0$ by Lemma~\ref{lem::commutator_perfect}. The map $\operatorname{BPGL}(X; \Fbb) \to B\mathcal{Q}(X; \Fbb) \to  B\mathcal{Q}(X; \Fbb)^+$ also sends the commutator subgroup of $\operatorname{PGL}(X; \Fbb)$ to $0$. Thus, the universal property of plus-construction gives a commutative diagram:
\[\begin{tikzcd}
	BG && \\
	{\operatorname{BGL}_{\otimes}(X; \mathbb{F})} & {\operatorname{BPGL}(X; \mathbb{F})} & {B\mathcal{Q}(X; \mathbb{F})} \\
	{\operatorname{BGL}_{\otimes}(X; \mathbb{F})^+} & {\operatorname{BPGL}(X; \mathbb{F})^+} & {B\mathcal{Q}(X; \mathbb{F})^+}
	\arrow[from=1-1, to=2-1]
	\arrow[from=1-1, to=2-2]
	\arrow[from=1-1, to=2-3]
	\arrow[from=2-1, to=2-2]
	\arrow[from=2-1, to=3-1]
	\arrow[from=2-2, to=2-3]
	\arrow[from=2-2, to=3-2]
	\arrow[from=2-3, to=3-3]
	\arrow[from=3-1, to=3-2]
	\arrow[from=3-2, to=3-3]
\end{tikzcd}\]
Thus, we see that the stabilization $BG \to B\mathcal{Q}(X; \Fbb)^+$ factors through $\operatorname{BGL}_{\otimes}(X; \Fbb)^+$. We claim any map $f: BG \to \operatorname{BGL}_{\otimes}(X; \Fbb)^+$ is null-homotopic. Indeed, since $G$ is rationally acyclic, we observe that $\operatorname{Hom}_{Grp}(G, V) = 0$ for any $V$ that is rational, as $V$ is torsion-free. Thus, any map $BG \to \operatorname{BGL}_{\otimes}(X; \Fbb)^+$ lifts to the universal cover $\widetilde{\operatorname{BGL}_{\otimes}(X; \Fbb)^+}$. We now show that any map $BG \to \tau_{>1} \operatorname{BGL}_{\otimes}(X; \Fbb)^+$ is null-homotopic. A null-homotopy is equivalent to being able to lift the map to along the path-space fibration. By standard obstruction theory (e.g., Theorem 7.37 of \cite{DavisKirk2001}, which we can apply as the universal cover is simply-connected and the loop space is simple), the obstructions to doing so lie in $H^{n+1}(BG; \pi_n(\Omega \widetilde{\operatorname{BGL}_{\otimes}(X; \Fbb)^+})) = H^{n+1}(BG; \pi_{n+1}(\operatorname{BGL}_{\otimes}(X; \Fbb)^+))$. Since $G$ is rationally acyclic, the cohomology groups of $G$ is rational coefficients are zero in degree above zero. Lemma~\ref{lem::BGL_rational} then concludes the proof.
\end{proof}

As a corollary of Lemma~\ref{lem::factorized_null_homotopuic}, we also obtain the following.
\begin{cor}\label{cor::lin_factorizable_null}
Let $G$ be a rationally acyclic group. Suppose a QCA representation $\varphi: G \to \mathcal{Q}(X, q)$ is linearizable or trivially stably linearizable, then the stabilization $\varphi^{st}$ is null-homotopic.
\end{cor}

\begin{proof}
The case for $\varphi$ being trivially stably linearizable reduces to the case of being linearizable, as we could just replace $\varphi$ by $\varphi \otimes c_{id}$ without changing the induced map $BG \to B\mathcal{Q}(X)$.

Now suppose $\varphi: G \to \mathcal{Q}(X, q)$ is linearizable. Recall that two group homomorphisms $f, g: G_1 \to G_2$ induce freely homotopic maps $Bf, Bg: BG_1 \to BG_2$ if and only if they differ by conjugations. Thus, the map $B\varphi: BG \to B\mathcal{Q}(X, q)$ is freely homotopic to a map $B\varphi': BG \to B\mathcal{Q}(X, q)$ such that $\varphi': G \to \mathcal{Q}(X, q)$ is linearly onsite. Lemma~\ref{lem::factorized_null_homotopuic} then shows that
\[(\varphi')^{st}: BG \xrightarrow{B\varphi'} B\mathcal{Q}(X, q) \to B\mathcal{Q}(X) \to B\mathcal{Q}(X)^+ \xrightarrow{\sim} K(\mathbf{C}(X))_1\]
is null-homotopic. This shows that the map
\[(\varphi)^{st}: BG \xrightarrow{B\varphi} B\mathcal{Q}(X, q) \to B\mathcal{Q}(X) \to B\mathcal{Q}(X)^+ \xrightarrow{\sim} K(\mathbf{C}(X))_1\]
is freely null-homotopic. For a space $Y$ and a connected $H$-space $Z$, based-homotopy class of maps $Y \to Z$ are the same as free homotopy classes $Y \to Z$ (see Proposition 1.4.3 of \cite{may2011more}, for example). Thus, $(\varphi')^{st}$ is also based null-homotopic. This concludes the proof.
\end{proof}

We are now able to prove Theorem~\ref{thm::fin_gp_null}.
\begin{proof}[Proof of Theorem~\ref{thm::fin_gp_null}]
Suppose $\varphi$ is stably linearizable, then there is a linearizable $\psi: G \to \mathcal{Q}(X, r)$ such that $(\varphi \otimes \psi)$ is linearizable. Consider the stabilization map
\[(\varphi \otimes \psi)^{st}: BG \xrightarrow{B(\varphi \otimes \psi)} B\mathcal{Q}(X, qr) \to B\mathcal{Q}(X) \to B\mathcal{Q}(X)^+ \xrightarrow{\sim} K(\mathbf{C}(X))_1.\]
This is homotopic to the map
\[BG \xrightarrow{(\varphi^{st}, \psi^{st})} B\mathcal{Q}(X)^+ \times B\mathcal{Q}(X)^+ \xrightarrow{\mu} B\mathcal{Q}(X)^+, \]
where $\mu$ is the $H$-space structure on $B\mathcal{Q}(X)^+$. In particular, in the group $[BG, B\mathcal{Q}(X)^+]_*$, $(\varphi \otimes \psi)^{st}$ represents the class $[\varphi^{st}] + [\psi^{st}]$. From Corollary~\ref{cor::lin_factorizable_null}, we know $[\psi^{st}] = 0$ in this group. Thus we have that $0 = [\varphi^{st}] + 0$, so $\varphi^{st}$ is also based null-homotopic.\\

Suppose $\varphi$ is weakly linearizable, evidently the argument for null-homotopy only relied on the stable diagram being factorizable, and we analogously can do this modulo conjugations. Thus, $\varphi^{st}$ is null-homotopic. 
\end{proof}

The rest of this subsection proves Lemma~\ref{lem::commutator_perfect} and Lemma~\ref{lem::BGL_rational}. To do this, we introduce the following local notations.
\begin{defn} \label{def::key_category}
Let $(\mathcal{C}, \oplus, \otimes)$ be a category defined as follows: \begin{enumerate} \item The objects of $\mathcal{C}$ are assignments to each $x \in X$ of a vector space $\mathbb{F}^{f(x)}$, where $f(x) \in \mathbb{Z}_{\geq 0}$. \item The morphisms are all automorphisms given on each $f$ by \[ \prod_{x \in X} \operatorname{GL}_{f(x)}(\mathbb{F}). \] \item The operations $\oplus$ and $\otimes$ are given pointwise by direct sum and tensor product, respectively. \end{enumerate} Let $\mathcal{D}$ be the subcategory of $\mathcal{C}$ consisting of those objects for which $f(x) \geq 1$ for all $x \in \mathbb{Z}^d$. Observe that $(\mathcal{D}, \otimes)$ is a skeletal permutative category. \end{defn}

In particular, $B\mathcal{C}$ is an example of an $\mathbb{E}_{\infty}$-ring space (see the introduction of Chapter II of \cite{CohenLadaMay1976}). We also observe $\mathcal{D}$ fits the description for Proposition 1.5 of \cite{951cf383-5a4d-3775-b8b2-0bafd2f162fe}, which directly yields the following corollary.

\begin{cor}\label{cor::zd_h_space_identification}
   The commutator subgroup of $\operatorname{GL}_{\otimes}(X; \Fbb)$ is perfect. Furthermore, the plus-construction $\operatorname{BGL}_{\otimes}(X; \Fbb)^+$ at its commutator is the component at identity of $K(\mathcal{D})$. In particular, $\operatorname{BGL}_{\otimes}(X; \Fbb_p)^+$ is a group-like H-space.
\end{cor}

By changing the morphisms in $\mathcal{D}$ from $\prod_{x \in X} \operatorname{GL}_{f(x)}(\Fbb)$, $f(x) \geq 1$ to $\prod_{x \in X} \operatorname{PGL}_{f(x)}(\Fbb)$, Proposition 1.5 of \cite{951cf383-5a4d-3775-b8b2-0bafd2f162fe} also applies to show $\operatorname{PGL}_{\otimes}(X; \Fbb)$ has perfect commutator subgroup. This fact combined with Corollary~\ref{cor::zd_h_space_identification} proves Lemma~\ref{lem::commutator_perfect}.

To prove Lemma~\ref{lem::BGL_rational}, we record the following fact from rational homotopy theory.
\begin{lem}[Proposition V.3.3 of \cite{Bousfield1972}, See also Section 8 of \cite{Ivanov2022}]\label{lem::h_space_rational}
Let $Y$ be a nilpotent (e.g., simple) space and $\Tilde{H}_*(Y; \Zbb)$ is rational, then $\pi_*(Y)$ is rational.
\end{lem}

Finally, we prove Lemma~\ref{lem::BGL_rational}. Here we adapt a strategy analogous to that of Theorem V.5.3 of \cite{May1977EInfinityRingSpaces}, see also Theorem 10.2 of \cite{may2009preciselyeinftyringspaces}.

\begin{proof}[Proof of Lemma~\ref{lem::BGL_rational}]
We first note that $\operatorname{BGL}(X; \Fbb)^+$ is connected, so there is no issue at $\pi_0$ or $H_0$. Write $\operatorname{QS}(X)$ as a poset whose objects are function $q: X \to \Zbb_{>0}$ and $q \leq r$ if $q(x)$ divides $r(x)$ for all $x \in X$. In this case, we write $r = q \cdot q'$. For a natural number $n \in \Zbb_{>0}$, we also write $n$ to denote the constant function $n: X \to \{n\}$.\\

We note that each $q \in \operatorname{QS}(X)$ corresponds to the component $B\mathcal{C}_{q}$ corresponding to $q \in \pi_0(B\mathcal{C})$. We write $[q]$ to denote the corresponding element in $H_0(K\mathcal{C})$. Furthermore, for $q \leq r$ with $r = q \cdot q'$, there is a natural map $B\mathcal{C}_q \to B\mathcal{C}_r$ given by multiplication by $q'$.\\

The set-up above enables us to consider the following strict filtered colimit
\[L \coloneqq \operatorname{colim}_{q \in \operatorname{QS}(X)} B\mathcal{C}_{q}.\]
By Corollary~\ref{cor::zd_h_space_identification}, $\operatorname{BGL}(X; \Fbb)^+$ is the identity component of $K(\mathcal{D})$. Since $K(\mathcal{D})$ is the group completion of $B\mathcal{D}$, the axiomatic definition of localization for group completion implies that for any commutative ring $R$,
\[H_*(K(\mathcal{D}); R) \cong \pi_0(B\mathcal{D})^{-1} H_*(B\mathcal{D}; R).\]
In particular, the homology at the identity component corresponds to applying the localization (which is equivalently a colimit) at the homology of the identity component of $B\mathcal{D}$. In particular, this shows
\[H_*(\operatorname{BGL}(X; \Fbb)^+; R) \cong H_*(K(\mathcal{D})_1; R) \cong H_*(L; R) \cong \operatorname{colim}_{q \in \operatorname{QS}(X)} H_*(B\mathcal{C}_{q}; R).\]
For a prime $p$, we wish to show that $\Tilde{H}_*(L; \Fbb_p) = 0$. The universal coefficient theorem for homology will then imply the multiplication by $p$, in the sense of adding an element with itself $p-1$ times, is an isomorphism on $\Tilde{H}_*(L; \Zbb)$. If this holds for all prime $p$, then $\Tilde{H}_*(L; \Zbb)$ is rational. It then follows from Lemma~\ref{lem::h_space_rational} that $\operatorname{BGL}(X; \Fbb)^+$ has rational homotopy groups.\\

Since $B(\mathcal{C})$ is an $\Ebb_{\infty}$-ring space, the homology group $H_*(B\mathcal{C}; \Fbb_p)$ has two ring structures under the H-space structure given by $\oplus$ and the H-space structure given by $\otimes$. We denote products with respect to $\oplus$ as $*$ and products with respect to $\otimes$ as $\bullet$. We now make a subclaim that for any $x \in H_n(B\mathcal{C}_q; \Fbb_p)$ for $n > 0$ and $q \in \operatorname{QS}(X)$, for $t$ large enough, $x \bullet [p^t] \in H_n(B\mathcal{C}_{q\cdot p^k}; \Fbb_p)$ is $0$. This subclaim will imply that $\Tilde{H}_*(L; \Fbb_p) = 0$ for the following reason. Since $\Tilde{H}_*(L; \Fbb_p) = \operatorname{colim}_{q \in \operatorname{QS}(X)} \Tilde{H}_*(B\mathcal{C}_q; \Fbb_p)$ is a filtered colimit of abelian groups, we just need to show that any element $x \in \Tilde{H}_*(B\mathcal{C}_q; \Fbb_p)$ becomes $0$ at some point in the colimit stage. By the subclaim, it suffices to find a path in the poset $\operatorname{QS}(X)$ that goes through infinitely many $\bullet$ by $[p]$, and this path always exist.\\

It remains for us to show the subclaim. We first discuss a construction that will be helpful. In general for a space $Y$, the homology $H_*(Y; \Fbb_p)$ has a coproduct map $H_*(Y; \Fbb_p) \to H_*(Y \times Y; \Fbb_p) \cong H_*(Y; \Fbb_p) \otimes H_*(Y; \Fbb_p)$ induced by first the the diagonal map $\Delta: Y \to Y \times Y$ and then the Kunneth isomorphism. We use $\psi: H_*(Y; \Fbb_p) \to H_*(Y; \Fbb_p)^{\otimes p}$ to denote the $p$-fold iterated coproduct map. The map $\psi$ satisfies a co-commutativity condition in the sense that $\psi$ is invariant under any permutation $\sigma$ on the tensors of $H_*(Y; \Fbb_p)^{\otimes p}$ (i.e., $\sigma \circ \psi = \psi$). Write $\{x_i\}_{i \in I}$ as a basis of $H_*(Y; \Fbb_p)$ where each $x_i$ has pure degree. Write $I = (i_1, ..., i_p)$ as a multi-index such that $\deg(x_{i_1}) + ... \deg(x_{i_p}) = k$. An element $x \in H_k(Y; \Fbb_p)$ will be sent to $\psi(x) = \sum a_I x_{i_1} \otimes ... \otimes x_{i_p}$. Let $\sigma$ be a permutation of $p$-elements, we note that the co-commutativity of $\psi$ implies that $a_{I} = a_{\sigma I}$.\\

Proposition II.1.5 \cite{CohenLadaMay1976} computed how $*$ and $\bullet$ interact with each other in the homology of any $\Ebb_{\infty}$-ring space. In particular, for $x \in H_n(B\mathcal{C}_q; \Fbb_p) \subseteq H_n(B\mathcal{C}; \Fbb_p)$ it implies that
\[x \bullet [p] = x\bullet (\underbrace{[1] * ... * [1]}_{\text{$p$ times}}) = \sum a_I x_{i_1} * ... * x_{i_p}.\]
Since $*$ is commutative, we can group together any summand that differs from each other by a permutation $\sigma$ of $p$-elements. There is then $\frac{p!}{n_1! ... n_k!}$ many copies of the same summand where $n_1, ..., n_k$ are the number of the same elements in $I$. In particular, $\frac{p!}{n_1! ... n_k!}$ is divisible by $p$ unless $x_{i_1} = ... = x_{i_p} = z_j$ for some basis element $z_j$. Thus, we have that
\[x \bullet [p] = \sum z_j^{[p]},\]
where $z_j^{[p]}$ denotes $p$-th power of $z_j$ under $*$. Now $H_*(B(\mathcal{C}); \Fbb_p)$ is an $\Fbb_p$-algebra under $*$, so we have that
\[x \bullet [p] = (\sum z_j)^{[p]}.\]
Write $y = \sum z_j$, we now have that $x \bullet [p] = y^{[p]}$. Now observe that if $p$ does not divide $\deg(x) = n$, then $y = 0$ (this is because $z_j^{[p]}$ has degree $p \deg(z)$, which is equal to $q > 0$). Otherwise, we multiply $x$ by $[p^k] = [p] \bullet ... \bullet [p]$ until $\frac{\deg(x)}{p^k} = \frac{n}{p^k}$ is no longer an integer. This shows that $x \bullet [p^k] = 0$ for $k$ sufficiently large.
\end{proof}

\begin{remark}
    The same proof for Lemma~\ref{lem::BGL_rational} also shows that $\operatorname{BGL}(X; R)^+$ has rational homotopy groups for any commutative ring $R$ with unity. The proof for Lemma~\ref{lem::commutator_perfect} also works for any $R$. Consequently, Theorem~\ref{thm::fin_gp_null} holds for any $R$ with the property that all automorphisms of $M_n(R)$ are inner (e.g., fields, local rings). For rings where the automorphisms of $M_n(R)$ may not be inner, a lift to GL does not really make sense unless one interprets being linearly onsite to mean restricting to pointwise inner automorphisms. Under this interpretation, Theorem~\ref{thm::fin_gp_null} would hold for any $R$.
\end{remark}

\subsection{Projectivizable Null-Homotopy is Linearizable}

In this section, we will prove Theorem~\ref{thm::converse_on_PGL}, which would get a converse to Theorem~\ref{thm::fin_gp_null} when $X = *$.

\begin{proof}[Proof of Theorem~\ref{thm::converse_on_PGL}]
Suppose $\varphi: G \to \mathcal{Q}(X, q)$ restricts to $\prod_{x \in X} \operatorname{PGL}_{q(x)}(\Fbb)$ and induces a null-homotopic map $BG \to \operatorname{BPGL}(X; \Fbb)^+$. The surjection $\operatorname{GL}_{\otimes}(X; \Fbb) \to \operatorname{PGL}(X; \Fbb)$ has central kernel and satisfies the main theorem of \cite{BERRICK1983172}, so the induced map $\operatorname{BGL}_{\otimes}(X; \Fbb)^+ \to \operatorname{BPGL}(X; \Fbb)^+$ is a fibration. The homotopy lifting property of fibrations allows us to lift the null-homotopy $BG \to \operatorname{BPGL}(X; \Fbb)^+$ to a map $BG \to \operatorname{BGL}_{\otimes}(X; \Fbb)^+$. We now have a square
\[\begin{tikzcd}
	BG && \\
	& {\operatorname{BGL}_{\otimes}(X; \mathbb{F})} & {\operatorname{BPGL}(X; \mathbb{F})} \\
	& {\operatorname{BGL}_{\otimes}(X; \mathbb{F})^+} & {\operatorname{BPGL}(X; \mathbb{F})^+}
	\arrow["{B\varphi}", from=1-1, to=2-3]
	\arrow["{\text{lift}}"', dashed, from=1-1, to=3-2]
	\arrow[from=2-2, to=2-3]
	\arrow[from=2-2, to=3-2]
	\arrow[from=2-3, to=3-3]
	\arrow["{\text{fibration}}"', from=3-2, to=3-3]
\end{tikzcd}\]
By Lemma~\ref{lem::homotopy_pullback_zero} below, the bottom square is in fact a pullback square, so we have a map $BG \to \operatorname{BGL}_{\otimes}(X; \Fbb)$ given by the pullback. This shows $\varphi$ is weakly linearizable. When $G$ is finite, the induced map $G \to \operatorname{GL}_{\otimes}(X; \Fbb)$ must land in some finite stage. This shows $\varphi: G \to \mathcal{Q}(*, q)$ is stably linearizable. 
\end{proof}

\begin{remark}
    When $X = *$, the map $\operatorname{BGL}_{\otimes}(\Fbb)^+ \to \operatorname{BPGL}(\Fbb)^+$ being a fibration is Theorem 7 of \cite{c80f3555-900f-35b5-ab05-6d397e6a44e6} and Proposition 6.2 of \cite{qca_grp_space_spectrum}.
\end{remark}

\begin{lemma}\label{lem::homotopy_pullback_zero}
Let $H$ be central in $G$ and suppose $BH \to BG \to BG/H$  induces a fibration $BH^+ \cong BH \to BG^+ \to BG/H^+$, then the square
\[\begin{tikzcd}
	BG & {BG/H} \\
	{BG^+} & {BG/H^+}
	\arrow[from=1-1, to=1-2]
	\arrow[from=1-1, to=2-1]
	\arrow[from=1-2, to=2-2]
	\arrow[from=2-1, to=2-2]
\end{tikzcd}\]
is a homotopy pullback square. Since the horizontal arrows are fibrations, this is a pullback square.
\end{lemma}

\begin{remark}\label{rmk::pullback}
   This lemma follows from the bottom diagram on Page 54 of \cite{Berrick1982}. Furthermore, (5.11) of \cite{Berrick1982} shows the square is homotopy bicartesian.
\end{remark}

For completeness, we include a proof below.
\begin{proof}
    For a fibration $F \to E \to B$ where:
    \begin{itemize}
        \item $B$ is path-connected.
        \item $\pi_1(B)$ acts trivially on $\pi_n(F)$. 
        \item $F$ is connected simple with $\pi_1(F)$.
    \end{itemize}
    (call these conditions $(\dagger)$), there is a canonical class in $H^2(B; \pi_1(F))$ given as follows. Consider the cohomological Serre spectral sequence of the fibration $F \to E \to B$, there is a $d_2$ differential arising
    \[d_2: E^{0,1}_2 = H^0(B; H^1(F; \pi_1F)) \cong H^1(F; \pi_1 F) \cong \operatorname{Hom}(\pi_1 F; \pi_1 F)\]\[ \to E^{2,0}_2 = H^2(B; H^0(F; \pi_1 F)) = H^2(B; \pi_1 F). \]
    The image of the identity map $\pi_1 F \to \pi_1 F$ under $d_2$ defines a canonical class in $H^2(B; \pi_1 F)$. Note that when we consider the fibration $BH \to BG \to BG/H$, which satisfies $(\dagger)$, this class is the cohomology class $\beta$ classifying the central extension $H \to G \to G/H$.\\

    The induced fibration $BH \to BG^+ \to BG/H^+$ also satisfies the condition $(\dagger)$ above, and hence as a cohomology class $\alpha$. Pullbacks respect this cohomology class. Note that a pull-back 
\[\begin{tikzcd}
	P & {BG/H} \\
	{BG^+} & {BG/H^+}
	\arrow[from=1-1, to=1-2]
	\arrow[from=1-1, to=2-1]
	\arrow[from=1-2, to=2-2]
	\arrow[from=2-1, to=2-2]
\end{tikzcd}\]
implies that there is a fibration $BH \to P \to BG/H$, which must imply $P \cong B\pi_1 P$ from a long exact sequence computation. It suffices to show the extension class $\gamma$ classifying $H \to \pi_1(P) \to G/H$ is the same as the class $\beta$ for $H \to G \to G/H$.\\

Let $i: BG/H \to BG/H^+$ denote the plus-construction map. By naturality we have that $i^*(\alpha) = \gamma$. On the other hand, clearly the Serre spectral sequence for the fibration $BH \to BG^+ \to BG/H^+$ is isomorphic to the Serre spectral sequence for the fibration $BH \to BG \to BG/H$ via the plus-construction map. This isomorphism would send $i^*(\alpha) = \beta$ in the $E_2$-page, in particular. Thus, we have that $\beta = \gamma$, and hence $P = BG$ and the initial square is a pullback square.
\end{proof}

\subsection{Linearization Obstructions in the Unitary Case}\label{sec::unitary_obstruction}

For the unitary case over $\mathbb{C}$, the authors of \cite{qca_grp_space_spectrum} constructed a certain symmetric monoidal category $\mathbf{C}^*(X)$ of unitary quantum spin systems over $X$. The \textit{space of unitary QCA} $\mathbf{Q}^*(X)$ is defined to be
\[\mathbf{Q}^*(X) = \Omega K(\mathbf{C}^*(X))_1\]
such that $\pi_0(\mathbf{Q}^*(X))$ is the classical QCA classification group of $X$ in the physics literature (see Theorem 5.10 of \cite{qca_grp_space_spectrum}), denoted $\mathcal{Q}^*(X)/\mathcal{C}^*(X)$. Furthermore, in the same theorem, $K(\mathbf{C}^*(X))_1$ may be identified as $B\mathcal{Q}^*(X)^+$. Theorem 5.11 of \cite{qca_grp_space_spectrum} also gives an analogous $\Omega$-spectrum of the form $\{\mathbf{Q}^*(X \times \Zbb^n)\}_{n \geq 0}$.\\

We can obtain the unitary analogs of Theorem~\ref{thm::fin_gp_null} and Theorem~\ref{thm::converse_on_PGL} as follows.

\begin{theorem}
Let $G$ be a rationally acyclic group. Suppose an unitary QCA representation $\varphi: G \to \mathcal{Q}^*(X, q)$ is stably linearizable or weakly linearizable with respect to (\ref{eq::unitary_unstable_lift}), then the stabilization $\varphi^{st}: BG \to K(\mathbf{C}^*(X))_1$ is null-homotopic.
\end{theorem}

\begin{proof}
We modify the construction to Definition~\ref{def::key_category} to consider a groupoid $\mathcal{C}^*$ whose objects are assignments to each $x \in X$ a vector space $\Cbb^{f(x)}$ and the  automorphisms are given by $\prod_{x \in X} U_{f(x)}$, with operations given by direct sum and tensor product. Proposition 1.5 of \cite{951cf383-5a4d-3775-b8b2-0bafd2f162fe} still applies to identify $K(\mathcal{C}^*)$ as $\operatorname{BU}_{\otimes}(X)^+$. A nearly identical argument as in the proof of Lemma~\ref{lem::BGL_rational} shows that $K(\mathcal{C}^*)$ is rational. Thus, the rest of the argument in Theorem~\ref{thm::fin_gp_null} proceeds similarly.
\end{proof}

\begin{theorem}
Let $G$ be a group. Suppose an unitary QCA representation $\varphi: G \to \mathcal{Q}^*(X, q)$ lands in $\prod_{x \in X} \operatorname{PU}(q_x)$, and the induced map $\varphi': BG \to \operatorname{BPU}(X) \to \operatorname{BPU}(X)^+$ is null-homotopic, then $\varphi$ is weakly linearizable with respect to (\ref{eq::unitary_unstable_lift}). If $G$ is furthermore finite, then $\varphi$ is stably linearizable with respect to (\ref{eq::unitary_unstable_lift}).
\end{theorem}

\begin{proof}
The exact sequence $U(1) \to U_{\otimes} \to PU$ induces a surjective map $U_{\otimes}(X) \to \operatorname{PU}(X)$ with central kernel. This satisfies the criterion in \cite{BERRICK1983172} and hence gives a fibration $BU_{\otimes}(X)^+ \to B\operatorname{PU}(X)^+$. The proof then proceeds similarly as in Theorem~\ref{thm::converse_on_PGL} with the relevant square still being a pullback, due to Lemma~\ref{lem::homotopy_pullback_zero}.
\end{proof}

\section{Cohomological Interpretation of Linearization}\label{sec::cohomology}
Adapting a construction of Dror, we define a complete set of universal obstruction classes for the stable linearization of a QCA representation. The construction works over any field and also in the unitary case. These classes take values in group cohomology, with coefficients in the classification groups of QCA, together with $\mu(\Fbb)$ (the roots of unity in $\Fbb$) or $U(1)$ in the top degree. This agrees with predictions from~\cite{Tu_2026}.

A noteworthy feature of this construction is that it produces a universal class in each degree. Consequently, any otherwise constructed obstruction class with values in the same cohomology group is weaker than or equal to ours: if the corresponding universal obstruction class vanishes, then any cohomology class obtained by other means must vanish as well. 

\subsection{Dror's Thesis}\label{sec::dror}
Consider a group $Q$ whose commutator subgroup $C=[Q,Q]$ is perfect. Examples include $Q=GL(R)$, relevant to algebraic $K$-theory, $Q=\CQ(X)$, the total QCA group over a field $\mathbb F$, as well as, $Q=\CQ^*(X)$, the unitary QCA group. In these examples, the commutator subgroups are respectively the elementary subgroup $E(R)$, the group of special circuits $\CalC^{sp}(X)$ and the group of unitary circuits $\CalC^*(X)$. Dror’s thesis \cite{Dror1972} gives an explicit construction, in terms of group cohomology, of the homotopy fiber of the natural map
\begin{equation}
    BQ\rightarrow BQ^+.
\end{equation}
We record this construction below.
\begin{construction}[Killing homology]
\label{constr:homological-killing}
Let $n\geq 2$, and let $X_n$ be a path-connected, homologically $n$-connected space. Set
\begin{equation}
A_n:=H_{n+1}(X_n;\mathbb Z).
\end{equation}
Since $H_n(X_n;\mathbb Z)=0$, the universal coefficient theorem gives a natural identification
\begin{equation}
H^{n+1}(X_n;A_n)\cong \operatorname{Hom}(H_{n+1}(X_n;\mathbb Z),A_n).
\end{equation}
Let $\iota_n\in H^{n+1}(X_n;A_n)$ be the universal class corresponding to the identity map $\operatorname{id}_{A_n}:A_n\to A_n$. We define $X_{n+1}$ to be the total space of the principal $K(A_n,n)$-fibration
\begin{equation}
K(A_n,n)\longrightarrow X_{n+1}\longrightarrow X_n
\end{equation}
classified by $\iota_n$. In other words, $X_{n+1}$ is the homotopy fiber of $X_n \to K(A_n, n+1)$.
\end{construction}

\begin{lemma} \label{lem:homological-killing}
The space $X_{n+1}$ is homologically $(n+1)$-connected.
\end{lemma}

Before proving the lemma, let us see how it gives a construction of the homotopy fiber. Let $X_1=BC$. Since $H_1(X_1;\mathbb Z)=C/[C,C]=0$, the universal coefficient theorem gives a canonical isomorphism
\begin{equation}
H^2(X_1;H_2(X_1))\cong
\operatorname{Hom}(H_2(X_1),H_2(X_1)).
\end{equation}
Let $\iota_1\in H^2(X_1;H_2(X_1))$ be the class corresponding to the identity map in the right-hand side. This determines the universal central extension of the group $C$
\begin{equation}
    1\longrightarrow A_1\longrightarrow \widetilde C\longrightarrow C\longrightarrow 1,
\end{equation} where $A_1=H_2(X_1)$.\\

The name is justified by the following universal property. Let
$1\to A\to E\to C\to 1$ be any central extension of $C$. Its extension class
$\alpha\in H^2(X_1;A)$ corresponds, under the universal coefficient isomorphism
$H^2(X_1;A)\cong \operatorname{Hom}(H_2(X_1),A)$, to a homomorphism
$\phi_\alpha:A_1\to A$. Then there is a unique homomorphism
$\widetilde C\to E$ over $C$, and its restriction to $A_1$ is $\phi_\alpha$:
\begin{equation}
\begin{tikzcd}
1 \arrow[r] & A_1 \arrow[r] \arrow[d, "\phi_\alpha"'] &
\widetilde C \arrow[r] \arrow[d] &
C \arrow[r] \arrow[d, equal] &
1 \\
1 \arrow[r] & A \arrow[r] &
E \arrow[r] &
C \arrow[r] &
1 .
\end{tikzcd}
\end{equation}
\begin{prop}\label{prop:superperfect}
The group $\widetilde C$ is superperfect; that is,
\begin{equation}
H_1(\widetilde C;\mathbb Z)=H_2(\widetilde C;\mathbb Z)=0.
\end{equation}
\end{prop}

\begin{proof}
We first show that $\widetilde C$ is perfect. The  five-term exact sequence for the central extension
$1\to A_1\to \widetilde C\to C\to 1$ is
\begin{equation}
H_2(\widetilde C;\mathbb Z)\longrightarrow H_2(C;\mathbb Z)\longrightarrow A_1\longrightarrow
H_1(\widetilde C;\mathbb Z)\longrightarrow H_1(C;\mathbb Z)\longrightarrow 0.
\end{equation}
Since $C$ is perfect, $H_1(C;\mathbb Z)=0$. Moreover, by construction of the universal central extension, the map
$H_2(C;\mathbb Z)\to A_1$ is the identity map on
$A_1=H_2(C;\mathbb Z)$. Exactness therefore gives
$H_1(\widetilde C;\mathbb Z)=0$.

It remains to prove that $H_2(\widetilde C;\mathbb Z)=0$. Let
$1\to A\to E\to \widetilde C\to 1$ be any central extension. Since
$\widetilde C$ is perfect, the composite $E\to \widetilde C\to C$ is again
a central extension of $C$. Indeed, the
preimage of $A_1\subset \widetilde C$ is central in $E$; commutators with
such lifts define homomorphisms $\widetilde C\to A$, and these vanish
because $\widetilde C$ is perfect.

By the universal property above, there is a unique map
$\widetilde C\to E$ over $C$. Composing with $E\to \widetilde C$ gives a
map $\widetilde C\to \widetilde C$ over $C$, hence the identity by the same
universal property. Thus, the central extension $E$ of $\widetilde C$ splits.

Since $\widetilde C$ is perfect, central extensions of $\widetilde C$ by an
abelian group $A$ are classified by
$H^2(B\widetilde C;A)\cong \operatorname{Hom}(H_2(\widetilde C;\mathbb Z),A)$.
Taking $A=H_2(\widetilde C;\mathbb Z)$, the identity map must correspond to
a split extension, hence to the zero class. Therefore
$H_2(\widetilde C;\mathbb Z)=0$.
\end{proof}

Let $X_2=B\widetilde C$. Proposition~\ref{prop:superperfect} implies that $X_2$ is homologically $2$-connected. We now apply Construction~\ref{constr:homological-killing} repeatedly. Thus, for each $n\geq 2$, we set $A_n:=H_{n+1}(X_n;\mathbb Z)$ and define $X_{n+1}$ by a principal fibration
\begin{equation}
K(A_n,n)\longrightarrow X_{n+1}\longrightarrow X_n .
\end{equation}
This gives a tower
\begin{equation}\label{eq:Dror}
\cdots \longrightarrow X_{n+1}\longrightarrow X_n\longrightarrow \cdots \longrightarrow X_3\longrightarrow X_2\longrightarrow X_1\longrightarrow X_0\coloneqq BQ.
\end{equation}
By Lemma~\ref{lem:homological-killing}, the space $X_n$ is homologically $n$-connected for every $n\geq 2$.

\begin{theorem}[Dror \cite{Dror1972}]\label{thm::dror}
With
\begin{equation}
X_\infty:=\operatorname*{holim}_{n\geq 0} X_n,
\end{equation}
there is a fibration sequence
\begin{equation}
X_\infty\longrightarrow BQ\longrightarrow BQ^+.
\end{equation}
\end{theorem}

\begin{proof}
The homotopy limit $X_\infty$ is acyclic, i.e., its reduced integral homology vanishes in all degrees. Moreover, the map $X_\infty\to BC$ is surjective on fundamental groups. Therefore the homotopy cofiber of $X_\infty\to BC$ is simply connected, and the induced map from $BQ$ to this homotopy cofiber is a homology equivalence. By the characterization of the plus construction, this homotopy cofiber is equivalent to $BC^+$. In other words, $X_\infty$ is the homotopy fiber of $BC\rightarrow BC^+.$ Since the square
\begin{equation}
\begin{tikzcd}
BQ \arrow[r] \arrow[d]
& BQ^+ \arrow[d] \\
BC \arrow[r]
& BC^+
\end{tikzcd}
\end{equation}
is cartesian and the vertical maps are covering maps, the homotopy fibers of the two horizontal maps are naturally equivalent, after choosing compatible basepoints. Since $X_\infty$ is the homotopy fiber of $BC\to BC^+$, it follows that $X_\infty$ is also the homotopy fiber of $BQ\to BQ^+$.

\end{proof}

\begin{proof}[Proof of Lemma~\ref{lem:homological-killing}]
Consider the Serre spectral sequence of the fibration
\begin{equation}
K(A_n,n)\longrightarrow X_{n+1}\longrightarrow X_n .
\end{equation} with
\begin{equation}
E^2_{p,q}\cong H_p(X_n;H_q(K(A_n,n);\mathbb Z)).
\end{equation}

The fiber $K(A_n,n)$ is $(n-1)$-connected. In low degrees, its homology is 
\begin{equation}
H_q(K(A_n,n);\mathbb Z)\cong
\begin{cases}
\mathbb Z, & q=0,\\
0, & 0<q<n,\\
A_n, & q=n,\\
0, & q=n+1.
\end{cases}
\end{equation} On the other hand, $H_p(X_n;\mathbb Z)=0$ for $0<p\leq n$.

It follows that, in total degree at most $n+1$, the only possibly nonzero terms are $E^2_{0,0}\cong \mathbb Z$, $E^2_{n+1,0}\cong H_{n+1}(X_n;\mathbb Z)$, and $E^2_{0,n}=H_0(X_n, H_n(K(A_n,n);\mathbb Z))$.

The differential has bidegree
\begin{equation}
d^r:E^r_{p,q}\longrightarrow E^r_{p-r,q+r-1}.
\end{equation}
The term $E^r_{n+1,0}$ receives no differentials, since an incoming differential would have source $E^r_{n+1+r,1-r}$, which vanishes for $r\geq 2$. 

For $2\leq r\leq n$, the outgoing differential
\begin{equation}
d^r:E^r_{n+1,0}\longrightarrow E^r_{n+1-r,r-1}
\end{equation}
has target in a row $0<r-1<n$, and this row vanishes because $H_{r-1}(K(A_n,n);\mathbb Z)$ vanishes. Therefore the first possible nonzero differential out of $E_{n+1,0}$ is $d^{n+1}.$
Under the identifications above, this is a homomorphism 
\begin{equation}
d^{n+1}:H_{n+1}(X_n;\mathbb Z)\to H_0(X_n, H_n(K(A_n,n);\mathbb Z))\cong A_n.
\end{equation}

We now identify this differential. By naturality of the Serre spectral sequence the transgression $d^{n+1}$ is evaluation against the classifying class $\iota_n\in H^{n+1}(X_n;A_n)$. Therefore the differential is an isomorphism.

Consequently, both terms $E^{n+1}_{n+1,0}$ and $E^{n+1}_{0,n}$ are killed by this differential. It follows that $\widetilde H_i(X_{n+1};\mathbb Z)=0$ for $i\leq n+1$. Therefore $X_{n+1}$ is homologically $(n+1)$-connected.
\end{proof}
\subsection{Universal Obstruction Classes}
The construction above produces a sequence of classes whose pullbacks give the successive obstruction classes for lifting along the Dror tower. Let
\begin{equation}
\widetilde{\varphi}_0:=B\varphi:BG\longrightarrow BQ=X_0
\end{equation}
be the map induced by a group homomorphism $\varphi:G\to Q$. The first question is whether $\widetilde{\varphi}_0$ lifts to $X_1$:
\begin{equation}
\begin{tikzcd}
& X_1 \arrow[d] \\
BG \arrow[r, "\widetilde{\varphi}_0"] \arrow[ru, dashed] & X_0 .
\end{tikzcd}
\end{equation}
Such a lift exists precisely when $\varphi$ factors through $C\subset Q$; equivalently, when the composite
\begin{equation}
G\xrightarrow{\varphi}Q\longrightarrow Q/C=: A_0
\end{equation}
is trivial. This composite represents the first obstruction class in $H^1(BG;A_0)\cong \Hom(G, A_0)$.

More generally, suppose a lift $\widetilde{\varphi}_n:BG\to X_n$ has been chosen. The obstruction to lifting it further,
\begin{equation}
\begin{tikzcd}
& X_{n+1} \arrow[d] \\
BG \arrow[r, "\widetilde{\varphi}_n"] \arrow[ru, dashed] & X_n ,
\end{tikzcd}
\end{equation}
is given by
\begin{equation}
\widetilde{\varphi}_n^*\iota_n\in H^{n+1}(BG;A_n),
\end{equation}
where $\iota_n\in H^{n+1}(X_n;A_n)$ is the universal class and $A_n=H_{n+1}(X_n;\mathbb Z)$.

\begin{prop}\label{prop::obstruction}
Let $\widetilde{\varphi}_0=B\varphi:BG\to X_0=BQ$. Suppose that the first $n$ obstruction classes
\begin{equation}
\widetilde{\varphi}_i^*\iota_i\in H^{i+1}(BG;A_i),
\qquad 0\leq i\leq n-1,
\end{equation}
vanish. Then $\widetilde{\varphi}_0$ admits a lift
\begin{equation}
\widetilde{\varphi}_n:BG\longrightarrow X_n.
\end{equation}
In particular, if all obstruction classes vanish, then $\widetilde{\varphi}_0$ admits a compatible lift
\begin{equation}
\widetilde{\varphi}_\infty:BG\longrightarrow X_\infty.
\end{equation}
Equivalently, the induced map
\begin{equation}
BG\longrightarrow BQ^+
\end{equation}
is null-homotopic.
\end{prop}
\begin{proof}
Let
\begin{equation}
p_i:X_{i+1}\longrightarrow X_i
\end{equation}
denote the $i$th stage of the tower. By the construction of the tower, the class
\begin{equation}
\iota_i\in H^{i+1}(X_i;A_i)
\end{equation}
is the universal obstruction to lifting a map into $X_i$ through $p_i$. More precisely, if $f:Y\to X_i$ is a map, then the primary obstruction to the existence of a lift $\widetilde f:Y\to X_{i+1}$ with $p_i\circ \widetilde f\cong f$ is the class
\begin{equation}
f^*\iota_i\in H^{i+1}(Y;A_i).
\end{equation}
The vanishing of this class is equivalent to the existence of such a lift.

We apply this with $Y=BG$ and argue by induction. If all obstruction classes vanish, the inductive argument produces a compatible system of lifts ${\widetilde{\varphi}_i:BG\to X_i}$. These maps determine a map into the homotopy inverse limit
\begin{equation}
\widetilde{\varphi}_\infty:BG\longrightarrow X_\infty.
\end{equation}

Finally, by construction, $X_\infty$ is the homotopy fiber of the map
\begin{equation}
BQ\longrightarrow BQ^+.
\end{equation}
Therefore a map $BG\to BQ$ admits a lift to $X_\infty$ if and only if its composite
\begin{equation}
BG\longrightarrow BQ\longrightarrow BQ^+
\end{equation}
is null-homotopic. This proves the final assertion.
\end{proof}

\subsection{Anomalies}\label{sec::anomaly}
We now apply the universal obstruction classes to QCA representations. In the unitary case, these classes recover all conjectural obstruction classes appearing in the literature, often referred to as anomalies or anomaly indices~\cite{Tu_2026, kawagoe2025anomaly, kapustinXu2025higher}. Moreover, the present framework provides a concrete group-theoretic procedure for constructing them.

Let $Q=\mathcal Q(X)$. A QCA representation determines a homomorphism
\begin{equation}
  \varphi \colon G \longrightarrow \mathcal Q(X,q) \longrightarrow \mathcal Q(X).  
\end{equation}
For concreteness, we take $X=\mathbb Z^n$. The discussion below extends without essential change to the case $X=Y\times \mathbb Z^n$, where $Y$ is any metric space of bounded geometry. That $\{\mathbf Q(\ZZ^n; \mathbb F)\}_n$ and $\{\mathbf Q^*(\ZZ^n)\}_n$ assemble into an $\Omega$-spectrum greatly simplifies the situation. 
\begin{lemma}\label{lem::indexF}
Let $Q=\CQ(\ZZ^n)$ be the total QCA group over $\mathbb F$, and let $C=\CalC(\ZZ^n)$ be the subgroup of quantum circuits. Then, for each $i$, there are natural isomorphisms
\begin{equation}
A_i = \pi_i \mathbf Q(\ZZ^n; \mathbb F).
\end{equation}
In particular, for $0\leq i\leq n$, we have by Proposition~\ref{prop::circuit_type}
\begin{equation}
A_i \cong \mathcal Q(\ZZ^{n-i})/\mathcal C(\ZZ^{n-i}) \text{ and } A_n \cong \mathcal{Q}(\Zbb^0)/\mathcal{C}^{sp}(\Zbb^0).
\end{equation}
\end{lemma}
\begin{proof}
    First, $A_0=\CQ(\ZZ^n)/ \CalC(\ZZ^n)=\pi_0 \mathbf Q(\ZZ^n; \mathbb F)$. Second, $A_1=H_2(\CalC(\ZZ^n); \ZZ)=\pi_1 \mathbf Q(\ZZ^n; \mathbb F)$ according to Corollary~C of~\cite{qca_grp_space_spectrum}. For $n\geq 2$, $A_n=\pi_{n}X_\infty$ by Dror's tower. Since $BQ$ is a classifying space, it is a $K(Q,1)$. Thus
\begin{equation}
\pi_i(BQ)=0
\end{equation}
for $i\geq 2$, while $\pi_1(BQ)\cong Q$. By Theorem~\ref{thm::dror}, there is a long exact sequence for $n\geq 3$,
\begin{equation}
0=\pi_n(BQ)\longrightarrow \pi_n(BQ^+)
\longrightarrow \pi_{n-1}(X_\infty)
\longrightarrow \pi_{n-1}(BQ)=0.
\end{equation}
Hence the connecting homomorphism is an isomorphism
\begin{equation}
\pi_n(BQ^+)\cong \pi_{n-1}(X_\infty).
\end{equation}
We obtain
\begin{equation}
\pi_{n+1}(BQ^+)\cong A_{n},
\qquad n\geq 2.
\end{equation}
The result is then implied by Theorem~D and E of~\cite{qca_grp_space_spectrum}. 
\end{proof}
\begin{lemma}\label{lem::indexUnitary}
Let $Q=\CQ^*(\ZZ^n)$ be the total unitary QCA group, and let $C=\CalC^*(\ZZ^n)$ be the subgroup of unitary quantum circuits. Then, for each $i$, there are natural isomorphisms
\begin{equation}
A_i = \pi_i \mathbf Q^*(\ZZ^n).
\end{equation}
In particular, for $0\leq i\leq n$, we have
\begin{equation}
A_i \cong \mathcal Q^*(\ZZ^{n-i})/\mathcal C^*(\ZZ^{n-i}).
\end{equation}
\end{lemma}
\begin{proof}
    The proof is identical, using the unitary analogues of the same results from~\cite{qca_grp_space_spectrum}.
\end{proof}
\begin{table}[th]
\centering
\caption{The groups $A_i$ over a field $\mathbb F$, specifically over the complex numbers $\mathbb{C}$, and in the unitary case.}
\label{table_q_z}
\begin{adjustbox}{width=0.8\linewidth}
\begin{tabular}{lccc}
\toprule
The groups $A_i$ & Over a field $\mathbb F$ & Over $\mathbb{C}$ & Unitary\\
\midrule
$i=0$ & $\CQ(\ZZ^{n})/\mathcal C(\ZZ^{n})$ & $\CQ(\ZZ^{n})/\mathcal C(\ZZ^{n})$ & $\CQ(\ZZ^{n})/\mathcal C(\ZZ^{n})$\\
$i=1$ & $\CQ(\ZZ^{n-1})/\mathcal C(\ZZ^{n-1})$ & $\CQ(\ZZ^{n+1-i})/\mathcal C(\ZZ^{n+1-i})$ & $\CQ^*(\ZZ^{n+1-i})/\mathcal C^*(\ZZ^{n+1-i})$\\
$\cdots$ & $\cdots$ &$\cdots$ &$\cdots$ \\
$i=n-2$ & $\CQ(\ZZ^{2})/\mathcal C(\ZZ^{2})$ & $\CQ(\ZZ^{2})/\mathcal C(\ZZ^{2})$ & 0\\
$i = n-1$ & $K_0(\operatorname{Az}(\mathbb F)) = \operatorname{Br}(\mathbb{F}) \oplus \Qbb_{>0}$ & $\Qbb_{>0}$ & $\Qbb_{>0}$\\
$i = n$ & $\frac{\Qbb}{\Zbb} \otimes \mathbb F^{\times}$ & 0 & 0\\
$i = n+1$ & $\mu(\mathbb F) \oplus (K_2(\mathbb F) \otimes \Qbb)$ & $\frac{\Qbb}{\Zbb} \oplus (K_2(\Cbb) \otimes \Qbb)$ & $U(1) \oplus \text{rational vector space}$\\
$i>n+1$ & rational vector spaces & rational vector spaces & rational vector spaces\\
\bottomrule
\end{tabular}
\end{adjustbox}
\end{table}

\begin{defn}\label{def::universalObstruction}
Suppose $G$ is a group. Let 
\begin{equation} \varphi^\mathrm{st} \colon G \xrightarrow{\varphi} \CQ(\ZZ^n,q; \Fbb) \longrightarrow \CQ(\ZZ^n; \Fbb) \end{equation} be a QCA representation on the $n$-dimensional lattice $\ZZ^n$. We define the degree-1 universal obstruction class to be
\[u_1(\varphi) \coloneqq [\varphi^{st}] \in \operatorname{Hom}(G, \mathcal{Q}(\Zbb^n; \Fbb)^{ab}) = H^1(G; \pi_0 \mathbf{Q}(\Zbb^n; \Fbb)).\]
Suppose the degree-$k$ obstructions up to $k = i$ vanish. We define the degree-$(i+1)$ \textit{universal obstruction class} to be
\begin{equation} 
u_{i+1}(\varphi)\coloneqq(\widetilde{\varphi}_i^\mathrm{st})^*\iota_i \in H^{i+1}\bigl(BG;\pi_i \mathbf{Q}(\mathbb{Z}^n; \Fbb)\bigr). \end{equation} 
Similarly, given a unitary QCA representation
\begin{equation} 
\varphi^\mathrm{st} \colon G \longrightarrow \CQ^*(\ZZ^n,q) \longrightarrow \CQ^*(\ZZ^n), 
\end{equation} we define the \textit{universal obstruction classes} iteratively to be 
\begin{equation} u_{i+1}(\varphi)\coloneqq (\widetilde{\varphi}_i^\mathrm{st})^*\iota_i \in H^{i+1}\bigl(BG;\pi_i \mathbf{Q}^*(\mathbb{Z}^n)\bigr). \end{equation} 
\end{defn}

Note that Definition~\ref{def::universalObstruction} applies equally as well to give universal obstruction classes group homomorphisms $G \to \mathcal{Q}(\Zbb^n; \Fbb)$ and $G \to \mathcal{Q}^*(\Zbb^n)$.

\anomalies*
\begin{proof}
    It follows from Theorem~\ref{thm::fin_gp_null}, Proposition~\ref{prop::obstruction}, Lemmas~\ref{lem::indexF} and~\ref{lem::indexUnitary}.
\end{proof}
\begin{cor}
Suppose that $\mathbb F=\CC$ and that $G$ is finite (or more generally, rationally acyclic). Then the only potentially nonzero universal obstruction classes occur, for $0\leq i\leq n+1$, in
\begin{equation}
H^{i+1}\bigl(BG;\mathcal Q(\ZZ^{n-i})/\mathcal C(\ZZ^{n-i})\bigr),
\end{equation}
and, in addition, in
\begin{equation}
H^{n+2}(BG;\Qbb/\ZZ).
\end{equation}
\end{cor}

\begin{proof}
The obstruction class associated to $A_i$ lies in $H^{i+1}(BG;A_i)$. By the second columns of Table~\ref{table_q_z} (which are derived from Table 1 of \cite{qca_grp_space_spectrum}), for $0\leq i\leq n$ the groups $A_i$ are identified with
\begin{equation}
\mathcal Q(\ZZ^{n-i})/\mathcal C(\ZZ^{n-i}).
\end{equation}
The next non-trivial contribution is the $\Qbb/\ZZ$ summand in $A_{n+2}$. All remaining terms, including $K_2(\CC) \otimes \mathbb{Q}$, are rational vector spaces. Since $G$ is finite (or rationally acyclic), its rational cohomology vanishes in positive degree, and hence these rational summands do not give nonzero obstruction classes.
\end{proof}

Similarly, in the unitary case we have
\ObstructionUnitary*

\begin{proof}
The proof is identical to the algebraic counterpart over $\CC$ using the unitary column of Table~\ref{table_q_z} instead. The unitary column is derived from a later computation in Proposition~\ref{prop::unitary_pt_groups}.
\end{proof}

\begin{remark}
    One could take $Q=\mathcal Q(X,q)$ and thereby obtain a sequence of \textit{unstable} obstruction classes. These classes may be related to (unstable) linearizability, but the precise relationship is unclear.
\end{remark}

Lastly, we discuss examples of obstruction classes on QCA representations.
\begin{example}[Revisiting annular shifts]\label{exp::revisit}
The obstruction classes seem less effective when $G$ is not rationally acyclic. Recall the annular shift QCA in Example~\ref{exp::annular}. Viewed as a QCA representation of the group $\ZZ$, it has no nontrivial obstruction classes in the above sense.
Indeed, the degree-one obstruction lies in
\begin{equation}
    H^1\bigl(B\ZZ;\pi_0\mathbf Q^*(\ZZ^n)\bigr).
\end{equation}
In the two-dimensional unitary case, $\pi_0\mathbf Q^*(\ZZ^2)=0$, since every two-dimensional unitary QCA is stably a circuit~\cite{freedman2020classification}; see also the simplified proof in the appendix of~\cite{haah2021clifford}. Hence the degree-one obstruction vanishes. Moreover, since $B\ZZ\cong S^1$, one has
\begin{equation}
    H^k(B\ZZ;M)=0
\end{equation}
for every $\ZZ$-module $M$ and every $k>1$. Therefore the annular shift admits no nontrivial obstruction class in our obstruction theory. Nevertheless, we suspect that it is not stably linearizable.
\end{example}

\begin{example}[Arithmetic obstruction]
Let $\Fbb = \Qbb$ be the rationals and $X = \Zbb^1$. From Table~\ref{table_q_z}, the degree-2 obstruction class has coefficients in $\pi_1(\mathcal{Q}(\Zbb^1;\Qbb)) = \frac{\Qbb}{\Zbb} \otimes \Qbb^{\times}$. Observe that
\begin{equation}
    \Qbb^{\times} \cong C_2 \times (\bigoplus_{p\ \text{prime}} \Zbb) \implies \pi_1(\mathcal{Q}(\Zbb^1;\Qbb)) \cong \bigoplus_{p\ \text{prime}} \Qbb/\Zbb.
\end{equation}
There could be many QCA representations over a line and the field $\Qbb^1$ whose degree-1 obstruction vanish but degree-2 obstruction does not vanish. The inclusion map $G = \mathcal{C}(\Zbb^1; \Qbb) \to \mathcal{Q}(\Zbb^1; \Qbb)$ would be an example of such QCA representations. In Appendix~\ref{appendix::arithmetic}, we give detailed examples of finitely generated QCA representations whose degree-1 obstruction vanish but degree-2 does not.

This can be generalized to when $\Fbb = K$ is a number field. From Page 22 of \cite{Neukirch1999}, there is an exact sequence
\begin{equation}
    0 \to \mathcal{O}_K^{\times} \to K^{\times} \to J_K \to \operatorname{Cl}(K) \to 0,
\end{equation} where $\mathcal{O}_K$ is the associated ring of integers, $J_K$ is the free abelian group on non-zero prime ideals of $\mathcal{O}_K$, and $\mathrm{Cl}(K)$ is the ideal class group. Since $\operatorname{Cl}(K)$ is finite, this gives a surjection
\begin{equation}
    \pi_1(\mathcal{Q}(\Zbb^1; K)) = K^{\times} \otimes \Qbb/\Zbb \to \bigoplus_{\mathfrak{p} \neq 0 \in \operatorname{Spec}(\mathcal{O}_k)} \Qbb/\Zbb \to 0.
\end{equation}
Thus, $\pi_1(\mathcal{Q}(\Zbb^1; K))$ is non-zero and quite large, and we can again consider the inclusion of the circuit group. This is very different from when $\Fbb = \Rbb, \Cbb$ or in the unitary case, where the corresponding group is $0$, so their QCA representations over a line have no degree-2 obstructions. 
\end{example}

\section{QCA Space Computations}\label{sec::computation}

From the stable homotopy theory interpretation of linearization, we are motivated to study the based homotopy classes of maps from $BG$ into the QCA space over $X$. Although the structure of the QCA space is in general complicated, we show that they are surprisingly simple in special cases. In this section, we will prove the following theorem.
\CSplit*

Theorem~\ref{thm::complex_qca_split}(1-3) is obtained by computing the stable $k$-invariants of their (connective) deloopings and showing that they vanish, except for $\mathbf{Q}(\Zbb^2; \mathbb{C})$ which relied on the vanishing of an unstable $k$-invariant. For Theorem~\ref{thm::complex_qca_split}(4), the stable $k$-invariants also vanish but required separate arguments for even and odd primes.

As a consequence, by computing along the Atiyah–Hirzebruch spectral sequence and noting that negative homotopy groups do not factor in, we have that
\begin{cor}\label{cor::compute}
For all spaces $T$ in Theorem~\ref{thm::complex_qca_split} that is not $\mathbf{Q}(\Zbb^2; \mathbb{C})$, the homotopy classes of maps $BG \to T$ is equivalent to $\prod_{i \geq 0} H^i(G; \pi_i(T))$ as groups.
\end{cor}

In the rest of this section, we freely adopt the language of spectra. For an abelian group $A$, $HA$ denotes its Eilenberg-MacLane spectrum. For spectra $E$ and $X$, we write the $k$-th $E$-cohomology of $X$ to be $E^k(X) = [X, \Sigma^k E]$. If $E = HA$, we also may write $H^k(X; A)$ interchangeably to mean $(HA)^k(X)$. 

\subsection{Complex Case}

The purpose of this section is to prove Theorem~\ref{thm::complex_qca_split}(1) and~\ref{thm::complex_qca_split}(2). We will first prove this for the case of $\Zbb^1$.

\begin{theorem}\label{thm::qca_line_complex}
    The space $K(\mathbf{C}(\Zbb^1; \Cbb))$, and hence $\mathbf{Q}(\Zbb^1; \Cbb)$ is equivalent to a product of Eilenberg-MacLane spaces.
\end{theorem}

\begin{proof}[Proof of Theorem~\ref{thm::qca_line_complex}]
Since $\mathbf{Q}(\Zbb^1; \Cbb)$ is the (based) loop space of $K(\mathbf{C}(\Zbb^1; \Cbb))$ and $\Omega$ commutes with products, it suffices to show that $K(\mathbf{C}(\Zbb^1; \Cbb))$ is equivalent to a product of Eilenberg-MacLane spaces. Since $K(\mathbf{C}(\Zbb^1; \Cbb))$ is Segal's K-theory of a symmetric monoidal category, it is in particular the zeroth space of a connective spectra $F$ (see Appendix of \cite{Thomason01011982}).

Theorem D of \cite{qca_grp_space_spectrum} shows us that $K(\mathbf{C}(\Zbb^1; \Cbb)) \cong K(\mathbf{C}(\Zbb^1; \Cbb))_1$. We are reduced to showing $K(\mathbf{C}(\Zbb^1; \Cbb))_1$ is a product of Eilenberg-MacLane spaces, which we will do by showing $F$ has vanishing stable $k$-invariants and is hence a product of Eilenberg-MacLane spaces.

Indeed, the homotopy groups of $F$ are the homotopy groups of $K(\mathbf{C}(\Zbb^1; \Cbb))_1$, which are given in Table~\ref{table_q_z} of this paper (and also Section 6 of \cite{qca_grp_space_spectrum}):
\[
\begin{array}{c|ccccc}
i & 0 & 1 & 2 & 3 & \geq 3  \\
\hline
\pi_i(F) & 0 & \Qbb_{> 0} & 0 & \frac{\Qbb}{\Zbb} \oplus (K_2(\mathbb{C}) \otimes \Qbb) & \text{rational vector space} 
\end{array}
\]
The first possibly non-trivial $k$-invariant occurs between $\pi_1$ and $\pi_{3}$ as a map $\tau_{\leq 3-1} F \cong \Sigma^1 H\Qbb_{>0} \to \Sigma^4 H(\frac{\Qbb}{\Zbb} \oplus K_2(\mathbb{C} \otimes \Qbb))$ which represents an element in
\[[\Sigma^1 H\Qbb_{>0}, \Sigma^4 H(\frac{\Qbb}{\Zbb} \oplus K_2(\mathbb{C} \otimes \Qbb))] \cong [H\Qbb_{>0}, \Sigma^3 H(\frac{\Qbb}{\Zbb} \oplus K_2(\mathbb{C} \otimes \Qbb))],\]
which is $H^3(H\Qbb_{>0}; \frac{\Qbb}{\Zbb} \oplus K_2(\mathbb{C} \otimes \Qbb)) \cong H^3(H\Qbb_{>0}; \frac{\Qbb}{\Zbb}) \oplus H^3(H\Qbb_{>0}; K_2(\mathbb{C}) \otimes \Qbb)$. Observe that $\Qbb_{>0}$ is isomorphic to a countably infinite direct sum of integers, indexed over the prime numbers, so we can without loss write $\Qbb_{>0} \cong \bigoplus_{\Nbb} \Zbb$. We can then rewrite
\[H^3(H\Qbb_{>0}; \frac{\Qbb}{\Zbb}) \cong H^3(H\bigoplus_{\Nbb} \Zbb; \frac{\Qbb}{\Zbb}) \cong \prod_{\Nbb} H^3(H\Zbb; \frac{\Qbb}{\Zbb}),\]
\[H^3(H\Qbb_{>0}; K_2(\mathbb{C}) \otimes \Qbb) \cong H^3(H\bigoplus_{\Nbb} \Zbb; K_2(\mathbb{C}) \otimes \Qbb) \cong \prod_{\Nbb} H^3(H\Zbb; K_2(\mathbb{C}) \otimes \Qbb).\]
Now we claim that both $H^3(H\Zbb; \frac{\Qbb}{\Zbb})$ and $H^3(H\Zbb; K_2(\mathbb{C}) \otimes \Qbb)$ are $0$, which would show this $k$-invariant is trivial. We defer the proof of the first case to the technical Lemma~\ref{lem::technical_steenrod}. For the second case, since $K_2(\mathbb{C}) \otimes \Qbb$ is a rational vector space, it suffices to check that $H^3(H\Zbb; \Qbb) = 0$. In fact, we will show that $H^k(H\Zbb; \Qbb) = 0$ for all $k > 0$. The cohomology $H^k(H\Zbb; \Qbb)$ can be computed as the colimit of unstable cohomology groups $\operatorname{colim}_n H^{n+k}(K(\Zbb, n); \Qbb)$. For $n > 1$, $K(\Zbb, n)$ is simply connected and hence the rational equivalence $K(\Zbb, n) \cong_{\mathbb Q} K(\Qbb, n)$ shows that $H^{n+k}(K(\Zbb, n); \Qbb) \cong H^{n+k}(K(\Qbb, n); \Qbb)$. For $n$ odd, $S^n$ is rationally equivalent to $K(\Qbb, n)$, so $H^{n+k}(K(\Qbb, n); \Qbb) = 0$. For $n$ even, we have that $\widetilde{H}_i(K(\Qbb, n); \Zbb) = \Qbb$ for every $i$ a positive multiple of $n$ and is $0$ otherwise. The universal coefficient theorem then shows that $H^{n+k}(K(\Qbb, n); \Qbb) = 0$ for $n$ sufficiently large. It follows that $H^k(H\Zbb; \Qbb) = 0$ for $k > 0$. This also shows that $H^k(H\Qbb; \Qbb) = 0$ for $k > 0$.

We have now shown that $\tau_{\leq 3} F$ is a wedge of suspensions of Eilenberg-MacLane spectra. To show $F$ is a wedge of suspensions of Eilenberg-MacLane spectra, we need to check the higher $k$-invariants. However, we see that the higher $k$-invariants all take 1 of the following 3 forms: (a) $\Qbb_{>0} \to$ \text{rational}, (b) \text{rational $\to$ rational}, and (c) $\frac{\Qbb}{\Zbb} \to $ rational. Case (a) and (b) will have trivial $k$-invariants since $H^k(H\Zbb; \Qbb) = H^k(H\Qbb; \Qbb) = 0$ for $k > 0$. For Case (c), the fiber sequence $H\Zbb \to H\Qbb \to H\frac{\Qbb}{\Zbb}$ gives a long exact sequence in rational cohomology. Since $H^k(H\Qbb; \Qbb) = 0$ for $k > 0$, it follows that
\[H^k(H\frac{\Qbb}{\Zbb}; \Qbb) = H^{k+1}(H\Zbb; \Qbb) = 0.\]
Thus, Case (c) also gives trivial $k$-invariants. We conclude that $F$ is a wedge of suspensions of Eilenberg-MacLane spectra, so its zeroth space $K(\mathbf{C}(\Zbb^1; \mathbb{C}))_1$ is a product of Eilenberg-MacLane spaces.
\end{proof}

Following the same arguments along that of Theorem~\ref{thm::qca_line_complex}, we similarly have that:
\begin{cor}\label{cor::qca_plane_complex}
Both $\mathbf{Q}(\Zbb^2; \Cbb)$ and the universal cover of $K(\mathbf{C}(\Zbb^2; \Cbb))_1$ are a product of Eilenberg-MacLane spaces. Furthermore, $K(\mathbf{C}(\Zbb^0; \Cbb))$ and $\mathbf{Q}(\Zbb^0; \Cbb)$ are a product of Eilenberg-MacLane spaces.
\end{cor}

\begin{proof}
Let $\widetilde{K(\mathbf{C}(\Zbb^2; \Cbb)))_1}$ denote the universal cover. The main theorem of \cite{qca_grp_space_spectrum} tells us that the homotopy groups of $K(\mathbf{C}(\Zbb^2; \Cbb)))_1$ above degree $1$ can be obtained from the homotopy groups of $K(\mathbf{C}(\Zbb^1; \Cbb)))_1$. Since the universal cover is simply connected, we then have the following table:
\[
\begin{array}{c|cccccc}
i & 0 & 1 & 2 & 3 & 4 & \geq 5  \\
\hline
\pi_i \text{ of the universal cover} & 0 & 0 & \Qbb_{> 0} & 0 & \frac{\Qbb}{\Zbb} \oplus (K_2(\mathbb{C}) \otimes \Qbb) & \text{rational vector space} 
\end{array}
\]
We see that the $k$-invariant questions we encounter are in the cohomology groups as those occuring in that of the proof of Theorem~\ref{thm::qca_line_complex}, which we previously checked are zero. Thus, we have that $\widetilde{K(\mathbf{C}(\Zbb^2; \Cbb)))_1}$ is a product of Eilenberg-MacLane spaces. Since $\mathbf{Q}(\Zbb^2; \Cbb)$ is group-like, we know it splits off as $\pi_0 \mathbf{Q}(\Zbb^2; \mathbb{C}) \times (\mathbf{Q}(\Zbb^2; \mathbb{C}))_{1}$. The component at identity $\mathbf{Q}(\Zbb^2; \mathbb{C})$ is the loop space of $\widetilde{K(\mathbf{C}(\Zbb^2; \Cbb)))_1}$. Thus, we have that $\mathbf{Q}(\Zbb^2; \Cbb)$ is a product of Eilenberg-MacLane spaces.

A similar discussion, without needing to go to the universal cover, shows the case for $\Zbb^0$ due to the structure of its homotopy groups from the delooping.
\end{proof}

Now we prove the technical lemma used in the proof of Theorem~\ref{thm::qca_line_complex}.
\begin{lemma}\label{lem::technical_steenrod}
   The cohomology group $(H\frac{\Qbb}{\Zbb})^3(H\Zbb) = (H\Zbb)^4(H\Zbb) $ is $0$.
\end{lemma}

For a ring $R$, $(HR)^*(HR)$ is referred to as the $R$-Steenrod algebras. The proof of Lemma~\ref{lem::technical_steenrod} involves the theory of integral Steenrod algebras, which are studied using integral cohomology operations. Here, we record the following fact.
\begin{lemma}[\cite{Kochman1982}]\label{lem::kochman}
For $n > 0$, $H\Zbb^n(H\Zbb)$ is finite. Furthermore for each prime $p$, the $p$-primary part of $H\Zbb^n(H\Zbb)$ is simple $p$-torsion (i.e., has order dividing $p$).
\end{lemma}

\begin{proof}[Proof of Lemma~\ref{lem::technical_steenrod}]
Consider the fiber sequence $H\Zbb \to H\Qbb \to H\frac{\Qbb}{\Zbb}$. In the proof of Theorem~\ref{thm::qca_line_complex}, we showed that $(H\Qbb)^k(H\Zbb) = 0$ for $k > 0$, independently of the present lemma. The long exact sequence then shows that
\[(H\frac{\Qbb}{\Zbb})^3(H\Zbb) \cong (H\Zbb)^4(H\Zbb)\]
are the integral cohomology operations in degree $4$. We will now show that $(H\Zbb)^4(H\Zbb) = 0$. For convenience, we write $X = H\Zbb$. The proof of why $(H\Zbb)^4(H\Zbb) = 0$ will follow from a series of sub-statements which we prove one by one.\\

\noindent (a) The Steenrod algebra $H^k(H\Fbb_p; \Fbb_p) = (H\Fbb_p)^k(H\Fbb_p)$ have explicit basis that have been computed. Here we will use the basis description of Serre-Cartan \cite{Cartan1954-1955, Serre1953} (see also the \texttt{Sage} package page for Steenrod algebras \cite{SageSteenrodAlgebra}). These basis show that:
\begin{itemize}[leftmargin=*]
    \item For any prime $p$, $H^1(H\Fbb_p; \Fbb_p) = \Fbb_p\{\beta\}$, generated by what is called the Bockstein map $\beta$, and $\beta^2 = 0$.
    \item For any prime $p > 2$, the next $k$ for which $H^k(H\Fbb_p; \Fbb_p) \neq 0$ is $k = 2(p-1)$. 
    \item $H^2(H\Fbb_2; \Fbb_2)$ is $\Fbb_2 \{\operatorname{Sq}^2\}$ and $H^3(H\Fbb_2; \Fbb_2) = \Fbb_2 \{\operatorname{Sq}^3, \operatorname{Sq}^2 \operatorname{Sq}^1\}$, where $\operatorname{Sq}^1$ is identified as the Bockstein map.
\end{itemize}

\noindent(b) $H^1(X; \Fbb_p) = 0$. Indeed, consider the fiber sequence $X \xrightarrow{\times p} X \to X/p = H\Fbb_p$. This induces a long exact sequence
\[... \to H^0(X; \Fbb_p) \xrightarrow{\times p} H^0(X; \Fbb_p) \to H^1(H\Fbb_p; \Fbb_p) \to H^1(X; \Fbb_p) \xrightarrow{\times p} H^1(X; \Fbb_p) \to ...\]
Clearly both $\times p$ maps are zero maps, as the cohomology groups have coefficients in $\Fbb_p$. This gives an exact sequence
\[0 \to H^0(X; \Fbb_p) \xrightarrow{\beta} H^1(H\Fbb_p; \Fbb_p) \to H^1(X; \Fbb_p) \to 0,\]
where $\beta$ is the Bockstein map. Since $H^0(X; \Fbb_p) = [H\Zbb; H\Fbb_p]= \operatorname{Hom}(\Zbb; \Fbb_p) = \Fbb_p$, the map $\beta$ is an isomorphism by (a). Thus, $H^1(X; \Fbb_p) = 0$.\\

\noindent(c) $H^2(X; \Fbb_2) = H^3(X; \Fbb_2) = \Fbb_2$. For $k \in \{2, 3\}$, running the long exact sequence in $\Fbb_2$-cohomology on the fiber sequence $X \xrightarrow{\times 2} X \to X/p$, we again have an exact sequence
\[0 \to H^{k-1}(X; \Fbb_2) \xrightarrow{\beta} H^k(HX/2 = H\Fbb_2; \Fbb_2) \to H^k(X; \Fbb_2) \to 0.\]
We now see that $H^k(X; \Fbb_2)$ is the cokernel of the Bockstein map. From (a), we then know that $H^k(X; \Fbb_2)$ is 1-dimensional for $k \in \{2, 3\}$.\\

Now we will prove $(H\Zbb)^4(H\Zbb) = H^4(X; \Zbb) =  0$. By Lemma~\ref{lem::kochman}, it suffices to show that each $p$-primary part of $H^4(X; \Zbb)$ is $0$ and the $p$-primary part of $H^k(X; \Zbb)$ is exactly $\ker(\times p: H^k(X; \Zbb) \to H^k(X; \Zbb))$ for $k > 0$. Consider the exact sequence $\Zbb \xrightarrow{\times p} \Zbb \to \Fbb_p$, which induces a long exact sequence in cohomology groups of $X$ as:
\[\begin{tikzcd}
	{H^1(X; \Fbb_p)=0} & {H^2(X; \Zbb)} & {H^2(X; \Zbb)} & {H^2(X; \Fbb_p)} & \\
	& {H^3(X; \Zbb)} & {H^3(X; \Zbb)} & {H^3(X; \Fbb_p)} \\
	& {H^4(X; \Zbb)} & {H^4(X; \Zbb)} & {H^4(X; \Fbb_p)} & {...}
	\arrow[from=1-1, to=1-2]
	\arrow["{\times p}", from=1-2, to=1-3]
	\arrow[from=1-3, to=1-4]
	\arrow[from=1-4, to=2-2]
	\arrow["{\times p}"', from=2-2, to=2-3]
	\arrow[from=2-3, to=2-4]
	\arrow[from=2-4, to=3-2]
	\arrow["{\times p}"', from=3-2, to=3-3]
	\arrow[from=3-3, to=3-4]
	\arrow[from=3-4, to=3-5]
\end{tikzcd}\]
where $H^1(X; \Fbb_p) = 0$ by (b). We know that $H^3(X; \Fbb_p) = 0$ is the cokernel of the Bockstein $\beta: H^2(X; \Fbb_p) \to H^3(H\Fbb_p; \Fbb_p)$ and hence (by (a)) is $0$ for $p > 2$, so the map $\times p: H^4(X; \Zbb) \to H^4(X; \Zbb)$ is injective. Thus, $H^4(X; \Zbb)$ has zero $p$-primary parts for $p > 2$. When $p = 2$, (a) implies the diagram becomes
\[\begin{tikzcd}
	0 & {H^2(X; \Zbb)} & {H^2(X; \Zbb)} & {\Fbb_2} & \\
	& {H^3(X; \Zbb)} & {H^3(X; \Zbb)} & {\Fbb_2} \\
	& {H^4(X; \Zbb)} & {H^4(X; \Zbb)} & {H^4(X; \Fbb_2)} & {...}
	\arrow[from=1-1, to=1-2]
	\arrow["{\times 2}", from=1-2, to=1-3]
	\arrow[from=1-3, to=1-4]
	\arrow[from=1-4, to=2-2]
	\arrow["{\times 2}"', from=2-2, to=2-3]
	\arrow[from=2-3, to=2-4]
	\arrow[from=2-4, to=3-2]
	\arrow["{\times 2}"', from=3-2, to=3-3]
	\arrow[from=3-3, to=3-4]
	\arrow[from=3-4, to=3-5]
\end{tikzcd}.\]
Since the map $\times p: H^2(X; \Zbb) \to H^2(X; \Zbb)$ is injective for any prime, Lemma~\ref{lem::kochman} implies $H^2(X; \Zbb) = 0$. The exact sequence now reduces to
\[\begin{tikzcd}
	0 & {\Fbb_2} & {\Fbb_2} & {\Fbb_2} & {\Fbb_2} & \\
	&& {H^4(X; \Zbb)} & {H^4(X; \Zbb)} & {H^4(X; \Fbb_2)} & {...}
	\arrow[from=1-1, to=1-2]
	\arrow["\cong", from=1-2, to=1-3]
	\arrow["{\times 2}", from=1-3, to=1-4]
	\arrow["\phi", from=1-4, to=1-5]
	\arrow["\delta", from=1-5, to=2-3]
	\arrow["{\times 2}"', from=2-3, to=2-4]
	\arrow[from=2-4, to=2-5]
	\arrow[from=2-5, to=2-6]
\end{tikzcd}.\]
Since $\ker(\phi) = \im(\times 2) = 0$, the map $\phi: \Fbb_2 \to \Fbb_2$ is injective and hence an isomorphism. Thus, $\ker(\delta) = \operatorname{im}(\phi) = \Fbb_2$, and $\delta$ is the zero map. This shows that the map $\times 2: H^4(X; \Zbb) \to H^4(X; \Zbb)$ is injective, so $H^4(X; \Zbb)$ has no non-trivial 2-primary parts.

Thus, $H^4(X; \Zbb)$ has no non-trivial $p$-primary parts for any prime $p$. Since $H^4(X; \Zbb)$ is finite by Lemma~\ref{lem::kochman}, this shows that $(H\Zbb)^4(H\Zbb) = H^4(X; \Zbb) = 0$.
\end{proof}

\subsection{Unitary Case}

We now prove that the  $*$-QCA spaces $\mathbf{Q}^*(X)$ (described in Section~\ref{sec::unitary_obstruction}) constructed by \cite{qca_grp_space_spectrum} are equivalent to a product of Eilenberg-MacLane spaces when $X = \Zbb^0, \Zbb^1, \Zbb^2$. This is the unitary analog of the results over $\mathbb{C}$ discussed above.

We first note the unitary QCA spaces are related to the following categories.
\begin{defn}
We consider three categories associated to unitary matrices
\begin{enumerate}[leftmargin=*]
    \item We let $\mathcal{U}_{\oplus}$ denote the groupoid whose objects are $\Cbb^n, n \geq 0$ and automorphisms are the unitary matrices $U(n)$. $\mathcal{U}_{\oplus}$ is symmetric monoidal under direct sum.
    \item We let $\mathcal{U}_{\otimes}$ denote the groupoid whose objects are $\Cbb^n, n > 0$ and automorphisms are the unitary matrices $U(n)$. $\mathcal{U}_{\otimes}$ is symmetric monoidal under tensor product.
    \item We let $\mathcal{PU}_{\otimes}$ denote the groupoid whose objects are $U(n), n > 0$, and automorphisms are the projective unitary matrices $PU(n)$. $\mathcal{PU}_{\otimes}$ is symmetric monoidal under tensor product.
\end{enumerate}
Here we do not associate any topology to the set of morphisms.
\end{defn}

Note that $\pi_0(\mathcal{PU}_{\otimes})$ is the group completion of the multiplicative monoid $\Zbb_{>0}$, which is $\Qbb_{>0}$. The delooping theorem of \cite{qca_grp_space_spectrum} implicitly shows that $K(\mathcal{PU}_{\otimes}) \cong \Omega K(\mathbf{C}^*(\Zbb^1))$. Taking $\pi_0$ on both sides gives an isomorphism $\Qbb_{> 0} \cong \pi_1(K(\mathbf{C}^*(\Zbb^1))) = \mathcal{Q}^*(\Zbb^1)/\mathcal{C}^*(\Zbb^1)$, which recovers the usual GNVW index \cite{Gross_2012}.

\begin{prop}\label{prop::unitary_pt_groups}
The homotopy groups of $K(U_{\otimes})$ fall in the following pattern.
\[
\begin{array}{c|cccc}
i & 0 & 1 & 2 & \geq 3  \\
\hline
K_i(\mathcal{PU}_{\otimes}) & \Qbb_{> 0} & 0 & \frac{\Qbb}{\Zbb} \oplus (\text{rational group}) & \text{rational group} 
\end{array}
\]
Furthermore, $K_2(\mathcal{PU}_{\otimes})$ can be non-canonically split into $U(1) \oplus (\text{rational group})$.
\end{prop}

\begin{remark}
    The splitting $K_2(\mathcal{PU}_{\otimes}) \cong U(1) \oplus (\text{rational group})$ being non-canonical is expected. The unitary QCA considered do not take into account the topology of the unitary matrices. Witout topology, one does not have a canonical ordering on $\Rbb$ in $\Rbb/\Zbb \cong U(1)$, so the splitting is expected to be non-canonical.
\end{remark}

In the proof of Proposition~\ref{prop::unitary_pt_groups}, we adopt the following convention: for a group $G$ with a non-discrete topology, $G^{\delta}$ denotes its underlying discrete group.

\begin{proof}
    The case of $\pi_0$ had been discussed right before this proposition. For $\pi_1$, Bass's definition of $K_1$ (see Proposition 1 of \cite{c80f3555-900f-35b5-ab05-6d397e6a44e6}) tells us that $K_1(\mathcal{PU}_{\otimes}) = \operatorname{colim}_{n \in (\Zbb_{> 0}, |)} (PU(n)^{\delta})^{ab}$. The main theorem of \cite{10.2969/jmsj/00130270} shows that each $PU(n)^{\delta}$ is perfect, so the colimit becomes $0$. 
    
    For the higher homotopy groups, we will construct a fibration sequence to compute them. Let $U_{\otimes} = \operatorname{colim}_{(\Zbb_{>0}, |)} U(n)$ and $PU = \operatorname{colim}_{(\Zbb_{>0}, |)} PU(n)$, where the colimit is taken over the divisibility poset and an arrow $d \xrightarrow{\times k} n$ is sent to the Kronecker product $- \otimes I_k$.
    
    The short exact sequences $U(1) \to U(n) \to PU(n)$ assembles to give a short exact sequence $U(1) \to U_{\otimes} \to PU$, which induces a fibration
    \begin{equation}\label{eq::unitary_fibration}
    BU(1)^{\delta} \to BU_{\otimes}^{\delta} \to BPU^{\delta}.
    \end{equation}
    For a fibration $F \to E \to B$ in general, the main theorem of \cite{BERRICK1983172} characterizes when $F^+ \to E^+ \to B^+$ is a fibration: (a) $\pi_1(E) \to \pi_1(B)$ sends the maximal perfect normal subgroup $\mathcal{P}(\pi_1(E)) \subseteq \pi_1(E)$ surjectively onto the maximal perfect normal subgroup $\mathcal{P}(\pi_1(B)) \subseteq \pi_1(B)$, (b) $\mathcal{P}(\pi_1(B))$ acts trivially on $\pi_*(F^+)$.\\
    
    The sequence in (\ref{eq::unitary_fibration}) satisfies condition (b) as $U(1)$ is central in $U_{\otimes}$ and the plus construction of $BU(1)^{\delta}$ is itself. Condition (a) is also satisfied as the maximal perfect normal subgroup of $U_{\otimes}$ is $SU_{\otimes}$, which is sent surjectively onto $PU^{\delta}$, which is itself perfect because each $PU(n)^{\delta}$ is perfect. Note this is similar to Proposition 6.2 of \cite{qca_grp_space_spectrum}. Thus we have that
    \begin{equation}\label{eq::unitary_fibration_plus}
    BU(1)^{\delta} \to (BU_{\otimes}^{\delta})^+ \to (BPU^{\delta})^+
    \end{equation}
    is a fibration. Furthermore, Proposition 3 of \cite{c80f3555-900f-35b5-ab05-6d397e6a44e6} implies that we can identify $(BU^{\delta}_{\otimes})^+ = K(\mathcal{U_{\otimes}})_1$ and $(BPU^{\delta})^+ = K(\mathcal{PU}_{\otimes})_1$. Thus, we now have
    \begin{equation}\label{eq::unitary_k_theory_fib}
    BU(1)^{\delta} \to K(\mathcal{U}_{\otimes})_1 \to K(\mathcal{PU}_{\otimes})_1
    \end{equation}
    As $\pi_*(BU(1)^{\delta})$ is concentrated in degree $1$, the long exact sequence in homotopy groups show that $\pi_{i}(K(\mathcal{U}_{\otimes})_1) \cong \pi_i(K(\mathcal{PU}_{\otimes})_1)$ for $i \geq 3$. A similar argument as to that of Lemma~\ref{lem::BGL_rational} would show that the homotopy groups of $K(\mathcal{U_{\otimes}})_1$ are all rational. In fact, one can use Theorem V.5.3 of \cite{May1977EInfinityRingSpaces} directly here to show that $K_i(\mathcal{U}_{\otimes}) = K_i(\mathcal{U}_{\oplus}) \otimes \Qbb$ for $i > 0$. Regardless, we have that $\pi_i(K(\mathcal{U}_{\otimes})_1)$ is rational for $i \geq 3$.\\

    Finally for $i = 2$, we investigate the long exact sequence:
    \[0 \to \pi_2(K(\mathcal{U}_{\otimes})_1) \to \pi_2(K(\mathcal{PU}_{\otimes})_1) \to U(1) \xrightarrow{\phi} \pi_1(K(\mathcal{U}_{\otimes})_1) \to 0.\]
    Since $\pi_1(K(\mathcal{U}_{\otimes})_1)$ is rational, the universal property of rationalization shows $\phi$ factors through the rationalization $U(1) \cong \Rbb/\Zbb  \to U(1) \otimes \Qbb \cong (\Rbb/\Zbb)/(\Qbb/\Zbb)$, which has kernel being the roots of unity $\frac{\Qbb}{\Zbb}$. Thus, $\frac{\Qbb}{\Zbb} \subseteq \ker(\phi)$. Since $\frac{\Qbb}{\Zbb}$ is injective, we can decompose $\ker(\phi) = \frac{\Qbb}{\Zbb} \oplus \text{rational group}$. We also know the exact sequence splits off as
    \[0 \to \pi_2(K(\mathcal{U}_{\otimes})_1) \to \pi_2(K(\mathcal{PU}_{\otimes})_1) \to \ker(\phi) \to 0.\]
    Since $\pi_2(K(\mathcal{U}_{\otimes})_1)$ is rational and hence injective, we have that $\pi_2(K(\mathcal{PU}_{\otimes})_1) \cong \frac{\Qbb}{\Zbb} \oplus \text{rational group}$.\\

    Using the fact that $K_i(\mathcal{U}_{\otimes}) = K_i(\mathcal{U}_{\oplus}) \otimes \Qbb$ for $i > 0$, we can obtain the stronger second statement in the proposition. The split inclusion $\frac{\Qbb}{\Zbb} \subset \frac{\Rbb}{\Zbb}$ gives $\frac{\Rbb}{\Zbb} = \frac{\Qbb}{\Zbb} \oplus V$ where $V$ is a rational vector space. It suffices to show we can extract $V$ out of $\pi_2(K(\mathcal{U}_{\otimes})_1)$. Now we know that $\pi_2(K(\mathcal{U}_{\otimes})_1) = \pi_2(K(\mathcal{U}_{\oplus})_1) \otimes \Qbb$. Furthermore, Proposition 3 and Theorem 4 of \cite{c80f3555-900f-35b5-ab05-6d397e6a44e6} tells us that 
    \[\pi_2(K(\mathcal{U}_{\oplus})_1) = H_2(BSU^{\delta}; \Zbb),\]
    where $\operatorname{SU} = \operatorname{colim}_{(\Zbb, \leq)} \operatorname{SU}(n)$ is the infinite special unitary group. \cite{10.1007/BFb0080009} shows that $H_2(BSU^{\delta}; \Zbb) \cong H_2(SU(n); \Zbb)$ for a finite $n$ sufficiently large (see also Section 4 of \cite{Milnor1983}). \cite{Sah01011977} shows that for any non-trivial and connected Lie group $G$, $H_2(BG^{\delta}; \Zbb)$ surjects onto an uncountable rational vector space. Thus, $H_2(BSU^{\delta}; \Zbb) = H_2(BSU(n)^{\delta})$ surjects onto an uncountable rational vector space. This shows that $\pi_2(K(\mathcal{U}_{\otimes})_1)$ contains an uncountable rational vector space, so we can split  $V$ non-canonically off. This shows that $K_2(\mathcal{PU}_{\otimes}) \cong U(1) \oplus (\text{rational group})$.
\end{proof}

\begin{remark}
Let $R$ be a ring and $K^{\otimes}(R)$ be the group completion of $\bigsqcup_{n > 0} \operatorname{BGL}_n(R)$ under tensor product. Note that the usual K-theory space $K(R)$ is the group completion of $\bigsqcup_{n \geq 0} \operatorname{BGL}_n(R)$ under direct sum. The statement that $K_i(\mathcal{U}_{\otimes}) = K_i(\mathcal{U}_{\oplus}) \otimes \Qbb$ for $i > 0$ has an algebraic analog that $K_i^{\otimes}(R) = K_i(R) \otimes \Qbb$ for $i > 0$. This is Proposition 5 of \cite{c80f3555-900f-35b5-ab05-6d397e6a44e6} and follows from Theorem VII.5.3 of \cite{may2009preciselyeinftyringspaces}. Recently, the expository note \cite{dubey2026tensorproductktheoryrational} also gave an argument to explain how to prove the algebraic case, which can be adapted to the unitary case.
\end{remark}

Using Proposition~\ref{prop::unitary_pt_groups}, we can compute the structure of the unitary QCA spaces similarly as in the complex case, which is the statement for Theorem~\ref{thm::complex_qca_split}(3).

\begin{theorem}\label{thm::unitary_qca_split}
For $n = 0, 1, 2$, the space $K(\mathbf{C}^*(\Zbb^n))$ and hence $\mathbf{Q}^*(\Zbb^n)$ are product of Eilenberg-MacLane spaces.
\end{theorem}

\begin{remark}[Comparison with invertible phases]
Unitary QCA are closely related to the classification of invertible quantum phases of matter. In both cases, one obtains an $\Omega$-spectrum indexed by the lattice dimension. Recently, Kubota constructed an $\Omega$-spectrum
\begin{equation}
\mathbb{IP}
=
{\operatorname{IP}(\Zbb^d)}_{d\geq 0}
\end{equation}
of invertible states \cite{kubota2025stable}. In particular, Propositions 5.2 and 5.10 of \cite{kubota2025stable} show $\operatorname{IP}(\Zbb^0), \operatorname{IP}(\Zbb^1),$ and $\operatorname{IP}(\Zbb^2)$ are products of Eilenberg--MacLane spaces. This is parallel to the splitting result for unitary QCA in Theorem~\ref{thm::unitary_qca_split}.

This parallel is appealing in the present context. The relation between QCA and invertible phases is discussed in the introduction of \cite{qca_grp_space_spectrum}; morally, QCA could be regarded as dynamical analogues of invertible phases. We do not pursue this comparison further here, but the agreement provides additional evidence that the spectrum of unitary QCA is closely related to the spectrum of invertible phases.
\end{remark}

\begin{proof}[Proof of Theorem~\ref{thm::unitary_qca_split}]
  Due to the description of homotopy groups obtained in Proposition~\ref{prop::unitary_pt_groups}, the unitary delooping result in Theorem 5.10 of \cite{qca_grp_space_spectrum}, the proof of Theorem~\ref{thm::qca_line_complex} and Corollary~\ref{cor::qca_plane_complex} adapts similarly to show that $K(\mathbf{C}^*(\Zbb^0)), K(\mathbf{C}^*(\Zbb^1))$, and the universal cover $\widetilde{K(\mathbf{C}^*(\Zbb^2))_1}$ are all products of Eilenberg-MacLane spaces. Now $\pi_1(K(\mathbf{C}^*(\Zbb^2)))$ is the unitary QCA classification group of a plane, which had been shown to be $0$ (see Section 3 of \cite{freedman2020classification} or Appendix of \cite{haah2021clifford}), so $K(\mathbf{C}^*(\Zbb^2))$ is its own universal cover. This concludes the proof. 
\end{proof}

\subsection{Finite Field Case}\label{sec::finite}

Here we will prove Theorem~\ref{thm::complex_qca_split}(4). We take $\Fbb = \Fbb_{p^k}$ to denote a finite field of order $p^k$. Recall that Theorem D of \cite{qca_grp_space_spectrum} shows $K(\mathbf{C}(\Zbb^1; \Fbb_{p^k})$ is connected. Since Quillen \cite{quillen_plus} showed that $K_i(\Fbb_{p^k})$ is finite for $i > 0$, the first column of Table~\ref{table_q_z} gives:
\begin{prop}\label{prop::line_fin}
The homotopy groups of $K(\mathbf{C}(\Zbb^1; \Fbb_{p^k}))$ are:
\[
\begin{array}{c|ccccc}
i & 0 & 1 & 2 & 3 & \geq 4  \\
\hline
\pi_i(F) & 0 & \Qbb_{> 0} & 0 & \Zbb/(p^k-1)\Zbb & 0
\end{array}
\]
\end{prop}
The stable $k$-invariant associated to $\pi_1$ and $\pi_3$ specifies an element 
\[\kappa \in H^3(H\Qbb_{> 0}; \Zbb/(p^k-1)\Zbb) \cong \bigoplus_{\Nbb} H^3(H\Zbb; \Zbb/(p^k-1)\Zbb).\]
Now we can check that $\kappa = 0$ when $p = 2$. This follows from the following more general lemma.
\begin{lem}
The group $H^3(H\Zbb; \Zbb/n) = 0$ if $n$ is odd and $\Fbb_2$ if $n$ is even.
\end{lem}

\begin{proof}
Write $X = H\Zbb$ and consider the long exact sequence in coefficients of $H^k(X; -)$ with respect to $\Zbb \xrightarrow{\times n} \Zbb \to \Zbb/n$. In the proof of Lemma~\ref{lem::technical_steenrod}, we showed that $H^3(X; \Zbb) = \Fbb_2$ and $H^4(X; \Zbb ) = 0$. This turns the sequence to:
\[\begin{tikzcd}
	{...} & {\mathbb{F}_2} & {\mathbb{F}_2} & {H^3(X; \mathbb{Z}/n)} & 0 & {...}
	\arrow[from=1-1, to=1-2]
	\arrow["{\times n}", from=1-2, to=1-3]
	\arrow[from=1-3, to=1-4]
	\arrow[from=1-4, to=1-5]
	\arrow[from=1-5, to=1-6]
\end{tikzcd}\]
If $n$ is even, then $\times n: \Fbb_2 \to \Fbb_2$ is the zero map, so $H^3(X; \Zbb/n) = \Fbb_2$. If $n$ is odd, $\times n: \Fbb_2 \to \Fbb_2$ is an isomorphism, so $H^3(X; \Zbb/n) = 0$.
\end{proof}

As a corollary, we see that $\kappa = 0$ when $p = 2$. Identitical computations also show that the stable $k$-invariants for $K(\mathbf{C}(\Zbb^0; \Fbb_{2^k})), \mathbf{Q}(\Zbb^1; \Fbb_{2^k})$ vanish. Although we do not get as much information at odd primes, we do have the following.
\begin{lemma}
Let $p$ be an odd prime, the stable $k$-invariants of $\mathbf{Q}(\Zbb^1; \Fbb_{p^n})$ vanish.
\end{lemma}

\begin{proof}
By Corollary 4.18 of \cite{qca_grp_space_spectrum}, we can identify $\mathbf{Q}(\Zbb^1; \Fbb_{p^n})$ as $K(\operatorname{Az}(\Fbb_{p^n}))$, the $K$-theory space of Azumaya $\Fbb_{p^n}$-algebras, which can be equivalently identified by $\operatorname{BPGL}(\Fbb_{p^n})^+$. By \cite{c80f3555-900f-35b5-ab05-6d397e6a44e6} or the discussions in Section~\ref{sec::anomaly}, there is a map of infinite loop spaces $i: \operatorname{BGL}_{\otimes}(\Fbb_p^n)^+ \to \operatorname{BPGL}(\Fbb_{p^n})^+$ that is an equivalence on $\pi_0$. On the other hand, since the higher homotopy groups of $\operatorname{BGL}_{\otimes}(\Fbb_p^n)^+$ are rational $K$-theory of finite fields, they are all $0$, so $\operatorname{BGL}_{\otimes}(\Fbb_p^n)^+$ is equivalent to $K(\Qbb_{>0}, 0)$ as infinite loop spaces. Thus, we see the composition
\[K(\Qbb_{>0}, 0) \simeq \operatorname{BGL}_{\otimes}(\Fbb_p^n)^+ \xrightarrow{i} \operatorname{BPGL}(\Fbb_{p^n})^+ \xrightarrow{\text{0-th Postnikov truncation}} K(\Qbb_{>0}, 0)\]
is an equivalence and yields a splitting of infinite loop spaces. Since $\mathbf{Q}(\Zbb^1; \Fbb_{p^n})$ is concentrated at $\pi_0$ and $\pi_2$, this shows that its stable $k$-invariants vanish.
\end{proof}

Now we will prove the remaining item in Theorem~\ref{thm::complex_qca_split}.

\begin{proof}[Proof of Theorem~\ref{thm::complex_qca_split}(4)]
The spaces $K(\mathbf{C}(\Zbb^n; \Fbb_{2^k}))$ for $n = 0,1$ and $\mathbf{Q}(\Zbb^1; \Fbb_{p^k})$ are covered by the discussions above. Finally, $\mathbf{Q}(\Zbb^0; \Fbb_{p^k})$ has only one homotopy group, so its stable $k$-invariants are all $0$.
\end{proof}

\appendix

\section{Low Degree Obstruction Classes}\label{appendix::obstruction}
Here, we give descriptions of several low-degree obstruction classes that already appear in the literature. In degree $1$, this obstruction agrees with Definition~\ref{def::universalObstruction}. In degree $2$, these obstructions take values in the same cohomology groups as the universal obstruction classes of Definition~\ref{def::universalObstruction}. Their construction depends on the choice of a hyperplane, and the dependence on this choice, as well as the relation between different choices, had remained open. By comparing these obstructions with the universal obstruction class, we resolve this ambiguity: the obstruction class vanishes for one such choice if and only if it vanishes for any other.

In degree $3$, the constructions described here take values in a different cohomology group, but are refined by the universal obstruction classes. Furthermore, the converse is true over finite groups.

From the descriptions, it becomes clear that definitions in this manner would quickly get out of hand for higher cohomological degrees. 
\stoptoc
\subsection{Degree-1}
Given a QCA representation $\varphi: G\rightarrow \CQ(\ZZ^d)$, the degree-1 obstruction class is simply the composite map 
\begin{equation}
    G\xrightarrow{\varphi} \CQ(\ZZ^d)\rightarrow \CQ(\ZZ^d)/ \CalC(\ZZ^d)
\end{equation}
viewed as a class in $H^1(BG; \CQ(\ZZ^d)/ \CalC(\ZZ^d))$. This is exactly the degree-1 universal obstruction class. The unitary case is also nearly identical.
\subsection{Degree-2}
Here we describe a degree-2 obstruction class that first lies in
\begin{equation}
H^2\bigl(BG;\CQ(\ZZ^{n-1})/\CalC^{sp}(\ZZ^{n-1})\bigr)
\end{equation}
associated to a QCA representation $\varphi:G\to \CQ(\ZZ^n)$ for $n \geq 1$. Note that for $n > 1$, $\mathcal{C}^{sp}(\Zbb^{n-1}) = \mathcal{C}(\Zbb^{n-1})$. The following construction is inspired by the $H^2$ obstruction class in \cite{kawagoe2025anomaly} constructed for $n = 2$ in the unitary case. Throughout, we assume that the degree-one obstruction vanishes, so that $\varphi$ takes values in $\CalC(\ZZ^n)$.

\begin{construction}\label{constr::hyperplane_class}
Let $H\subset \ZZ^n\otimes_\ZZ \RR$ be an oriented hyperplane, and write $H^+$ and $H^-$ for the lattice points lying in the two closed half-spaces, with $H^+\cup H^-=\ZZ^n$. One can define a truncation operation
\begin{equation}
t_+:\CalC(\ZZ^n)\longrightarrow \CalC(H^+)
\end{equation}
by retaining, in each layer of a circuit, only those local automorphisms supported entirely inside $H^+$. This operation is not a group homomorphism. It is also not unique because each $\al\in \CalC(\ZZ^n)$ has infinitely many different circuit expressions. However, the difference between any two truncations $t_+(\al)$ and $t'_+(\al)$, \begin{equation}
    t_+(\al)^{-1} t'_+(\al),
\end{equation} 
 acts as identity on all matrix algebras sufficiently far from the hyperplane $H$. Thus it is supported in a bounded neighborhood of $H$, and may be regarded as an element of $\CQ(\ZZ^{n-1})$, after identifying such a neighborhood with $\ZZ^{n-1}\times[-R,R]$. For $\alpha,\beta\in \CalC(\ZZ^n)$, the same can be said for
\begin{equation}
t_+(\alpha)t_+(\beta)t_+(\alpha\beta)^{-1}.
\end{equation}

Applying this to $\alpha=\varphi(g)$ and $\beta=\varphi(h)$ gives a function
\begin{equation}
\tau_{H^+}:G\times G\longrightarrow \CQ(\ZZ^{n-1}),
\qquad
\tau_{H^+}(g,h)=t_+(\varphi(g))t_+(\varphi(h))t_+(\varphi(gh))^{-1}.
\end{equation}
One may check, using associativity of multiplication, that $\tau_{H^+}$ defines a $2$-cocycle after passing to the quotient by (special) circuits:
\begin{equation}
\overline{\tau}_{H^+}\in Z^2\bigl(G;\CQ(\ZZ^{n-1})/\CalC^{sp}(\ZZ^{n-1})\bigr).
\end{equation} And for different choices of truncations, the cocycles obtained differ by a coboundary. Indeed, 
\begin{equation*}
\overline{t'_+(\varphi(g))t'_+(\varphi(h))t'_+(\varphi(gh))^{-1}}= \overline{t_+(\varphi(g))t_+(\varphi(h))t_+(\varphi(gh))^{-1}}\cdot  \overline{\mu(g)\mu(h)\mu(gh)^{-1}},
\end{equation*}
with $\mu(g)\coloneqq t_+(\varphi(g))^{-1}t'_+(\varphi(g))$. 
Its cohomology class, denoted by $\alpha_{\varphi}(H^+)$, is the degree-2 obstruction associated to the chosen oriented hyperplane $H$. This construction also adapts similarly to the unitary case.
\end{construction}

Reversing the orientation of $H$ exchanges the two half-spaces and changes the resulting class by a sign. Thus the construction depends on the choice of oriented hyperplane. It is less clear how to compare the classes obtained from two non-parallel hyperplanes. However, we will now show that \textit{they are actually all equivalent} to the universal obstruction class $u_2$ in Definition~\ref{def::universalObstruction}. More precisely,
\begin{prop}\label{prop::hyperplane}
Fix an oriented hyperplane $H$ with postive half space $H^+$. When $G = \mathcal{C}(\Zbb^n)$ and the representation is the inclusion map $\iota: \mathcal{C}(\Zbb^n) \hookrightarrow \mathcal{Q}(\Zbb^n)$, the 2-cochain $\bar{\tau}_{H^+}$ is a 2-cocycle that is a model for the universal central extension of $\mathcal{C}(\Zbb^n)$. The analogous statement holds for the unitary case.
\end{prop}

Here by a model, we just mean that the cohomology class $\alpha_{\varphi}(H^+)$ gives a central extension that has the universal property of being a central extension. Evidently, the 2-cochain $\bar{\tau}_{H^+}$ is natural with respect to group homomorphisms. Any QCA representation $G \to \mathcal{Q}(\Zbb^n)$ whose degree-1 obstruction vanish must factor through the circuit group $G \to \mathcal{C}(\Zbb^n) \hookrightarrow \mathcal{Q}(\Zbb^n)$. The universal obstruction class is also natural with respect to group homomorphisms. Thus, 
\begin{cor}
The cohomology class $\alpha_{\varphi}(H^+)$ is a model for the universal obstruction class $u_2(\varphi)$ for any $\varphi: G \to \mathcal{C}(\Zbb^n)$. The analogous statement holds for the unitary case.
\end{cor}
\begin{cor}\label{cor::hyperplane_algebraic}
    Let $\varphi$ be a QCA representation over $\mathbb F$. Vanishing of the cohomology class $\alpha_{\varphi}(H^+)$ is independend of the choice of hyperplane $H.$
\end{cor}
\begin{cor}\label{cor::hyperplane_unitary}
    Let $\varphi$ be a unitary QCA representation. Vanishing of the cohomology class $\alpha_{\varphi}(H^+)$ is independend of the choice of hyperplane $H.$
\end{cor}

\begin{remark}[Computation]
A hyperplane $H\subset \ZZ^n$ gives a concrete way to compute $u_2(\varphi)$ for $\varphi: G \to \mathcal{C}(\Zbb^n)$. After choosing an arbitrary finite thickening of $H$, the group $\CQ(H)/\CalC^{sp}(H)$ may be used as a model for the coefficient group $A_1$. Proposition~\ref{prop::hyperplane} shows that, for a given QCA representation, one may choose any convenient hyperplane and compute the degree-$2$ universal obstruction class in the corresponding concrete model.
\end{remark}

\begin{remark}[Dependence on hyperplane]\label{rmk::computation}
 The resulting classes need not be canonically comparable for two hyperplanes $H$ and $H'$. Indeed, there is usually no canonical isomorphism between
$\CQ(H)/\CalC^{sp}(H)$ and $\CQ(H')/\CalC^{sp}(H')$, even though these groups are abstractly isomorphic. The exception is when $H$ and $H'$ are parallel; after passing to suitable finite thickenings, their QCA groups can be canonically identified. Then it is meaningful to compare the resulting obstruction classes. For example, merely reversing the orientation on $H$ results in the opposite class. Nevertheless, the vanishing of the obstruction is independent of the choice of hyperplane. 
\end{remark}

\begin{remark}
Toward the end of the preparation of this manuscript, we were informed by Kapustin and Xu that they have independently obtained a proof of Corollary~\ref{cor::hyperplane_unitary}. Their work is currently in preparation~\cite{kapustinXuInPreparation}.
\end{remark}

The proof of Proposition~\ref{prop::hyperplane} follows from a more general phenomenon. To state this, we generalize the hyperplane construction to the following context.

\begin{construction}
    Let $X$ a metric space as in Section~\ref{subsec::qca} and $\varphi: G \to \mathcal{Q}(X \times \Zbb)$ be a QCA representation such that $\varphi$ lands in $\mathcal{C}(X \times \Zbb)$. We can define a similar truncation operation
    \[t_+: \mathcal{C}(X \times \Zbb) \to \mathcal{C}(X \times \Zbb_{\geq 0})\]
    and a similar 2-cochain
    \[\tau_{\varphi}: G \times G \to \mathcal{Q}(X) \text{ with associated } \bar{\tau}_{\varphi}: G \times G \to \mathcal{Q}(X)/\mathcal{C}^{sp}(X). \]
\end{construction}

\begin{proposition}\label{prop::u2_id1}
Let $G = \mathcal{C}(\Zbb^n)$ and consider the inclusion map $\iota: \mathcal{C}(X \times \Zbb) \hookrightarrow \mathcal{Q}(X \times \Zbb)$. The 2-cochain $\bar{\tau}_{\iota}$ gives a 2-cocycle that is a model for the universal central extension of $\mathcal{C}(X \times \Zbb)$. The analogous statement holds for the unitary case.
\end{proposition}

Using the same naturality argument, one has that:
\begin{cor}
The associated cohomology class to $\bar{\tau}$, denoted $\alpha_{\varphi}(X) \in H^2(\mathcal{C}(X \times \Zbb); \mathcal{Q}(X)/\mathcal{C}^{sp}(X))$ is a model for the universal obstruction class $u_2(\varphi)$. The analogous statement holds for the unitary case.
\end{cor}

\begin{proof}[Proof of Proposition~\ref{prop::hyperplane}]    
This becomes a special case of Proposition~\ref{prop::u2_id1} as there is a coarse equivalence $\Zbb^n \cong H \times \Zbb$, and $H^+$ can be identified as $H \times \Zbb_{\geq 0}$.
\end{proof}

We will now discuss how to prove Proposition~\ref{prop::u2_id1} in the remainder of the subsection. In the proof of the delooping result (Theorem E) and Remark in \cite{qca_grp_space_spectrum}, we showed that, modulo a $\pi_0$-correction, the following diagram is a homotopy pullback:
\[\begin{tikzcd}
	{K(\mathbf{C}(X))} & {K(\mathbf{C}(X \times \mathbb{Z}_{\geq 0}))} \\
	{K(\mathbf{C}(X \times \mathbb{Z}_{\leq 0}))} & {K(\mathbf{C}(X \times \mathbb{Z}))}
	\arrow[from=1-1, to=1-2]
	\arrow[from=1-1, to=2-1]
	\arrow[from=1-2, to=2-2]
	\arrow[from=2-1, to=2-2]
\end{tikzcd}\]
where the maps are induced by the natural identification of a QCA in a subspace to be viewed as a QCA in the larger space. Furthermore, the spaces $K(\mathbf{C}(X \times \Zbb_{\geq 0}))$ and $K(\mathbf{C}(X \times \Zbb_{\leq 0})$ are contractible by an Eilenberg swindle argument (this is Lemma 4.5 of \cite{qca_grp_space_spectrum}). By restricting to the component at identity and taking the universal cover, the following is a fibration sequence
\[K(\mathbf{C}(X))_1 \to K(\mathbf{C}(X \times \Zbb_{\geq 0})) \to \widetilde{K(\mathbf{C}(X \times \Zbb))_1}.\]
We can translate the terms to plus-construction spaces as
\[B\mathcal{Q}(X)^+ \to B\mathcal{C}(X \times \Zbb_{\geq 0})^+ \to \widetilde{B\mathcal{Q}(X \times \Zbb)^+}.\]
Here we note that $\mathcal{Q}(X \times \Zbb_{\geq 0}) = \mathcal{C}(X \times \Zbb_{\geq 0})$. It is a general fact that for a group $G$ with maximal normal subgroup of $N$, the universal cover of $BG^+$ is $BN^+$ (see Section 1 of \cite{WAGONER1972349}). This fibration sequence then translates to
\begin{equation}\label{eq::path_fib}
B\mathcal{Q}(X)^+ \to B\mathcal{C}(X \times \Zbb_{\geq 0})^+ \to B\mathcal{C}(X \times \Zbb)^+
\end{equation}
The maps here are induced by the natural inclusion maps, and evidently the middle space is still contractible. 

There is a general obstruction theory for lifting a map to the base space along a fibration. Here we consider a special case of Section 7.10 of \cite{DavisKirk2001}.

\begin{construction}\label{primary}
Suppose $F \to E \xrightarrow{p} Y$ is a fibration such that $Y$ is simply connected and $F$ is a path-connected simple space. Let $f: W \to Y$ be a map. In this case, we always have a lift $g: W^{(1)} \to E$ is a lift of the 1-skeleton of $W$ along $p: E \to Y$. For any $2$-cell $e$ on $W$, the composition 
\[g|_{\partial e}: S^1 \to E \to Y \text{ is null-homotopic}\]
as it extends a 2-cell. This induces a lift $S^1 \to F$ and hence gives an element of $\pi_1(F)$, as $F$ is simple. This procedure now produces a 2-cochain
\[\gamma^{2}_g(p): C_2^{\operatorname{Cell}}(W) \to \pi_1(F),\]
where $C_2^{\operatorname{Cell}}(W)$ is the 2nd cellular chain complex of $W$.
\end{construction}

\begin{theorem}[\cite{DavisKirk2001}]\label{thm::davis_kirk}
The 2-cochain $\gamma^{2}_g(p)$ gives a well-defined class $[\gamma^{(2)}(p)] \in H^2(W; \pi_1(F))$, called the primary obstruction class, independent of the choice of lift to the $1$-skeleton. If $[\gamma^{(2)}(p)] = 0$, then $g$ can be redefined to give a lift of $W^{(2)} \to E$.
\end{theorem}

This class is also natural. When $W = Y \to Y$ is the identity map. The 2-cochain can be given by collapsing the 1-skeleton of $Y$ and then sending every 2-cell to a class in $\pi_2(Y)$. The Hurewicz map $\pi_2(Y) \to H_2(Y; \Zbb)$ corresponds exactly to the inverse. This means that, under the Hurewicz isomorphism, the class
\[[\gamma^{(2)}(p)] \in H^2(Y; \pi_2(Y)) \cong H^2(Y, H_2(Y; \Zbb))\]
is the identity map $H_2(Y; \Zbb) \to H_2(Y; \Zbb)$.\\

Before proving Proposition~\ref{prop::u2_id1}, we need an additional technical lemma.
\begin{lem}\label{lem::colimit_finite}
The homotopy colimit $$\operatorname{hocolim}(K(\mathbf{C}(X \times \{0\})) \to  K(\mathbf{C}(X \times \{0, 1\})) \to K(\mathbf{C}(X \times \{0, 1, 2\})) \to  ...)$$
induced by the inclusion of spaces, is equivalent to $K(\mathbf{C}(X))$ via the identification of the first term.
\end{lem}

\begin{proof}
It suffices to prove each map $K(\mathbf{C}(X \times \{0\})) \to K(\mathbf{C}(X \times \{0, 1, ..., n\}))$ is an equivalence. Here we use the explicit definition of $\mathbf{C}(X')$ (Definition 3.1 of \cite{qca_grp_space_spectrum}). For any space $X'$ as in Section~\ref{subsec::qca}, the category $\mathbf{C}(X')$ has objects $A(X', q)$ for every quantum spin system $q: X' \to \Zbb_{>0}$ and morphisms as locality-preserving isomorphisms. $\mathbf{C}(X')$ is symmetric monoidal under pointwise tensor product. The natural functor
\[F_n: \mathbf{C}(X \times \{0\}) \to \mathbf{C}(X \times \{0, 1, ..., n\})\]
is then lax symmetric monoidal. By (2.3) of \cite{Thomason01011982}, any lax symmetric monoidal functor $\phi: \mathcal{S}_1 \to \mathcal{S}_2$ between small symmetric monoidal categories that induce an equivalence on $B\mathcal{S}_1 \to B\mathcal{S}_2$ would induce an equivalence $K(\mathcal{S}_1) \to K(\mathcal{S}_2)$. We wish to show this for $F_n$.

We now apply Quillen's Theorem A to $F_n$. For any $D = \mathcal{A}(X \times \{0, 1, ..., n\}, q) \in \mathbf{C}(X \times \{0, 1, ..., n\})$, it suffices to show $B(d \downarrow F_n)$ is contractible. The category $D \downarrow F_n$ has objects $\alpha: D \to F_n(C) \in \mathbf{C}(X \times \{0, 1, ..., n\})$ and morphisms from $\alpha: D \to F_n(C)$ to $\alpha': D \to F_n(C')$ is a map $F_n(f): F_n(C) \to F_n(C')$ such that $F_n(f) \circ \alpha = \alpha'$.

Now write $q'(x, k) = 1$ for $k > 0$ and $q'(x, 0) = \prod_{j = 0}^n q(x, j)$, and $D' = \mathcal{A}(X \times \{0, 1, ..., n\}, q')$. Clearly $D'$ is in the image of $F_n$. Let $\alpha: D \to D'$ a locality preserving isomorphism that stacks the tensor factors on $x \times \{0, 1, ..., n\}$ to $x \times \{0\}$. We claim $\alpha$ is an initial object in $d \downarrow F_n$. Indeed, for any $\alpha': D \to F_n(C')$, since $\{0, 1, ..., n\}$ is bounded, one can recreate the same map $\alpha'$ to in the flattened strip $X \times \{0\}$, which gives the desired morphism. Uniqueness follows since for any $F(f), F(g)$ such that $F(f) \circ \alpha = F(g) \circ \alpha$, $F(f) = F(g)$ as $\alpha$ is an isomorphism, and $f = g$ as the inclusion of the QCA group to a larger set is an injection. Thus, $B(d \downarrow F_n)$ is contractible, and the proof follows. 
\end{proof}

Now we will prove Proposition~\ref{prop::u2_id1}.
\begin{proof}[Proof of Proposition~\ref{prop::u2_id1}]
We check the algebraic case as the unitary case follows similarly. By Lemma~\ref{lem::colimit_finite}, we may without loss replace the fiber in the fibration sequence (\ref{eq::path_fib}) into
\[\operatorname{hocolim}_n B\mathcal{Q}(X \times \{0, ..., n\})^+ \to B\mathcal{C}(X \times \Zbb_{\geq 0})^+ \xrightarrow{f} B\mathcal{C}(X \times \Zbb)^+.\]
We first replace $f$ in the fibration sequence by an strict fibration. This is given by the pullback of $f: B\mathcal{C}(X \times \Zbb_{\geq 0})^+ \to B\mathcal{C}(X \times \Zbb)^+$ and the map from the path space $ev_0: PB\mathcal{C}(X \times \Zbb)^+ \to B\mathcal{C}(X \times \Zbb)^+$ that evaluates a path at $0$. Call the pullback $E'$, this gives a fibration
\[\operatorname{hocolim}_n B\mathcal{Q}(X \times \{0, ..., n\})^+ \to E' \to B\mathcal{C}(X \times \Zbb)^+.\]
Now consider Construction~\ref{primary} on this sequence and the plus-construction map $\psi: B\mathcal{C}(X \times \Zbb) \to B\mathcal{C}(X \times \Zbb)^+$, whose criterion is satisfied. By our discussions above, $[\gamma^2]$ is the identity map $H_2(B\mathcal{C}(X \times \Zbb)^+) \to H_2(B\mathcal{C}(X \times \Zbb)^+)$. The homology isomorphism given by the plus-construction map ensures $\psi^*(\gamma^2)$ is the identity map from $\iota: H_2(B\mathcal{C}(X \times \Zbb)) \to H_2(B\mathcal{C}(X \times \Zbb))$. Thus, we see that $\iota$ is the class representing both universal central extension and primary obstruction.

By Theorem~\ref{thm::davis_kirk}, the primary obstruction class is independent of the choice of lift on the 1-skeleton. Thus, it suffices for us to show that $\bar{\tau}_{\iota}$ is a 2-cocycle that occurs with a particular lift on the 1-skeleton. Indeed, for a map $S^1 \to B\mathcal{C}(X \times \Zbb)^+$, a lift of $S^1$ to $E'$ is precisely a homotopy from $S^1 \to B\mathcal{C}(X \times \Zbb)^+$ to a loop $S^1 \to B\mathcal{C}(X \times \Zbb_{\geq 0})^+ \to B\mathcal{C}(X \times \Zbb)^+$.

The map $\psi: B\mathcal{C}(X \times \Zbb) \to B\mathcal{C}(X \times \Zbb)^+$ is the inclusion of a subcomplex, and a 1-cell in $B\mathcal{C}(X \times \Zbb)$ is given as a loop $\ell_{g}$ representing $g \in B\mathcal{C}(X \times \Zbb)$. For each $g$, we choose a lift by homotopy from the loop representing $g \in B\mathcal{C}(X \times \Zbb)^+$ to a choice of $\ell_{t_+(g)} \to B\mathcal{C}(X \times \Zbb_{\geq 0})^+$. This is the lift on the 1-skeleton. A 2-cell in $B\mathcal{C}(X \times \Zbb)$ is given by a triangle whose edges are $g, h, gh$. Going along the boundary of the triangle gives exactly $t_+(g) t_+(h) t_+(gh)^{-1}$, where we need to flip the last edge. The element $t_+(g) t_+(h) t_+(gh)^{-1}$ is supported on $X \times \{0, ..., m\}$ for some $m$ sufficiently large and hence naturally restricts to the image of the fiber. Thus,
\[\gamma^2(g, h) = t_+(g) t_+(h) t_+(gh)^{-1}.\]
This shows that a choice of $\gamma_2$ could have been $\bar{\tau}_{\iota}$, so their cohomology classes agree. This concludes the proof.
\end{proof}

\subsection{Degree-3}
We recall the construction of~\cite{else2014classifying}, which associates an obstruction class in $H^3(BG;U(1))$ to a QCA representation $\varphi:G\to \CQ^*(\ZZ^1)$ that lands in $\mathcal{C}^*(\Zbb^1)$. We outline a proposal to generalize to the algebraic setting to give an obstruction class in $H^3(BG;\mathbb F^\times)$. We wonder whether the same idea can be extended to systems on $\ZZ^n$ for $n>1$ to produce an obstruction class. 
\begin{construction}
    Choose a point $x\in \ZZ$ and let $r^+$ be the ray $\ZZ \cap [x, \infty)$. We can define the same truncation operation: 
\begin{equation}
    t_+: \CalC(\ZZ) \rightarrow \CalC(r^+).
\end{equation}
Then we again have 
\begin{equation}
    t_+(\alpha)t_+(\beta)t_+(\alpha \beta)^{-1}
\end{equation}
acting as identity on matrix algebras far from $x$. This means we regard it as a QCA on some finite segment $[x, x+l]\cap \ZZ$. Then it must fall into $\PGL_m(\mathbb F)$ for some $m$.

Applying this to $\alpha=\varphi(g)$ and $\beta=\varphi(h)$ gives a function
\begin{equation}
\tau:G\times G\longrightarrow \PGL_\otimes(\mathbb F),
\qquad
\tau(g,h)=t_+(\varphi(g))t_+(\varphi(h))t_+(\varphi(gh))^{-1}.
\end{equation}

Observe that given $g, h, k\in G$, we have equalities 
\begin{align*}
   &\tau(g, h)\tau(gh, k)t_+(\varphi(ghk)) \\&= t_+(\varphi(g))t_+(\varphi(h))t_+(\varphi(k))\\
   &= t_+(\varphi(g)) \tau(h, k) t_+(\varphi(hk))\\
   &= t_+(\varphi(g)) \tau(h, k) t_+(\varphi(g))^{-1} t_+(\varphi(g)) t_+(\varphi(hk))\\
   &= t_+(\varphi(g)) \tau(h, k) t_+(\varphi(g))^{-1} \tau(g, hk) t_+(\varphi(ghk)). 
\end{align*}

This means 
\begin{equation}
    \tau(g, h)\tau(gh, k)=t_+(\varphi(g)) \tau(h, k) t_+(\varphi(g))^{-1} \tau(g, hk)\in \PGL_\otimes(\mathbb F). 
\end{equation}

Choose a lift into $\GL_\otimes(\mathbb F)$ for $\tau(g, h), \tau(gh, k)$, $\tau(h, k)$, and $\tau(g, hk)$, respectively denoted by $\widetilde{\tau(g, h)}, \widetilde{\tau(gh, k)}$, $\widetilde{\tau(h, k)}$, and $\widetilde{\tau(g, hk)}.$ It can be verified that $t_+(\varphi(g))(\widetilde{\tau(h, k)})$ is also a lift for $t_+(\varphi(g)) \tau(h, k) t_+(\varphi(g))^{-1},$ where $t_+(\varphi(g))$ acts on $\widetilde{\tau(h, k)}$ as it is an element in the algebra of observables supported in an interval, so the image is supported in a larger but still bounded interval. The equality above implies there is a function $\omega(g, h, k)\in \mathbb F^\times$

\begin{equation}
    \widetilde{\tau(g, h)}\widetilde{\tau(gh, k)}=\omega(g, h, k)t_+(\varphi(g))(\widetilde{\tau(h, k)})  \widetilde{\tau(g, hk)}. 
\end{equation}
\end{construction}

In the unitary case, the construction proceeds in the exact same manner, with projective unitary groups and unitary groups instead, and $U(1)$ playing the role of $\Fbb^{\times}$. Appendix B of~\cite{else2014classifying} verified $\omega: G^3\rightarrow \mathbb U(1)$ is a cocycle whose cohomology class is independent of various choices made above. We write $[\omega(\varphi)] \in H^3(BG; U(1))$ to denote the corrresponding cohomology class.

Although we did not check this in detail, we expect the proof for the unitary case adapts to show that $\omega: G^3\rightarrow \mathbb F^\times$ gives a well-defined obstruction class in $H^3(BG; \mathbb F^\times)$, provided its degree-2 obstruction given by $\tau(g, h)$ is $0$. In the unitary case, this was not an issue as $\mathcal{Q}^*(\Zbb^0)/\mathcal{C}^*(\Zbb^0) = 0$.\\

We now compare the universal obstruction classes in Definition~\ref{def::universalObstruction} with this class. 
\begin{prop}
Let $\varphi: G \to \mathcal{Q}^*(\Zbb^1)$ be an unitary QCA representation that land in circuits. Suppose the universal obstruction $u_3(\varphi) = 0$, then $[\omega(\varphi)] = 0$. 
\end{prop}

\begin{proof}
Note here $u_3(\varphi)$ exists as the $u_2$-obstruction vanish from $H_2(\mathcal{C}^*(\Zbb^1)) = 0$ (see Column 3 of Table \ref{table_q_z}). This means that $u_3(\varphi) = \varphi^*(\iota_2)$, where $\iota_2$ is the identity map in $\Hom( H_3(\mathcal{C}^*(\Zbb^1),  H_3(\mathcal{C}^*(\Zbb^1))\cong H^3(\mathcal{C}^*(\Zbb^1); H_3(\mathcal{C}^*(\Zbb^1))$. The class $[\omega(\bullet)]$ is also natural, i.e., for any QCA representations $\Phi: G' \to \mathcal{C}^*(\Zbb^1) \subseteq \mathcal{Q}^*(\Zbb^1)$ and map $f: G'' \to G'$, we have
\[[\omega(\Phi \circ f)] = f^*([\omega(\Phi)]).\]
Now consider the inclusion map $i: \mathcal{C}^*(\Zbb^1) \to \mathcal{Q}^*(\Zbb^1)$. The class $$[\omega(i)] \in H^3(\mathcal{C}^*(\Zbb^1); U(1)) \cong \operatorname{Hom}(H_3(\mathcal{C}^*(\Zbb^1)), U(1))$$ defines a map $\psi: A \coloneqq H_3(\mathcal{C}^*(\Zbb^1)) \to B \coloneqq U(1)$. The map $\varphi: G \to \mathcal{C}^*(\Zbb^1)$ gives a natural diagram
\[\begin{tikzcd}
	{H^3(\mathcal{C}^*(\Zbb^1); A)} & {H^3(\mathcal{C}^*(\Zbb^1); B)} \\
	{H^3(G; A)} & {H^3(G; B)}
	\arrow["{\psi_*}", from=1-1, to=1-2]
	\arrow["{\varphi^*}"', from=1-1, to=2-1]
	\arrow["{\varphi^*}", from=1-2, to=2-2]
	\arrow["{\psi_*}", from=2-1, to=2-2]
\end{tikzcd}\]
with $\psi_*(\iota_2) = [\omega(i)]$. Now if $u_3(\varphi) = 0$, then $0 = \psi_*(u_3(\varphi)) = \psi^* \circ \varphi^*(\iota_2) = \varphi^* \circ \psi_*(\iota_2) = \varphi^*([\omega[i]]) = [\omega(\varphi)]$. 
\end{proof}

We establish a partial converse for a finite group $G$. 

\begin{prop}
Let $G$ be a finite group and $\varphi: G \to \mathcal{Q}^*(\Zbb^1)$ be an unitary QCA representation that land in circuits. Suppose $[\omega(\varphi)] = 0$, then $u_3(\varphi) = 0$. 
\end{prop}

\begin{proof}
    The main results of~\cite{bols2025classificationlocalitypreservingsymmetries, seifnashri2026disentangling} proved that $\varphi$ is stably linearizable in this case. Thus, $u_3(\varphi) = 0$ from Theorem~\ref{thm::anomalies}. 
\end{proof}
\resumetoc

\section{1D Arithmetic QCA Representations}\label{appendix::arithmetic}

We describe a family of 1D QCA representations over the rationals. They display obstruction classes not found for QCA representations over $\RR$, $\CC$, or the unitary case. Therefore, they have been missing from the literature.
\stoptoc
\subsection{Setup}
We place
\begin{equation}
    V_i \cong \mathbb Q^2
\end{equation}

to be the local vector space at the site \(i\in \mathbb Z\), with ordered basis
\begin{equation}
    |0\rangle,\ |1\rangle .
\end{equation}
We use the ordered tensor-product basis
\[
|00\rangle,\ |01\rangle,\ |10\rangle,\ |11\rangle
\]
on \(V_i\otimes V_{i+1}\).
The spin system correspondes to $q\equiv 2$ with algebra 
\begin{equation}
    \mathcal A(\ZZ, q)= \bigotimes_{i\in \ZZ} \End(V_i).
\end{equation}
Fix a parameter \(s\in \mathbb Q^\times\).  Eventually we will let
\(s\) a square-free number.  Define the two-site diagonal
gate
\begin{equation}
    D_s
=
\operatorname{diag}(1,s,1,1)
\in GL_4(\mathbb Q),
\end{equation}

and the onsite gate
\begin{equation}
    X
=
\begin{pmatrix}
0&1\\
1&0
\end{pmatrix}
\in GL_2(\mathbb Q).
\end{equation}
Technically, we are considering the elements they represent in the projective linear group. Nevertheless, we use the representatives for shorthand.

Let \(D_{s,i,i+1}\) denote \(D_s\) acting on \(\End(V_i)\otimes \End(V_{i+1})\) and let
\(X_i\) denote \(X\) acting on \(\End(V_i)\) both by conjugation.  We define
\[
A_s
=
\prod_{i\in \mathbb Z} D_{s,i,i+1}=\left(\prod_{i\text{ odd}} D_{s,i,i+1}\right)\left(\prod_{i\text{ even}} D_{s,i,i+1}\right),
\qquad
B
=
\prod_{i\in \mathbb Z} X_i .
\]
Since the gates \(D_{s,i,i+1}\) are all commuting, the product is well-defined and
\(A_s\) may be implemented as depth two circuit by separating even and odd edges.  The circuit \(B\) is an onsite
depth-one circuit.

The basic calculation is as follows.  Define
\[
L_s
=
\operatorname{diag}(1,s)
\in GL_2(\mathbb Q).
\]
Then
\[
D_s (X\otimes X)D_s^{-1}(X\otimes X)
=
\operatorname{diag}(1,s,s^{-1},1)
=
L_s^{-1}\otimes L_s .
\]
Equivalently, on the edge \((i,i+1)\),
\[
D_{s,i,i+1}(X_iX_{i+1})D_{s,i,i+1}^{-1}(X_iX_{i+1})
=
L_{s,i}^{-1}L_{s,i+1},
\]
where \(L_{s,i}\) denotes \(L_s\) acting on \(\End(V_i)\).

Now let \([m,n]\subset \mathbb Z\) be a finite interval and set
\[
A_{s,[m,n]}
=
\prod_{i=m}^{n-1} D_{s,i,i+1},
\qquad
B_{[m,n]}
=
\prod_{i=m}^{n} X_i .
\]
Using the preceding edge identity and telescoping, one obtains
\[
\begin{aligned}
A_{s,[m,n]}B_{[m,n]}A_{s,[m,n]}^{-1}B_{[m,n]}^{-1}
&=
\prod_{i=m}^{n-1} L_{s,i}^{-1}L_{s,i+1}  \\
&=
L_{s,m}^{-1}L_{s,n}.
\end{aligned}
\]
Thus the commutator is purely a boundary term.  In particular, $A_s$ and $B$ commute as the commutator vanishes telescopically. 

We now truncate to the right half-line.  Write
\[
A_{s,+}
=
\prod_{i\geq 0}D_{s,i,i+1},
\qquad
B_+
=
\prod_{i\geq 0}X_i .
\]
The discussion above implies
\[
A_{s,+}B_+A_{s,+}^{-1}B_+^{-1}=
L_{s,0}^{-1}.
\]
\subsection{QCA representations}
  Let
\[
G
=
\langle \alpha,\beta \mid \alpha\beta=\beta\alpha,\ \beta^2=1\rangle
\cong
\mathbb Z\times \mathbb Z/2.
\]  Thus a general
element of \(G\) is written as
\[
g=(a,\epsilon),
\qquad
a\in\mathbb Z,\quad \epsilon\in\{0,1\},
\]
with
\[
(a,\epsilon)(b,\delta)
=
(a+b,\epsilon+\delta \bmod 2).
\]

We represent \(\alpha\) by \(A_s\) and \(\beta\) by \(B\) to get QCA representations 
\begin{equation}
    \varphi_s: G\longrightarrow \CalC(\ZZ) \subset \CQ(\ZZ). 
\end{equation}

\subsection{Degree-2 obstruction}

Equipped with Proposition~\ref{prop::u2_id1}, we compute $u_2(\varphi_s)$ through truncation at $0$. 

Since $A_s$ and $B$ commute, we fix the ordering for their product to be $A_sB$ with truncation 
\begin{equation}
    t_+(B A_s)=t_+(A_s B)= A_{s, +} B_+. 
\end{equation}
Then
\begin{equation}
\tau((1,0),(0,1))
= A_{s,+}B_+(A_{s,+}B_+)^{-1}
= I,
\end{equation}
whereas
\begin{equation}
\tau((0,1),(1,0))
= B_+A_{s,+}(A_{s,+}B_+)^{-1}
= L_{s,0}.
\end{equation}
Observe that $[L_{s,0}]\in \PGL_2(\mathbb Q)$ does not lie in $\PSL_2(\mathbb Q)$ whenever $s\notin (\mathbb Q^\times)^2$, for instance when $s=2$. Indeed, $\det L_{s,0}=s$, while replacing $L_{s,0}$ by another representative of the same projective class multiplies the determinant by a square in $\mathbb Q^\times$. Thus no representative can have determinant $1$ unless $s$ is a square.

Moreover, this obstruction is not removed by stabilization. For any $k>0$, the class
\begin{equation}
[L_{s,0}\otimes I_k]\in \PGL_{2k}(\mathbb Q)
\end{equation}
has determinant $s^k$. If $[L_{s,0}\otimes I_k]$ were represented by an element of $\SL_{2k}(\mathbb Q)$, then $s^k$ would have to lie in $(\mathbb Q^\times)^{2k}$, which again forces $s$ to be a square in $\mathbb Q^\times$. Hence, for $s\notin (\mathbb Q^\times)^2$, the stabilized class remains nontrivial in
\begin{equation}
\PGL_\otimes(\mathbb Q)/\PSL_\otimes(\mathbb Q).
\end{equation}

It follows that the image
\begin{equation}
\bar{\tau}\in Z^2\bigl(G;\PGL_\otimes(\mathbb Q)/\PSL_\otimes(\mathbb Q)\bigr)
\end{equation}
is a nontrivial cocycle whenever $s\notin (\mathbb Q^\times)^2$. Thus this construction yields an infinite family of QCA representations with nontrivial obstruction class $u_2(\varphi_s)$. This also gives an infinite family of QCA representations that are not stably linearizable.

Of course, examples such as these can be defined for QCA representations over any number field. Therefore, we expect there to be a rich connection between QCA representations and number theory. 
\resumetoc

\noindent \textbf{Conflict of interest statement:} Not applicable.\\
\noindent \textbf{Data availability statement:} Not applicable.

\printbibliography

@article{freedman2020classification,
  title={Classification of quantum cellular automata},
  author={Freedman, Michael and Hastings, Matthew B},
  journal={Communications in Mathematical Physics},
  volume={376},
  number={2},
  pages={1171--1222},
  year={2020},
  publisher={Springer}
}

@article{shirley2022three,
  title={Three-dimensional quantum cellular automata from chiral semion surface topological order and beyond},
  author={Shirley, Wilbur and Chen, Yu-An and Dua, Arpit and Ellison, Tyler D and Tantivasadakarn, Nathanan and Williamson, Dominic J},
  journal={PRX quantum},
  volume={3},
  number={3},
  pages={030326},
  year={2022},
  publisher={APS}
}

@article{haah2021clifford,
  title={Clifford quantum cellular automata: Trivial group in 2D and Witt group in 3D},
  author={Haah, Jeongwan},
  journal={Journal of Mathematical Physics},
  volume={62},
  number={9},
  year={2021},
  publisher={AIP Publishing}
}

@article{haah2025topological,
  title={Topological phases of unitary dynamics: Classification in Clifford category},
  author={Haah, Jeongwan},
  journal={Communications in Mathematical Physics},
  volume={406},
  number={4},
  pages={76},
  year={2025},
  publisher={Springer}
}

@article{czajka2025anomalies,
  title={Anomalies on the Lattice, Homotopy of Quantum Cellular Automata, and a Spectrum of Invertible States},
  author={Czajka, Alexander M and Geiko, Roman and Thorngren, Ryan},
  journal={arXiv preprint arXiv:2512.02105},
  year={2025}
}

@article{kawagoe2025anomaly,
  title={Anomaly diagnosis via symmetry restriction in two-dimensional lattice systems},
  author={Kawagoe, Kyle and Shirley, Wilbur},
  journal={arXiv preprint arXiv:2507.07430},
  year={2025}
}

@article{kapustin2025higher,
  title={Higher symmetries, anomalies, and crossed squares in lattice gauge theory},
  author={Kapustin, Anton and Spodyneiko, Lev},
  journal={arXiv preprint arXiv:2507.16966},
  year={2025}
}

@article{kapustinXu2025higher,
  title={Higher symmetries and anomalies in quantum lattice systems},
  author={Kapustin, Anton and Xu, Shixiong},
  journal={arXiv preprint arXiv:2505.04719},
  year={2025}
}

@article{haah2023nontrivial,
  title={Nontrivial quantum cellular automata in higher dimensions},
  author={Haah, Jeongwan and Fidkowski, Lukasz and Hastings, Matthew B},
  journal={Communications in Mathematical Physics},
  volume={398},
  number={1},
  pages={469--540},
  year={2023},
  publisher={Springer}
}

@article{fidkowski2024qca,
   title={A quantum cellular automaton for every symmetry protected topological phase},
   volume={112},
   ISSN={2469-9969},
   url={http://dx.doi.org/10.1103/kw68-mkkd},
   DOI={10.1103/kw68-mkkd},
   number={3},
   journal={Physical Review B},
   publisher={American Physical Society (APS)},
   author={Fidkowski, Lukasz and Haah, Jeongwan and Hastings, Matthew B.},
   year={2025},
   month=July }

@article{sun2025clifford,
  title={Clifford quantum cellular automata from topological quantum field theories and invertible subalgebras},
  author={Sun, Meng and Yang, Bowen and Wang, Zongyuan and Tantivasadakarn, Nathanan and Chen, Yu-An},
  journal={PRX Quantum},
  volume={7},
  number={1},
  pages={010362},
  year={2026},
  publisher={APS}
}

@article{Thomason01011982,
author = {Robert W. Thomason},
title = {First quadrant spectral sequences in algebraic k-theory via homotopy colimits},
journal = {Communications in Algebra},
volume = {10},
number = {15},
pages = {1589--1668},
year = {1982},
publisher = {Taylor \& Francis},
doi = {10.1080/00927878208822794},
URL = { 
        https://doi.org/10.1080/00927878208822794
},
eprint = { 
        https://doi.org/10.1080/00927878208822794
}
}

@article{else2014classifying,
  title={Classifying symmetry-protected topological phases through the anomalous action of the symmetry on the edge},
  author={Else, Dominic V and Nayak, Chetan},
  journal={Physical Review B},
  volume={90},
  number={23},
  pages={235137},
  year={2014},
  publisher={APS}
}

@article{quillen_plus,
 ISSN = {0003486X, 19398980},
 URL = {http://www.jstor.org/stable/1970825},
 author = {Daniel Quillen},
 journal = {Annals of Mathematics},
 number = {3},
 pages = {552--586},
 publisher = {[Annals of Mathematics, Trustees of Princeton University on Behalf of the Annals of Mathematics, Mathematics Department, Princeton University]},
 title = {On the Cohomology and K-Theory of the General Linear Groups Over a Finite Field},
 volume = {96},
 year = {1972}
}

@book{weibel2013k,
  title={The $K$-book: An Introduction to Algebraic $K$-theory},
  author={Weibel, C.A.},
  isbn={9780821891322},
  lccn={2012039660},
  series={Graduate Studies in Mathematics},
  url={https://books.google.com/books?id=Ja8xAAAAQBAJ},
  year={2013},
  publisher={American Mathematical Society}
}

@article{Segal1974,
  author    = {Graeme Segal},
  title     = {Categories and cohomology theories},
  journal   = {Topology},
  volume    = {13},
  year      = {1974},
  pages     = {293--312},
  issn      = {0040-9383},
  doi       = {10.1016/0040-9383(74)90022-6},
  mrnumber  = {0353298},
}

@article{951cf383-5a4d-3775-b8b2-0bafd2f162fe,
 ISSN = {00029947},
 URL = {http://www.jstor.org/stable/1998675},
 author = {Charles A. Weibel},
 journal = {Transactions of the American Mathematical Society},
 number = {2},
 pages = {621--635},
 publisher = {American Mathematical Society},
 title = {$KV$-Theory of Categories},
 volume = {267},
 year = {1981}
}

@article{c80f3555-900f-35b5-ab05-6d397e6a44e6,
 ISSN = {00029939, 10886826},
 URL = {http://www.jstor.org/stable/2043975},
 author = {Charles A. Weibel},
 journal = {Proceedings of the American Mathematical Society},
 number = {1},
 pages = {1--7},
 publisher = {American Mathematical Society},
 title = {$K$-Theory of Azumaya Algebras},
 volume = {81},
 year = {1981}
}

@article{BERRICK1983172,
title = {Characterisation of plus-constructive fibrations},
journal = {Advances in Mathematics},
volume = {48},
number = {2},
pages = {172-176},
year = {1983},
issn = {0001-8708},
doi = {https://doi.org/10.1016/0001-8708(83)90087-7},
author = {A. J. Berrick}
}

@book{May1977EInfinityRingSpaces,
  author    = {May, J. Peter},
  title     = {{E$_\infty$ Ring Spaces and E$_\infty$ Ring Spectra}},
  series    = {Lecture Notes in Mathematics},
  volume    = {577},
  publisher = {Springer-Verlag},
  address   = {Berlin, Heidelberg},
  year      = {1977},
  isbn      = {978-3-540-08136-4},
  doi       = {10.1007/BFb0097608},
  url       = {https://link.springer.com/book/10.1007/BFb0097608},
  note      = {Lecture Notes in Mathematics, volume 577}
}

@article{10.2969/jmsj/00130270,
author = {Morikuni Goto},
title = {{A Theorem on compact semi-simple groups}},
volume = {1},
journal = {Journal of the Mathematical Society of Japan},
number = {3},
publisher = {Mathematical Society of Japan},
pages = {270 -- 272},
year = {1949},
doi = {10.2969/jmsj/00130270},
URL = {https://doi.org/10.2969/jmsj/00130270}
}

@misc{Nachtergaele2004QuantumSpinSystems,
  author       = {Bruno Nachtergaele},
  title        = {Quantum Spin Systems},
  howpublished = {arXiv preprint arXiv:math-ph/0409006},
  year         = {2004},
  note         = {Encyclopedia of Mathematical Physics (Elsevier), arXiv:math-ph/0409006},
  doi          = {10.48550/arXiv.math-ph/0409006},
  url          = {https://arxiv.org/abs/math-ph/0409006}
}

@book{naaijkens2017quantum,
  title={Quantum spin systems on infinite lattices},
  author={Naaijkens, Pieter},
  year={2017},
  publisher={Springer}
}

@article{kubota2025stable,
  title={Stable homotopy theory of invertible gapped quantum spin systems I: Kitaev's Omega-spectrum},
  author={Kubota, Yosuke},
  journal={arXiv preprint arXiv:2503.12618},
  year={2025}
}

@unpublished{kapustinXuInPreparation,
author = {Kapustin, Anton and Xu, Shixiong},
title  = {},
note   = {Work in preparation}
}

@article{feng2026higher,
  title={Higher-form anomalies on lattices},
  author={Feng, Yitao and Kobayashi, Ryohei and Chen, Yu-An and Ryu, Shinsei},
  journal={Physical Review Letters},
  volume={136},
  number={4},
  pages={046504},
  year={2026},
  publisher={APS}
}

@article{freed2023anomaly,
  title={What is an anomaly?},
  author={Freed, Daniel S},
  journal={arXiv preprint arXiv:2307.08147},
  year={2023}
}

@article{farrelly2020review,
  title={A review of quantum cellular automata},
  author={Farrelly, Terry},
  journal={Quantum},
  volume={4},
  pages={368},
  year={2020},
  publisher={Verein zur F{\"o}rderung des Open Access Publizierens in den Quantenwissenschaften}
}

@misc{qca_grp_space_spectrum,
      title={Quantum Cellular Automata: The Group, the Space, and the Spectrum}, 
      author={Mattie Ji and Bowen Yang},
      year={2026},
      eprint={2602.16572},
      archivePrefix={arXiv},
      primaryClass={math.AT},
      url={https://arxiv.org/abs/2602.16572}, 
}

@article{WAGONER1972349,
title = {Delooping classifying spaces in algebraic K-theory},
journal = {Topology},
volume = {11},
number = {4},
pages = {349-370},
year = {1972},
issn = {0040-9383},
doi = {https://doi.org/10.1016/0040-9383(72)90031-6},
url = {https://www.sciencedirect.com/science/article/pii/0040938372900316},
author = {J.B. Wagoner}
}

@Inbook{Bousfield1972,
author="Bousfield, Aldridge K.
and Kan, Daniel M.",
title="R-localizations of nilpotent spaces",
bookTitle="Homotopy Limits, Completions and Localizations",
year="1972",
publisher="Springer Berlin Heidelberg",
address="Berlin, Heidelberg",
pages="126--162",
isbn="978-3-540-38117-4",
doi="10.1007/978-3-540-38117-4_5"
}

@article{Ivanov2022,
author={Ivanov, Sergei O.},
title={An Overview of Rationalization Theories of Non-simply Connected Spaces and Non-nilpotent Groups},
journal={Acta Mathematica Sinica, English Series},
year={2022},
month={Oct},
day={01},
volume={38},
number={10},
pages={1705-1721},
issn={1439-7617},
doi={10.1007/s10114-022-2063-9},
url={https://doi.org/10.1007/s10114-022-2063-9}
}

@article{may2009preciselyeinftyringspaces,
    author={J. P. May},
     title={What precisely are $E_{\infty}$ ring spaces and $E_{\infty}$ ring spectra?}, 
    journal = {Geometry \& Topology Monographs},
    year = {2009},
    volume={16},
    number={09}
}

@article{shirley2026anomaly,
  title={Anomaly-free symmetries with obstructions to gauging and onsiteability},
  author={Shirley, Wilbur and Zhang, Carolyn and Ji, Wenjie and Levin, Michael},
  journal={Physical Review Letters},
  volume={136},
  number={21},
  pages={216602},
  year={2026},
  publisher={APS}
}

@article{seifnashri2026disentangling,
  title={Disentangling anomaly-free symmetries of quantum spin chains},
  author={Seifnashri, Sahand and Shirley, Wilbur},
  journal={Physical Review Letters},
  volume={136},
  number={21},
  pages={216603},
  year={2026},
  publisher={APS}
}

@article{bols2026classification,
author={Bols, Alex
and De Roeck, Wojciech
and De Wilde, Michiel
and Carvalho, Bruno de O.},
title={Classification of Locality Preserving Symmetries on Spin Chains},
journal={Communications in Mathematical Physics},
year={2025},
month={Dec},
day={08},
volume={407},
number={1},
pages={10},
issn={1432-0916},
doi={10.1007/s00220-025-05503-2},
url={https://doi.org/10.1007/s00220-025-05503-2}
}

@book{CohenLadaMay1976,
  author    = {Frederick R. Cohen and Thomas J. Lada and J. Peter May},
  title     = {The Homology of Iterated Loop Spaces},
  series    = {Lecture Notes in Mathematics},
  volume    = {533},
  publisher = {Springer-Verlag},
  address   = {Berlin and New York},
  year      = {1976},
  doi       = {10.1007/BFb0080464}
}

@misc{bols2025classificationlocalitypreservingsymmetries,
      title={Classification of locality preserving symmetries on spin chains}, 
      author={Alex Bols and Wojciech De Roeck and Michiel De Wilde and Bruno de O. Carvalho},
      year={2025},
      eprint={2503.15088},
      archivePrefix={arXiv},
      primaryClass={quant-ph},
      url={https://arxiv.org/abs/2503.15088}, 
}

@book{may2011more,
  title={More Concise Algebraic Topology: Localization, Completion, and Model Categories},
  author={May, J.P. and Ponto, K.},
  isbn={9780226511795},
  series={Chicago Lectures in Mathematics},
  url={https://books.google.com/books?id=QCuvnoqsMEgC},
  year={2011},
  publisher={University of Chicago Press}
}

@book{DavisKirk2001,
  author    = {James F. Davis and Paul Kirk},
  title     = {Lecture Notes in Algebraic Topology},
  series    = {Graduate Studies in Mathematics},
  volume    = {35},
  publisher = {American Mathematical Society},
  address   = {Providence, RI},
  year      = {2001},
  isbn      = {978-1-4704-7368-6},
  url       = {https://bookstore.ams.org/gsm-35/}
}

@inproceedings{Kochman1982,
  author    = {Kochman, Stanley O.},
  title     = {Integral Cohomology Operations},
  booktitle = {Current Trends in Algebraic Topology, Part 1},
  series     = {CMS Conference Proceedings},
  volume      = {2},
  pages       = {437--478},
  year        = {1982},
  publisher   = {American Mathematical Society},
  address     = {Providence, RI},
  note         = {Proceedings of the conference held in London, Ontario, 1981}
}

@misc{SageSteenrodAlgebra,
  author       = {Palmieri, John H.},
  title        = {The Steenrod Algebra},
  howpublished = {\url{https://doc.sagemath.org/html/en/reference/algebras/sage/algebras/steenrod/steenrod_algebra.html}},
  note         = {SageMath Reference Manual, accessed June 2026},
  year         = {2026}
}

@article{Cartan1954-1955,
author = {Cartan, H.},
journal = {Séminaire Henri Cartan},
language = {fre},
number = {1},
pages = {1-8},
publisher = {Secrétariat mathématique},
title = {Détermination des algèbres $H_*(\pi , n; Z_2)$ et $H^*(\pi , n; Z_2)$ ; groupes stables modulo $p$},
url = {http://eudml.org/doc/112298},
volume = {7},
year = {1954-1955},
}

@Article{Serre1953,
author={Serre, Jean-Pierre},
title={Cohomologie modulo 2 des complexes d'Eilenberg-MacLane},
journal={Commentarii Mathematici Helvetici},
year={1953},
month={Dec},
day={01},
volume={27},
number={1},
pages={198-232},
issn={1420-8946},
doi={10.1007/BF02564562},
url={https://doi.org/10.1007/BF02564562}
}

@article{Gross_2012,
   title={Index Theory of One Dimensional Quantum Walks and Cellular Automata},
   volume={310},
   ISSN={1432-0916},
   DOI={10.1007/s00220-012-1423-1},
   number={2},
   journal={Communications in Mathematical Physics},
   publisher={Springer Science and Business Media LLC},
   author={Gross, D. and Nesme, V. and Vogts, H. and Werner, R. F.},
   year={2012},
   month=Jan, pages={419–454} }

@article{Milnor1983,
author={Milnor, J.},
title={On the homology of Lie groups made discrete},
journal={Commentarii Mathematici Helvetici},
year={1983},
month={Dec},
day={01},
volume={58},
number={1},
pages={72-85},
issn={1420-8946},
doi={10.1007/BF02564625},
url={https://doi.org/10.1007/BF02564625}
}

@InProceedings{10.1007/BFb0080009,
author="Alperin, Roger",
editor="Stein, Michael R.",
title="{Stability for $H_2(SU_n)$}",
booktitle="Algebraic K-Theory",
year="1976",
publisher="Springer Berlin Heidelberg",
address="Berlin, Heidelberg",
pages="283--289",
isbn="978-3-540-37964-5"
}

@article{Sah01011977,
author = {Chih-Han Sah and John B. Wagoner},
title = {Second homology of lie groups made discrete},
journal = {Communications in Algebra},
volume = {5},
number = {6},
pages = {611--642},
year = {1977},
publisher = {Taylor \& Francis},
doi = {10.1080/00927877708822184},
URL = {https://doi.org/10.1080/00927877708822184},
eprint = {https://doi.org/10.1080/00927877708822184}
}

@article{Tu_2026,
   title={Anomalies of Global Symmetries on the Lattice},
   volume={16},
   ISSN={2160-3308},
   url={http://dx.doi.org/10.1103/m188-w1ct},
   DOI={10.1103/m188-w1ct},
   number={1},
   journal={Physical Review X},
   publisher={American Physical Society (APS)},
   author={Tu, Yi-Ting and Long, David M. and Else, Dominic V.},
   year={2026},
   month=Feb }

@article{Dror1972,
  author    = {Dror, Emmanuel},
  title     = {Acyclic Spaces},
  journal   = {Topology},
  volume    = {11},
  year      = {1972},
  pages     = {339--348}
}

@article{brown,
 ISSN = {0003486X, 19398980},
 URL = {},
 author = {Edgar H. Brown},
 journal = {Annals of Mathematics},
 number = {3},
 pages = {467--484},
 publisher = {[Annals of Mathematics, Trustees of Princeton University on Behalf of the Annals of Mathematics, Mathematics Department, Princeton University]},
 title = {Cohomology Theories},
 volume = {75},
 year = {1962}
}

@article{Yang_2026,
   title={Categorifying Clifford QCA},
   volume={407},
   ISSN={1432-0916},
   url={http://dx.doi.org/10.1007/s00220-026-05596-3},
   DOI={10.1007/s00220-026-05596-3},
   number={4},
   journal={Communications in Mathematical Physics},
   publisher={Springer Science and Business Media LLC},
   author={Yang, Bowen},
   year={2026},
   month=Mar }

@book{Berrick1982,
  author    = {A. J. Berrick},
  title     = {An Approach to Algebraic K-Theory},
  series     = {Pitman Research Notes in Mathematics},
  volume      = {56},
  publisher   = {Pitman},
  address     = {London},
  year         = {1982}
}

@misc{dubey2026tensorproductktheoryrational,
      title={Tensor Product $K$-theory is Rational Algebraic $K$-theory}, 
      author={Amartya Shekhar Dubey and Mattie Ji},
      year={2026},
      eprint={2606.11412},
      archivePrefix={arXiv},
      primaryClass={math.AT},
      url={https://arxiv.org/abs/2606.11412},
      note={To Appear in Homology, Homotopy and Applications.}
}

@book{Neukirch1999,
author="Neukirch, J{\"u}rgen",
title="Algebraic Number Theory",
year="1999",
publisher="Springer Berlin Heidelberg",
address="Berlin, Heidelberg",
doi="10.1007/978-3-662-03983-0",
url="https://doi.org/10.1007/978-3-662-03983-0"
}
\end{document}